%% file: optimal_coupling_problem.tex
\documentclass[reqno,centertags,12pt]{amsart}

\usepackage[normalem]{ulem}

\usepackage{amsmath,amsthm}
\usepackage{amssymb,amsbsy}
\usepackage{bbm}
\usepackage{tikz}

\usepackage{enumitem}
\usepackage{mleftright}
\usepackage{xcolor}
\usepackage{color}
\usepackage{hyperref}
\usepackage{cancel}
\usepackage{amsfonts}
\usepackage{graphicx}
\usepackage{bookmark}
\usepackage{comment}
\usepackage{longtable}

\long\def\metanote#1#2{{\color{#1}\
		\ifmmode\hbox\fi{\sffamily\mdseries\upshape [#2]}\ }}

-\marginparwidth \oddsidemargin -4mm \evensidemargin \oddsidemargin

\makeatletter
\newcommand\xleftrightarrow[2][]{%
	\ext@arrow 9999{\longleftrightarrowfill@}{#1}{#2}}
\newcommand\longleftrightarrowfill@{%
	\arrowfill@\leftarrow\relbar\rightarrow}
\makeatother

\makeatletter
\newcommand{\xRightarrow}[2][]{\ext@arrow 0359\Rightarrowfill@{#1}{#2}}
\makeatother

\newcounter{smalllist}

\allowdisplaybreaks
\numberwithin{equation}{section}

\newcommand{\al}{\alpha}
\newcommand{\be}{\beta}
\newcommand{\ga}{\gamma}
\newcommand{\Ga}{\Gamma}
\newcommand{\de}{\delta}
\newcommand{\De}{\Delta}
\newcommand{\ep}{\varepsilon}
\newcommand{\ta}{\theta}
\newcommand{\Ta}{\Theta}
\newcommand{\ka}{\kappa}
\newcommand{\la}{\lambda}
\newcommand{\La}{\Lambda}
\newcommand{\si}{\sigma}

\newcommand{\om}{\omega}

\newcommand{\vph}{\varphi}
\newcommand{\vphi}{\varphi}

\newcommand{\sqb}[1]{\left[#1\right]}
\newcommand{\cb}[1]{\left\{#1\right\}}
\newcommand{\abs}[1]{\left|#1\right|}

\newcommand{\inn}[1]{\left\langle #1 \right\rangle}

\newcommand{\supp}{\rm{supp}}

\newcommand{\tr}{\mathrm{tr}}

\newcommand{\mcl}{\mathcal}

\newcommand{\mbe}{\mathbb{E}}

\newcommand{\mcA}{\mathcal{A}}
\newcommand{\mcB}{\mathcal{B}}

\newcommand{\mcD}{\mathcal{D}}

\newcommand{\mcG}{\mathcal{G}}

\newcommand{\mcM}{\mathcal{M}}
\newcommand{\mcP}{\mathcal{P}}
\newcommand{\mcS}{\mathcal{S}}
\newcommand{\mcW}{\mathcal{W}}

\newcommand{\mcc}{\mathcal{C}}
\newcommand{\mcg}{\mathcal{G}}
\newcommand{\mcp}{\mathcal{P}}
\newcommand{\mA}{\mathcal{A}}
\newcommand{\mB}{\mathcal{B}}
\newcommand{\mC}{\mathcal{C}}
\newcommand{\mJ}{\mathcal{J}}

\newcommand{\mbr}{\mathbb R}

\newcommand{\f}{\frac}
\newcommand{\drv}[2]{\frac{d #1}{d #2}}

\newcommand{\Lag}{\Lambda_{\text{glob}}}

\newtheorem{theorem}{Theorem}[section]
\newtheorem{proposition}[theorem]{Proposition}
\newtheorem{lemma}[theorem]{Lemma}
\newtheorem{corollary}[theorem]{Corollary}

\theoremstyle{definition}
\newtheorem{definition}[theorem]{Definition}

\newtheorem{problem}[theorem]{Problem}

\newtheorem{example}[theorem]{Example}

\newtheorem{convention}[theorem]{Convention}
\newtheorem{remark}[theorem]{Remark}

\theoremstyle{definition}

\newcommand{\wg}{\mcl{W}_\mcl{G}}

\begin{document}
\title[]{On Optimal Markovian Couplings of L\'evy Processes}

\author{Wei Yang Kang, Tau Shean Lim}

\address{\noindent Department of Mathematics \\ Xiamen University Malaysia}
\email{taushean.lim@xmu.edu.my \\ wyang1217@gmail.com }

\begin{abstract}
	We study the optimal Markovian coupling problem for two \( \Pi \)-valued Feller processes \( \{X_t\} \) and \( \{Y_t\} \), which seeks a coupling process \( \{(X_t, Y_t)\} \) that minimizes the right derivative at \( t=0 \) of the expected cost \( \mathbb{E}^{(x,y)}[c(X_t, Y_t)] \), for all initial states \((x,y) \in \Pi^2\) and a given cost function \( c \) on \( \Pi \). This problem was first formulated and solved by Chen~\cite{Chen1994} for drift–diffusion processes and later extended by Zhang~\cite{Zhang2000} to Markov processes with bounded jumps. In this work, we resolve the case of L\'evy processes under the quadratic cost \( c(x,y) = \tfrac{1}{2}|x - y|^2 \) by introducing a new formulation of the \emph{L\'evy optimal transport problem} between L\'evy measures. 
	We show that the resulting optimal coupling process \( \{(X_t^*,Y_t^*)\}_{t \ge 0} \) satisfies a minimal growth property: for each $t\ge 0,x,y\in\mbr^d$, the expectation \( \mathbb{E}^{(x,y)}|X_t^* - Y_t^*|^2 \) is minimized among all Feller couplings. A key feature of our approach is the development of a dual problem, expressed as a variational principle over test functions of the generators. We prove \emph{strong duality} for this formulation, thereby closing the optimality gap. As a byproduct, we obtain a Wasserstein-type metric on the space of L\'evy generators and L\'evy measures with finite second moment, and establish several of its fundamental properties.
%
\vspace{6ex}

\end{abstract}

\maketitle
\tableofcontents  

\section{Introduction}

\input{S1}

\subsection*{List of symbols and notations}\,
\begingroup
\footnotesize
\input{List_of_Symbols}
\endgroup

\input{S2}

\input{S3}

\input{S4}

\input{S5}

\input{S6}


\input{S6_Appen1}

\input{S6_Appendix}

\input{S6_Appen_WeakTop}

\bibliography{refs}
\bibliographystyle{abbrv}

\end{document}

%% file: S1.tex
\emph{Coupling} is a fundamental probabilistic technique with applications across a wide range of stochastic models. The core idea is as follows: given two probability measures $(\Omega_1, \mathbb{P}_1)$ and $(\Omega_2, \mathbb{P}_2)$, one seeks a joint probability space $(\boldsymbol{\Omega}, \boldsymbol{P})$, with $\boldsymbol{\Omega} = \Omega_1 \times \Omega_2$, and a probability measure $\boldsymbol{P}$ on $\boldsymbol{\Omega}$ such that the marginals agree with the original measures. That is, for all measurable sets $E_1 \subset \Omega_1$ and $E_2 \subset \Omega_2$,
\[
\boldsymbol{P}(E_1 \times \Omega_2) = \mathbb{P}_1(E_1), \qquad 
\boldsymbol{P}(\Omega_1 \times E_2) = \mathbb{P}_2(E_2).
\]
Such a construction allows one to study the relationship between the two distributions, random variables, or processes by analyzing them on the same probability space. 
The technique is widely used in areas such as 
the analysis of ergodicity, convergence rates, and spectral gaps for Markov chains and MCMC \cite{doeblin1937theorie,Griffeath1975,levin2017markov,propp1996exact}; 
optimal transport theory \cite{MR3050280,MR2459454}; 
Poisson approximation via the Stein–Chen method \cite{barbour1992poisson, chen1975poisson}; 
interacting particle systems \cite{McKean1969,sznitman1991topics,Chaintron_2022_1,Chaintron_2022_2}; 
probabilistic analysis of PDEs \cite{kuwada2010duality,stroock2006multidimensional,wang2013harnack}; 
and geometric analysis through couplings of diffusions \cite{kendall1986nonnegative,vonrenesse2005transport}.

In many applications, a central objective is to construct a coupling that satisfies a specific \emph{optimality} criterion.  
One prominent example is the \emph{maximal coupling} of two probability measures on the sample space, which maximizes the probability that the two coupled random variables are equal.  
More precisely, under the setting above, $\boldsymbol{P}$ is a maximal coupling of $\mathbb{P}_1$ and $\mathbb{P}_2$ if  
\begin{align*}
	\boldsymbol{P}(X = Y) &= \sup_{\boldsymbol{Q}} \, \boldsymbol{Q}(X = Y),
\end{align*}
where the supremum is taken over all couplings $\boldsymbol{Q}$ of $\mathbb{P}_1$ and $\mathbb{P}_2$.  
Equivalently, $\boldsymbol{P}(X \neq Y) = d_{\mathrm{TV}}(\mathbb{P}_1,\mathbb{P}_2)$, where $d_{\mathrm{TV}}$ denotes the total variation distance.  
Such couplings are widely used in the study of mixing times and convergence rates of Markov chains/MCMC, in the construction of perfect sampling algorithms, and in Poisson approximation via the Stein--Chen method.  
It is a standard fact that maximal couplings always exist.

A broader class of optimal couplings is given by the viewpoint of \emph{optimal transport theory}.  
Let $\mu$ and $\nu$ be probability measures defined on Polish spaces $\Pi$ and $\Gamma$, respectively, equipped with a \emph{cost function} $c : \Pi \times \Gamma \to \mathbb{R}$.  
A coupling $\gamma_*$, which is a probability measure on $\Pi \times \Gamma$ with marginals $\mu$ and $\nu$, is called \emph{$c$-optimal} if  
\begin{align*}
	\int_{\Pi \times \Gamma} c(x,y) \, d\gamma_*(x,y)
	= \inf_{\gamma} \int_{\Pi \times \Gamma} c(x,y) \, d\gamma(x,y),
\end{align*}
where the infimum is taken over all couplings $\gamma$ of $\mu$ and $\nu$.  
Such $c$-optimal couplings naturally arise when one seeks to match two distributions while minimizing an underlying transportation cost, and they play a central role in probability theory, statistics, and PDEs.  
Existence of a $c$-optimal coupling is guaranteed under mild conditions, for example when $c$ is lower semicontinuous and bounded from below on $\Pi \times \Gamma$ (see, e.g., \cite{MR2459454}).  
The maximal coupling described above is in fact a special case of this framework, obtained by choosing the cost function $c(x,y) = \mathbf{1}_{\{x \neq y\}}$.

In the context of stochastic processes, one can view a $\Pi$-valued process as a probability measure on the corresponding path space, and a coupling can then be defined analogously as a probability measure on the product of two such path spaces with given marginals. For instance, given two processes $\{X_t\}_{t\ge 0}$ and $\{Y_t\}_{t\ge 0}$, one can define their \emph{maximal coupling}---which always exists in this abstract sense. However, such a coupling is typically of limited practical value, as it does not in general preserve the temporal or causal structure inherent to the stochastic evolution of the individual processes.
For this reason, when working \emph{at the process level}, couplings are often required to satisfy additional structural constraints that reflect the dynamic nature of the problem. Common examples include \emph{faithfulness} \cite{rosenthal1997faithful}, \emph{co-adaptedness} \cite{connor2009optimal,burdzy2000efficient}, and \emph{immersion} \cite{kendall2015coupling}, each imposing progressively stronger compatibility between the joint dynamics and the original processes. 

Let us now turn to the notion of \emph{Markovian coupling}, which is the regime of interest in the present work. Broadly speaking, a coupling $\{Z_t=(X'_t,Y'_t)\}_{t\ge 0}$ of two Markov processes $\{X_t\}_{t\ge 0},\{Y_t\}_{t\ge 0}$ is said to be Markovian if the joint process $\{Z_t\}_{t}$ itself forms a Markov process on the product state space, with respect to its natural filtration. In this setting, the coupling is determined by a transition kernel or generator on the product space whose marginals coincide with those of the individual processes. This ensures that the joint evolution preserves the memoryless property, which is especially useful in the study of interacting particle systems, where couplings can then be analyzed within the framework of Markov semigroups and generators.

An interesting and widely studied example of coupling processes is the \emph{coalescing coupling}, in which two processes evolve jointly until they first meet, after which they move together identically. In the setting of Markov chains on discrete state spaces, the meeting time often serves as a natural stopping time whose distribution directly informs quantitative mixing properties. This naturally motivates the search for \emph{optimal coalescing couplings}—those that bring the processes together as quickly as possible, for example by minimising the distribution or a moment of the coalescence time. In many discrete settings, such optimal constructions coincide with \emph{maximal couplings} in the sense of total variation distance, yielding the sharpest coupling-based bounds on convergence to equilibrium. 

Optimal coalescing couplings have also been studied in the context of \emph{Markovian} couplings on continuous state spaces, often referred to as \emph{Markovian maximal couplings} \cite{bottcher2020markovian}. A canonical example is the \emph{reflection coupling} of Brownian motions \cite{LindvallRogers1986}, where the processes evolve symmetrically with respect to the hyperplane bisecting their positions until they meet, after which they move together. This construction is both Markovian and optimal, achieving the minimal possible coalescence time among all admissible couplings. We remark also coalescing couplings have also been investigated in the L\'evy setting \cite{BoettcherSchillingWang2011,SchillingWang2011,LiangSchillingWang2020}.

Another important class of Markovian couplings is the \emph{synchronous coupling}, in which two Markov processes on $\mathbb{R}^d$—typically described by stochastic differential equations driven by a common source of noise (Brownian or L\'evy)—evolve under exactly the same noise realisation. For example, if the two processes $\{X_t\}_t$ and $\{Y_t\}_t$ on $\mathbb{R}^d$ admit the SDE representations
\[
dX_t = b(X_t)\,dt + \sigma(X_t)\,dW_t, \quad 
dY_t = b(Y_t)\,dt + \sigma(Y_t)\,d\tilde W_t,
\]
where $\{W_t\}$ and $\{\tilde W_t\}$ are $d$-dimensional L\'evy processes with the same law, their synchronous coupling is obtained by setting $W=\tilde W$. We note here that reflection couplings of Brownian motions is a special csae of synchronous couplings. 

In the analysis of interacting particle systems, one is often interested in comparing two such systems, described by Markov processes $\{X_t\}$ and $\{\bar X_t\}$ on some state space $\Pi$, in terms of their error measured by a cost function—for instance, the expected transport cost 
\begin{align}\label{def:coupcost}
	\mathbb{E}[c(X_t,\bar X_t)]	,\qquad t\ge 0
\end{align}
under an appropriate coupling. In the case where the systems are described via SDEs (e.g., diffusion processes), synchronous couplings are particularly useful. By aligning the driving noise across the two systems, synchronous coupling links probabilistic arguments with tools from stochastic calculus and martingale theory, such as It\^o's formula and maximal inequalities. This makes it a natural bridge between coupling methods and analytic techniques, particularly in the study of stability, moment bounds, and propagation of chaos in mean-field models.

For general interacting particle systems where an SDE representation is not available, synchronous couplings are no longer directly applicable. Nevertheless, the underlying idea--constructing a coupling process $\{(X_t, \bar X_t)\}_{t \ge 0}$ to bound the expected transport cost \eqref{def:coupcost}--remains central to the analysis. Within the Markovian coupling framework, as Chen and Li pointed out in \cite{ChenLi1989}, one can exploit the semigroup structure together with integral inequalities (for instance, Gr\"onwall's lemma) to obtain global-in-time estimates from bounds on the derivative of the expected cost at time $t=0$, that is, by analyzing the action of the infinitesimal generator of the coupling on the cost function $c$. 

This motivates the following optimality criterion, introduced by Chen \cite{Chen1994}: given two Markov generators $\mathcal{A}$ and $\mathcal{B}$ of two Markov processes on some state space $\Pi$, and a cost function $c:\Pi^2\to[0,\infty)$, a coupling generator $\mathcal{J}_*$ of $\mathcal{A}$ and $\mathcal{B}$ is said to be \emph{$c$-optimal} if 
\begin{align*}
	(\mathcal{J}_* c)(x,y) &= \inf_{\mathcal{J}} (\mathcal{J} c)(x,y), \qquad \text{for all } (x,y)\in \Pi^2,
\end{align*}
where the infimum is taken overall all coupling generators of $\mA,\mB$. 
A precise formulation of this optimality condition will be presented in the coming section. 

The central question is the following: given two Markov generators $\mathcal{A}$ and $\mathcal{B}$ and a cost function $c$, 
does there exist a $c$-optimal coupling generator $\mathcal{J}_*$ of $\mathcal{A}$ and $\mathcal{B}$? 
We refer to this as the \emph{$c$-optimal Markovian coupling problem}. 
This problem was first introduced by Chen and Li \cite{ChenLi1989}, and was subsequently given a precise formulation by Chen \cite{Chen1994, Chen2020OptimalCouplings} for \emph{bounded Markov jump operators}, 
who also established the existence of minimizers in the case of drift–diffusion generators under suitable assumptions on the cost function $c$.
Later, Zhang \cite{Zhang2000} solved the problem for bounded Markov jump operators. 
For further background, see Section \ref{subsec-history-ocp}. 
The problem has also been studied in the L\'evy setting for specific cost functions 
(e.g., \cite{LiangSchillingWang2020, KendallMajkaMijatovic2024}). 
To the best of our knowledge, however, the $c$-optimal Markovian coupling problem remains open for general unbounded Markov generators.

In the present work, our goal is twofold. First, we rigorously formulate the optimal Markovian coupling problem for general Feller generators. Second, we provide a complete solution to this problem in the case of two L\'evy generators on $\mathbb{R}^d$, with respect to the squared cost $c(x,y) = \tfrac{1}{2}|x-y|^2$. It is well known that L\'evy generators can be decomposed into drift--diffusion and jump components, the latter of which may be unbounded. As mentioned earlier, the general drift--diffusion case was addressed in \cite{ChenLi1989}, while the bounded jump case was solved in \cite{Zhang2000}. The main contribution of the present work is therefore the treatment of the unbounded jump component.

One interesting by-product of this research is the formulation and analysis of the \emph{L\'evy optimal transport problem}. This is a variation of the classical optimal transport problem, which seeks an optimal coupling, but now between two \emph{L\'evy measures} rather than probability measures. This is not the first time such a problem has appeared; it was already considered in the work of Kolokoltsov \cite{Kolokoltsov2010LevyKhintchine} (see the paragraph after Corollary \ref{cor:main4}). Nonetheless, the present work provides the first systematic study of this problem. In particular, the link between the L\'evy optimal transport problem and the optimal Markovian coupling problem for L\'evy jump processes is established here for the first time. This problem will be studied in detail in Section~\ref{sec-lotp}, where we shall see its close connection with the optimal Markovian coupling problem for jump generators.

Another theme of this work is the study of the \emph{dual formulation} associated with the optimal Markovian coupling problem and L\'evy optimal transport problem. We establish that strong duality holds, and this plays a crucial role in the analysis. In particular, the dual formulation provides a verification principle: once a candidate coupling is constructed, its optimality can be confirmed by showing that its cost coincides with the dual value. The question of whether the dual value equals the primal value is referred to as the \emph{optimality gap problem}, and strong duality ensures that this gap vanishes. This follows the standard approach in optimization theory, where strong duality guarantees that equality of primal and dual values certifies optimality.

Another by-product of this research is the introduction of a \emph{Wasserstein-type} metric on the space of L\'evy generators and L\'evy measures. This metric arises naturally from the optimal Markovian problem between L\'evy generators and L\'evy optimal transport problem between L\'evy measures, and provides a quantitative way to compare different L\'evy-type dynamics. This metric was first introduced in \cite{LimTeoh2025}, and to our knowledge, this is the first time such a metric has been systematically studied, and it has several potential applications. We present a detailed construction and analysis of this metric in Section \ref{sec-wass-metric}.

In the coming section, we formulate the two main research problems of this work: the \emph{optimal Markovian coupling problem} and the \emph{optimality gap problem}, both defined with respect to a cost function $c$ in a general abstract framework. Our study focuses on the special case of L\'evy generators on $\mathbb{R}^d$ with the squared cost $c(x,y) = \frac{1}{2}|x-y|^2$. The main results, stated as Theorems \ref{main1}, \ref{main2}, \ref{main-LK}, \ref{main3}, \ref{main5} and Corollary \ref{cor:main4}, will be presented subsequently. The detailed plan and organization of the remainder of the paper will be outlined at the end of Section \ref{sec-setting}.

%% file: List_of_Symbols.tex

\newcommand{\sym}[2]{#1 &:& #2 \\[1ex]} 

\begin{longtable}{lll}
	\sym{$\oplus$}{Direct sum of functions, e.g., $f\oplus g= f(x)+g(y)$, or spaces, e.g., $L^1(\mu)\oplus L^1(\nu)$}
	\sym{$\otimes$}{Tensor product (of functions, measures, operators, semigroups)}
	\sym{$\inn{\cdot,\cdot}$}{Natural pairing over $\mcM(\Pi)\times C_b(\Pi)$}
    \sym{$\xrightarrow{\La}$}{$\La$-weak convergence of L\'evy measures}
        \sym{$(\kappa,\al,\mu),(\zeta,\be,\nu)$}{L\'{e}vy triplet}
        \sym{$(\eta,\si,\ga)$}{Coupled L\'evy triplet}
        \sym{$(\eta_*,\si_*,\ga_*)$}{2-optimal coupled L\'evy triplet}
        \sym{$\al,\be$}{Diffusion matrices}
        \sym{$\mcA,\mcB$}{Feller generators on $\Pi$, or L\'evy generators on $\mbr^d$}
        \sym{$\mcA^\nabla,\mB^\nabla$}{Drift parts of $\mA,\mB$}
	\sym{$\mcA^\De,\mB^\De$}{Diffusion parts of $\mA,\mB$}
	\sym{$\mcA^J,\mB^J$}{Jump parts of $\mA,\mB$}
    \sym{$\mcl{B}(\Pi)$}{Space of Borel measurable functions on $\Pi$}
    \sym{$\mcl{B}_b(\Pi)$}{Space of bounded Borel measurable functions on $\Pi$}
        \sym{$c$}{Cost function over some state space $\Pi$}
        \sym{$c_2$}{Squared cost on $\mbr^{d}$, $c_2(x,y)=\f 12 |x-y|^2$}
        \sym{$\mathcal{C}_c$}{$c$-optimal transport cost}
        \sym{$\mathcal{C}_2$}{$2$-optimal transport cost}
        \sym{$\mathcal{C}_2^\Lambda$}{2-optimal L\'evy transport cost (between L\'evy measures)}
        \sym{$C(\mA),C(\mB)$}{Continuity domains of generators $\mA,\mB$}
        \sym{$C(\Pi)$}{Space of continuous functions on $\Pi$}
	\sym{$C_0(\Pi)$}{Space of continuous functions on $\Pi$ vanishing at infinity }
        \sym{$C_b(\Pi)$}{Space of bounded continuous functions}
        \sym{$C_c(\Pi)$}{Space of continuous functions $f:\Pi\to\mbr$ with compact support}
        \sym{$C_2(\mbr^d)$}{Space of continuous functions $\vphi:\mbr^d\to\mbr$ such that $|\vphi(x)|\le C(1+|x|^2)$}
        \sym{$C_0^2(\mbr^d)$}{Space of bounded twice differentiable functions on $\mbr^d$ vanishing at infinity} 
        \sym{$C_b^2(\mbr^d)$}{Space of bounded twice differentiable functions on $\mbr^d$}
        \sym{$C_2^2(\mbr^d)$}{Space of all $\vphi\in C_2(\mbr^d)$ that is continuously twice differentiable}
        \sym{$\de_x,\de_y$}{Dirac delta measures at $x,y\in\Pi$}
        \sym{$D^-$}{Right hand lower Dini derivative operator}
        \sym{$D(\mcA)$}{Domain of the generator $\mcA$}
        \sym{$\bar{D}(\mcA)$}{Extended domain of the generator $\mcA$}
        \sym{$\mcl{D}(x_0,y_0)$}{Set of admissible pairs for the definition of dual transport derivative }
        \sym{$\{e^{t\mcA}\}_{t\ge 0},\{e^{t\mB}\}_{t\ge 0}$}{Feller semigroups generated by $\mA,\mB$}
        \sym{$\eta$}{$\ka\oplus\zeta$, coupled drift vector}
        \sym{$\ga$}{Coupling of probability/L\'evy measures $\mu,\nu$}
        \sym{$\ga_*$}{L\'evy 2-optimal coupling of L\'evy measures $\mu,\nu$}
        \sym{$\Gamma(\mathcal{A},\mathcal{B})$}{Set of all Feller coupling generators of $\mathcal{A}$ and $\mathcal{B}$}
        \sym{$\Gamma^\La(\mathcal{A},\mathcal{B})$}{Set of all L\'evy coupling generators of $\mathcal{A}$ and $\mathcal{B}$}
        \sym{$\Gamma(\mu,\nu)$}{Set of all (classical) couplings of bounded measures $\mu$ and $\nu$ with equal mass}
        \sym{$\Gamma^\La(\mu,\nu)$}{Set of all (L\'evy) couplings of L\'evy measures $\mu$ and $\nu$}
        \sym{$\mcG(\Pi)$}{Space of probability generators on $\Pi$}
        \sym{$\mcG^\Lambda(\mathbb{R}^d)$}{Space of L\'evy generators on $\mathbb{R}^d$}
        \sym{$\mcG_2^\Lambda(\mathbb{R}^d)$}{Space of L\'evy generators with finite second moments}
        \sym{$\mathcal{J}$}{Coupling generator of $\mathcal{A,B}$}
        \sym{$\mJ^-$}{Lower generator of the generator $\mJ$}
        \sym{$\mathcal{J}_*$}{Optimal coupling generator of $\mathcal{A,B}$}
        \sym{$\kappa,\zeta$}{Drift vectors}
        \sym{$\La$}{L\'evy, particularly any notions related to L\'evy settings}
        \sym{$\Lambda(\mathbb{R}^d)$}{Class of L\'evy measures on $\mathbb{R}^d$}
        \sym{$\Lambda_2(\mathbb{R}^d)$}{Class of L\'evy measures on $\mathbb{R}^d$ with finite second moment}
        \sym{$\Lambda_{\text{glob}}(\kappa,\alpha,\mu)$}{L\'evy generator associated with the triplet $(\kappa,\alpha,\mu)$ expressed in global form}
        \sym{$\mu,\nu$}{Probability measures on $\Pi$, or L\'evy measures on $\mbr^d$}
        \sym{$m_\mathcal{A},m_\mB$}{Mean vectors of L\'evy generator $\mathcal{A},\mB$}
        \sym{$M(\mbr^d)$}{Space of squared $d\times d$ real matrices}
        \sym{$\mcM(\Pi)$}{Space of bounded (Borel) measures on $\Pi$}
        \sym{$\omega_c^\pm$}{Pointwise upper/lower transport derivative}
        \sym{$\omega_c^*$}{Dual transport derivative}
        \sym{$\Pi$}{State (Polish) space}
        \sym{$\mcP(\Pi)$}{Space of probability (Borel) measures on $\Pi$}
        \sym{$\mcp_2(\mbr^d)$}{Space of probability measures $\mu$ with finite second moments}
        \sym{$Q_\mA,Q_\mB$}{covariance matrices of L\'evy generators $\mA,\mB$}
        \sym{$\mathcal{S}_{\geq0}(\mathbb{R}^d)$}{Space of nonnegative definite $(d\times d)$-matrices with real entries}
        \sym{$\si$}{Coupled diffusion matrix}
        \sym{$\{T_t\}_{t\ge 0}$}{Markov or Feller semigroup}
        \sym{$\theta_c$}{Markovian transport derivative}
        \sym{$\ta_2,\om_2^\pm,\om_2^*$}{Transport derivatives w.r.t. the squared cost $c_2$}
        \sym{$\omega_2'$}{Restricted dual transport derivative}
        \sym{$\mcW_2$}{Wasserstein 2-metric between probability measures}
        \sym{$\mathcal{W}_\mathcal{G}$}{Wasserstein generator metric}
        \sym{$\mathcal{W}_\Lambda$}{$(\mcc_2^\La)^{1/2}$, L\'evy-Wasserstein metric}
        \sym{$\mathcal{W}_\mathcal{S}$}{Bures-Wasserstein distance}
        \sym{$\{X_t\}_{t\ge 0}, \{Y_t\}_{t\ge 0}$}{Markov, Feller or L\'evy processes}
\end{longtable}

%% file: S2.tex
\section{Framework, Research Problems, and Main Results}\label{sec-setting}

\newcommand{\WG}{\mcW_{\mcg}^\La}

In this section, we formulate the optimal Markovian coupling problem and the associated optimality gap problem in the abstract setting of Feller generators, and present the background required for their analysis. We begin with notations and the general framework of the study, followed by a brief review of Feller and L\'evy processes and their generators. Key notions, particularly transport derivatives, are introduced in Section~\ref{sec2.3}. The two research problems are then stated in Sections~\ref{sec2.4} and~\ref{sec2.5}, and related results from the literature are discussed in Section~\ref{sec2.6}. Our main contributions, concerning the case of L\'evy processes, are presented in Sections~\ref{sec2.7} and~\ref{sec2.8}. We conclude with a discussion of the results, proof strategies, and the organization of the remainder of the article.

\subsection{Notations and Feller semigroups/generators}\label{sec2.1}

Throughout this paper, we denote $\Pi$ a state space, which is a \emph{Polish} (completely metrizable separable) space. 
Let us introduce also the following notations for spaces of functions, measures, generators, pairings, and actions:
\begin{itemize}
	\item $\mcl{B}(\Pi)$: the space of Borel functions $f:\Pi\to\mbr$.
	
	\item $\mcl{B}_b(\Pi)$: the space of bound Borel functions $f:\Pi\to\mbr$.
	
	\item $C(\Pi)$: the space of continuous functions $f:\Pi\to\mbr$.

    \item $C_b(\Pi)$: the space of bounded continuous functions $f:\Pi\to\mbr$, with the supremum norm $\|f\|=\sup_{x\in \Pi} |f(x)|.$
    
	\item $C_0(\Pi)$: the space of continuous functions $f:\Pi\to\mbr$ vanishing at infinity.
	
	\item $C_c(\Pi)$: the space of continuous functions $f:\Pi\to\mbr$ with compact support. 
	
	\item $\mcl{M}(\Pi)$: the space of bounded (Borel) measures on $\Pi$.

	\item $\mcp(\Pi)$: the space of probabilty (Borel) measures on $\Pi$.
	
	\item $\mcg(\Pi)$: the space of all Feller generators on $\Pi$, see the following subsubsection.
	
	\item $\inn{\cdot,\cdot}:\mcl{M}(\Pi)\times C_b(\Pi)\to\mbr$: the natural pairing between bounded measures and bounded continuous functions
	\begin{align*}
		\inn{\mu,f}&= \int_{\Pi} fd\mu. 
	\end{align*}	
	This notion of pairing may potentially extend to unbounded functions $f$ or unbounded measures $\mu$ (e.g., L\'evy measures), so long as the integral makes sense. 
	\item \emph{Left action:} Given an operator $T:C_b(\Pi)\to C_b(\Pi)$ (for example, a Markov kernel, a Feller semigroup, or a generator) and a measure $\mu \in \mathcal{M}(\Pi)$, we define the \emph{left action} of $T$ on $\mu$, denoted $\mu T$, by
	\begin{align*}
		\langle \mu T, f \rangle := \langle \mu, T f \rangle, \qquad \text{for all } f \in C_b(\Pi).
	\end{align*}
	In particular, if $T$ is a Markov kernel or a Feller semigroup, then $\mu T$ is again a measure.
\end{itemize}

In the present work, although the research problems are formulated in the abstract framework introduced above, we will mainly concentrate on the case $\Pi=\mbr^d$. 
Accordingly, the spaces of continuous functions, measures, and generators will be understood with $\Pi=\mbr^d$. 
For later use, we also introduce the following spaces:
\begin{itemize}
	\item $C_2(\mbr^d)$: the space of continuous functions satisfying the quadratic growth bound for some constant $C>0$:
	\begin{align}\label{eq:quad-bound}
		|\vphi(x)| \le C(1+|x|^2), \quad x\in\mbr^d,
	\end{align}
	\item $C^2(\mbr^d)$: the space of twice continuously differentiable functions;
	\item $C_b^2(\mbr^d)$: the space of bounded twice continuously differentiable functions;
	\item $C_2^2(\mbr^d)=C_2(\mbr^d)\cap C^2(\mbr^d)$: the space of twice continuously differentiable functions $\vphi$ satisfying the quadratic growth bound \eqref{eq:quad-bound}.
	
	\item $\mcp_2(\mbr^d)$: the space of probability measures on $\mbr^d$ with finite second moment, i.e.
	\[
	\int_{\mbr^d} |x|^2\, d\mu(x)<\infty.
	\]
\end{itemize}

\subsubsection{Markov/Feller semigroups, generators, and processes}
We begin with a brief review of Markov and Feller semigroups, their generators, and the associated processes, as these form the central framework of the present work. For a comprehensive treatment, we refer the reader to \cite{ethier2009markov,kuhn2017,liggett2010continuous}.

A \emph{Markov semigroup} is a family of bounded linear operators $\{T_t\}_{t\ge 0}$ on the space $\mcl{B}_b(\Pi)$ of bounded Borel measurable functions, satisfying: 
\begin{itemize}
	\item Identity: $T_0 = I$.
	\item Semigroup property: for all $\vphi \in \mcl{B}_b(\Pi)$ and $s,t \ge 0$,  
	\[
	T_{t+s}\vphi = T_t(T_s\vphi).
	\]
	\item Positivity: if $\vphi \ge 0$, then $T_t\vphi \ge 0$ for all $t \ge 0$.
	\item Conservativeness: $T_t 1 = 1$ for all $t \ge 0$. 
\end{itemize}
A canonical example is the semigroup induced by a \emph{Markov process} $\{X_t\}_{t \ge 0}$, given by
\begin{align}\label{rep-mark}
	T_t\vphi(x) = \mbe^x[\vphi(X_t)], \quad \vphi \in \mcl{B}_b(\Pi), \; x \in \Pi.	
\end{align}
In this case, $\{T_t\}_{t\ge 0}$ admits a \emph{transition kernel} $\{\ka_t(x,dy)\}_{t\ge 0,\,x\in\Pi}$ such that
\begin{align*}
	T_t\vphi(x) &= \int_{\Pi} \vphi(y)\,\ka_t(x,dy) \;=\; \inn{\ka_t(x),\vphi}. 
\end{align*}
This kernel representation also permits extending $T_t\vphi$ to certain unbounded functions whenever the integral is well defined. In terms of the left action notation introduced earlier, we may write
\begin{align*}
	\ka_t(x) = \de_x T_t. 
\end{align*}

A \emph{Feller semigroup}\footnote{In \cite{liggett2010continuous}, it is referred to as a \emph{probability semigroup}.} is a Markov semigroup with the additional \emph{Feller property}:
\begin{itemize}
	\item $T_t(C_0(\Pi)) \subset C_0(\Pi)$ for all $t \ge 0$;
	\item $\lim_{t\searrow 0} T_t\vphi = \vphi$ in $C_0(\Pi)$ for all $\vphi \in C_0(\Pi)$. 
\end{itemize}
The \emph{infinitesimal generator} $\mA$ of a Feller semigroup and its domain $D(\mA)$ is defined by 
\begin{align}\label{def:domain}
	D(\mA) &:= \Big\{ \vphi \in C_0(\Pi): 
	\lim_{t\searrow 0}\frac{T_t\vphi - \vphi}{t} 
	\ \text{exists in } C_0(\Pi) \Big\},\qquad 
	\mA \vphi := \lim_{t\searrow 0}\frac{T_t\vphi-\vphi}{t}.
\end{align}
The domain $D(\mA)$ is a dense subspace of $C_0(\Pi)$, and $(\mA,D(\mA))$ is a (typically unbounded) closed operator on $C_0(\Pi)$.

To highlight the connection between the generator $\mA$ and its semigroup $\{T_t\}_{t\ge 0}$, it is customary to use the formal notation
\begin{align*}
	T_t = e^{t\mA}, \qquad t \ge 0,
\end{align*}
which expresses the fact that the semigroup is generated by $\mA$ in the sense of \eqref{def:domain}. 
As introduced earlier, the collection of all such generators on $\Pi$ is denoted by $\mcg(\Pi)$ and referred to as the set of \emph{Feller generators}. 
Every Feller semigroup arises from a Markov process $\{X_t\}_{t\ge 0}$ via \eqref{rep-mark}, called a \emph{Feller process}. 
Classically, there is a one-to-one correspondence between Feller processes, their semigroups, and their generators, providing a powerful analytic framework for their study; see, for example, \cite{ethier2009markov,liggett2010continuous}.

\subsubsection{Extended domains of generators} 

For a Feller generator $\mathcal{A} \in \mathcal{G}(\Pi)$, the domain $D(\mathcal{A})$ consists of all $\vphi \in C_0(\Pi)$ for which the limit in \eqref{def:domain} exists in the uniform topology.  
This definition, though natural in semigroup theory, is often too restrictive: for example, the quadratic cost $c(x,y)=\tfrac{1}{2}|x-y|^2$ does not lie in $C_0(\mbr^{2d})$.  
Yet $e^{t\mathcal{A}}\vphi$ can still be defined for such unbounded $\vphi$ whenever the integral against the transition kernel is well-defined.  
To handle these cases, one introduces the \emph{extended domain} of the generator.

\begin{definition}[Continuity domain and extended domain of a generator]\label{def:extdom}
	Let $\mA \in \mcg(\Pi)$.  
	The \emph{continuity domain} $C(\mA)$ of $\mA$ is the vector space of all $\vphi \in C(\Pi)$ such that 
	\[
	e^{t\mA}\vphi \in C(\Pi) \quad \forall\, t \ge 0, 
	\qquad \text{and} \qquad 
	e^{t\mA}\vphi \to e^{s\mA}\vphi \ \text{locally uniformly on $\Pi$ as } t \to s.
	\]
	The \emph{extended domain} $\bar D(\mA)$ of $\mA$ consists of all $\vphi \in C(\mA)$ such that the following holds: for each $t\ge 0$, there exists $g_t \in C(\mA)$ such that 
	\[
	\frac{e^{(t+h)\mA}\vphi - e^{t\mA}\vphi}{h} \;\to\; g_t 
	\qquad \text{locally uniformly on $\Pi$ as } h \searrow 0.
	\]
	In this case, we continue to write $\mA\vphi := g_0$ for the limit in \eqref{def:domain}, now understood in the local uniform sense.
\end{definition}

Clearly, we have $C_b(\Pi)\subset C(\mA)$ and $D(\mA)\subset \bar D(\mA)$. 
Moreover, for any $\vphi \in \bar D(\mA)$ and $t \ge 0$, one has $e^{t\mA}\vphi \in \bar D(\mA)$, and it holds
\begin{align*}
	\mA e^{t\mA}\vphi 
	= \lim_{h \searrow 0} \frac{e^{(t+h)\mA}\vphi - e^{t\mA}\vphi}{h} 
	= e^{t\mA}\mA \vphi.
\end{align*}
By the fundamental theorem of calculus, for any $\vphi \in \bar D(\mA)$ it holds
\begin{align}\label{eq:int-eq}
	e^{t\mA}\vphi 
	= \vphi + \int_0^t e^{s\mA}\mA\vphi \, ds. 
\end{align}

In the present work, our focus lies on the \emph{lower right-hand Dini derivative} of trajectories of the form $t \mapsto e^{t\mA}\vphi,$
where $\vphi \in \mcl{B}(\Pi)$ is a measurable test function without any regularity assumptions.  
This leads naturally to the following notion.  

\begin{definition}[Lower generator]\label{def:lower-gen}
	Let $\mA \in \mcg(\Pi)$. 
	The \emph{lower generator} $\mA^-$ is the operator acting on $\mcl{B}(\Pi)$ by
	\begin{align*}
		\mA^- \vphi(x) 
		:= \liminf_{h \searrow 0} \frac{e^{h\mA}\vphi(x) - \vphi(x)}{h},
		\qquad \vphi \in \mcl{B}(\Pi).
	\end{align*}
\end{definition}

Clearly, if $\vphi \in \bar D(\mA)$, then $\mA^- \vphi = \mA \vphi$.  
One may also define the \emph{upper generator} $\mA^+$ analogously, by replacing $\liminf$ with $\limsup$.

%

\subsection{L\'evy processes, generators, triplets}\label{sec2.2}
Our primary interest lies in L\'evy processes. We now provide a brief overview of their definition and generators; for a comprehensive treatment, see \cite{bottcher2013,Sato1999}

\subsubsection{Basics of L\'evy processes} 
A \emph{L\'evy process} \(\{X_t\}_{t \ge 0}\) on \(\mbr^d\) is a Feller process with \emph{stationary} and \emph{independent increments}. By the \emph{L\'evy--Khintchine formula}, the associated Feller generator \(\mA\), defined for test functions \(\varphi \in C_c^2(\mbr^d)\), is given by
\begin{align}\label{LK-decomp}
	\mA \varphi = \kappa \cdot \nabla \varphi + \f 12\tr(\alpha D^2 \varphi) + \int_{\mbr^d \setminus \{0\}} [\varphi(x + x') - \varphi(x) - \chi_{B_1(0)}(x')\, x' \cdot \nabla \varphi]\, d\mu(x').
\end{align}
where the triple \((\ka,\alpha, \mu)\) is called the \emph{L\'evy triplet} and consists of:
\begin{itemize}
	\item \(\ka \in \mbr^d\), a vector representing the \emph{drift} component;
	\item \(\al \in \mcl{S}_{\ge 0}(\mbr^d)\), a nonnegative definite matrix representing the \emph{diffusion} component;
	\item \(\mu\), a \emph{L\'evy measure}, i.e., a Borel measure on \(\mbr^d\) satisfying
	\begin{align}\label{Levy-m-def}
		\mu(\{0\})=0,\qquad \int_{\mbr^d \setminus \{0\}} \min\{1, |x|^2\}\, d\mu(x) < \infty.
	\end{align}
\end{itemize}
The three terms in \eqref{LK-decomp} are referred to as the \emph{drift}, \emph{diffusion}, and \emph{jump} parts of the generator \(\mA\), respectively.

\subsubsection{L\'evy generators with finite second moments}
Let \(\Lambda(\mathbb{R}^d)\) denote the family of all L\'evy measures on \(\mathbb{R}^d\), that is, measures \(\mu\) satisfying Condition \eqref{Levy-m-def}. As before, \(\mathcal{G}(\mathbb{R}^d)\) denotes the space of all Feller generators on \(\mathbb{R}^d\). We write \(\mathcal{G}^\Lambda(\mathbb{R}^d) \subset \mathcal{G}(\mathbb{R}^d)\) for the subspace of all \emph{L\'evy generators} on \(\mathbb{R}^d\), i.e., Feller generators of the form \eqref{LK-decomp}. 
Specifically, \emph{the Greek letter $\La$ is reserved for any notions that is related to ``L\'evy}."

For \(p \in [1, \infty)\), denote by \(\Lambda_p(\mbr^d) \subset \Lambda(\mbr^d)\) the class of L\'evy measures with finite \(p\)-th moment outside the unit ball:
\begin{align}\label{pmom-levy}
	\int_{B_1(0)^c} |x|^p\, d\mu(x) < \infty.
\end{align}
A L\'evy process \(\{X_t\}_{t \ge 0}\) has finite \(p\)-th moments for $p\ge 1$, i.e., \(\mathbb{E}[|X_t|^p] < \infty\) for all \(t \ge 0\), if and only if its generator (given by \eqref{LK-decomp}) is associated with a L\'evy measure \(\mu \in \Lambda_p(\mbr^d)\), see \cite{Sato1999}. We note also the inclusion holds: if $p\le q$, then 
\begin{align*}
	\La_q(\mbr^d)\subset \La_p(\mbr^d). 
\end{align*}

In this work, we focus on generators whose associated L\'evy processes that admit finite second moments, which corresponds to L\'evy measures \(\mu\) with finite second moment. Combining the conditions \eqref{Levy-m-def} and \eqref{pmom-levy}, $\mu\in \La_2(\mbr^d)$ if and only if $\mu$ admits (globally) finite second moment:
\begin{align}\label{levy-2moment}
	\int_{\mathbb{R}^d } |x|^2 \, d\mu(x) < \infty.
\end{align}
Similarly, let \(\mathcal{G}^\Lambda_2(\mathbb{R}^d) \subset \mathcal{G}^\Lambda(\mathbb{R}^d)\) denote the class of all L\'evy generators \(\mathcal{A}\) taking the form \eqref{LK-decomp} with \(\mu \in \Lambda_2(\mathbb{R}^d)\). As discussed in the preceding paragraph, $\mA\in\mcg_2^\La(\mbr^d)$ if and only if the semigroup admits a finite second moment, that is, $e^{t\mA}[|x|^2](0)<\infty$  for all $t\ge 0$. 

Recall the spaces \(C_2(\mathbb{R}^d)\), \(C_2^2(\mathbb{R}^d)\) introduced in Section~\ref{sec2.1}, together with the continuity domain \(C(\mathcal{A})\) and the extended domain \(\bar D(\mathcal{A})\) of a generator \(\mathcal{A}\) from Definition \ref{def:extdom}.  
If \(\mathcal{A} \in \mathcal{G}_2^\Lambda(\mathbb{R}^d)\), then  
\[
C_2(\mathbb{R}^d) \subset C(\mathcal{A}), 
\qquad 
C_2^2(\mathbb{R}^d) \subset \bar D(\mathcal{A}).
\]
In particular, the extended domain \(\bar D(\mathcal{A})\) of a L\'evy generator with finite second moments contains all quadratic functions.  
This fact is crucial, since the quadratic cost function used in our framework must belong to the domain of the generator.

\subsubsection{Global cutoff functions}

We briefly remark on the role of \emph{cutoff functions} in the definition of the jump operator in \eqref{LK-decomp}. The standard choice is the characteristic function \(\chi_{B_1(0)}\) of the unit ball, which yields the compensated expression
\[
\varphi(x + x') - \varphi(x) - \chi_{B_1(0)}(x')\, x' \cdot \nabla \varphi(x),
\]
designed to regularize the potential singularity near the origin when the L\'evy measure \(\mu\) is infinite there.
More generally, as noted in Sato \cite{Sato1999}, any compactly supported measurable function \(\chi(x')\) satisfying \(\chi(x') = 1 + o(1)\) as \(x' \to 0\) may be used in place of \(\chi_{B_1(0)}\). Different choices of cutoff result in equivalent generators on \(C_c^2(\mbr^d)\), with corresponding adjustments to the drift vector \(\kappa\). Thus, the cutoff is a matter of representation rather than intrinsic structure.

When the L\'evy measure \(\mu\) admits a finite \(p\)-th moment \eqref{pmom-levy} for some \(p \ge 1\),
the first moment outside the unit ball (as a vector in $\mbr^d$) is also finite (componentwise):
\[
\kappa' := \int_{B_1(0)^c} x'\, d\mu(x') < \infty.
\]
In this case, it is natural to adopt the \emph{global form} of the jump operator by taking \(\chi(x') \equiv 1\), which leads to
\[
\int_{\mbr^d} \left[ \varphi(x + x') - \varphi(x) - x' \cdot \nabla \varphi(x) \right] d\mu(x').
\]
This integral remains well-defined for \(\varphi \in C_c^2(\mbr^d)\) due to the cancellation near zero and the integrability of \(x'\) at infinity. The standard cutoff-based form \eqref{LK-decomp} and the global form are related by
\[
\mA\varphi(x)= -\kappa' \cdot \nabla \varphi(x) + \int_{\mbr^d} \left[ \varphi(x + x') - \varphi(x) - \chi_{B_1(0)}(x')x' \cdot \nabla \varphi(x) \right] d\mu(x'),
\]
so switching to the global form simply absorbs the tail drift \(\kappa'\) into the total drift.

While the global form is less common in the literature, it proves useful when studying couplings of processes or generators (see, e.g., \cite{Kolokoltsov2010LevyKhintchine}), 
as it simplifies both computations and comparisons. 
For notational consistency—especially in our later treatment of coupling generators—we adopt this form throughout the paper.

\begin{convention}
	Given a L\'evy triplet \((\ka,\alpha, \mu)\) with \(\mu \in \Lambda_p(\mbr^d),p\ge 1\), we denote the associated generator by
	\begin{align}\label{La-gen-def}
		\mA = \Lag(\ka,\alpha,\mu),
	\end{align}
	defined for \(\varphi \in C_0^2(\mbr^d)\) by the \emph{global (untruncated) form} (in contrast to the \emph{local (truncated) form} in \eqref{LK-decomp}):
	\begin{align}
		(\mA \varphi)(x)
		&:= \mA^\nabla \varphi +\mA^\De \varphi +  \mA^J \varphi \nonumber\\
		&:= \kappa \cdot \nabla \varphi + \f 12\tr[\alpha D^2 \varphi] + \int_{\mbr^d \setminus \{0\}} \left[ \varphi(x + x') - \varphi(x) - x' \cdot \nabla \varphi(x) \right] d\mu(x'). \label{LK-rep2}
	\end{align}
	This form is well-defined under \eqref{pmom-levy} and will be used as the default representation throughout the paper. Unless otherwise stated, all L\'evy measures \(\mu\) considered in this paper are assumed to have finite second moment \eqref{levy-2moment},
	i.e., \(\mu \in \Lambda_2(\mbr^d)\).
\end{convention}

\subsubsection{Mean vectors and covariance matrices of L\'evy processes}

Given a L\'evy process \( \{X_t\}_{t \geq 0} \) on \(\mathbb{R}^d\) with generator \( \mathcal{A} \), suppose that the process---equivalently, the L\'evy measure \( \mu \) restricted on $B_1(0)^c$---admits a finite first moment. Then there exists a vector \( m_{\mathcal{A}} \in \mathbb{R}^d \), called the \emph{mean vector} of the process (or of the generator \( \mathcal{A} \)), such that the expectation of the process starting from the origin satisfies
\[
e^{t\mathcal{A}}[x](0) = \mathbb{E}^0[X_t] = t m_{\mathcal{A}}.
\]
Here, the expression \( e^{t\mathcal{A}}[x] \) denotes the action of the semigroup \( e^{t\mathcal{A}} \) on the coordinate function \( x = (x_1, x_2, \ldots, x_d) \), applied componentwise.
Taking the time derivative at \( t = 0 \), we obtain the identity
\begin{align}\label{mean-vec}
	m_{\mathcal{A}} = \mathcal{A}[x](0)=\lim_{t\searrow 0} \frac{e^{t\mA}[x](0)}{t} ,
\end{align}
where \( \mathcal{A}[x] \) is interpreted componentwise.
Alternatively, for any vector \( \theta \in \mathbb{R}^d \), define the linear functional \(\vphi:\mbr^d\to\mbr, \vphi(x) := x^\top \ta\). Then $\vphi\in \bar D(\mA)$, and the generator acts on \( \vphi \) as
\[
(\mathcal{A}\vphi)(x) = \ta^\top(x+tm_\mA). 
\]

The mean vector \( m_{\mathcal{A}} \) can be explicitly expressed in terms of the L\'evy triplet \( (\ka, \alpha, \mu) \), but its precise form depends on the choice of cutoff function used in the definition of the jump operator. For example, under the standard (localized) representation \eqref{LK-decomp}, which uses the cutoff function \( \chi_{B_1(0)} \), the mean vector is given by
\[
m_{\mathcal{A}} = \kappa + \int_{B_1(0)^c} x'\, d\mu(x').
\]
In contrast, for the global (untruncated) form \eqref{LK-rep2}, where the compensation term is applied over the entire domain \( \mathbb{R}^d \setminus \{0\} \), the mean vector reduces to
\[
m_{\mathcal{A}} = \kappa.
\]
Both expressions follow directly from the general identity \eqref{mean-vec}.

The second-order structure of a L\'evy process can be described in terms of its \emph{covariance matrix}.  
Assume that the process admits a finite second moment, equivalently, the L\'evy measure \( \mu\in \La_2(\mbr^d) \) or the L\'evy generator $\mA\in\mcg_2^\La(\mbr^d)$.  
Then there exists a symmetric nonnegative definite matrix \( Q_{\mathcal{A}}  \), called the \emph{covariance matrix} of the process (or of the generator \( \mathcal{A} \)), such that the centered second moment satisfies
\[
\mathrm{Cov}^0(X_t) = \mathbb{E}^0\!\left[(X_t - t m_{\mathcal{A}})(X_t - t m_{\mathcal{A}})^\top\right] = t\, Q_{\mathcal{A}}.
\]
Equivalently, for any quadratic function $\vphi(x)=x^\top Ux$, where $U\in M(\mbr^d)$ is a $(d\times d)$-matrix, we have $\vphi\in \bar D(\mA)$, and
\begin{align*}
	(\mA\vphi)(x) = \tr(UQ_\mA)+ 2x^\top U m_\mA. 
\end{align*}
Particularly the following identity holds:
\begin{align*}
	(e^{t\mA}\vphi)(x)&= (x+tm_\mA)^\top U (x+tm_\mA)+t\,\tr(UQ_\mA). 
\end{align*}

In terms of the L\'evy triplet \( (\kappa, \alpha, \mu) \), the covariance matrix takes the explicit form
\[
Q_{\mathcal{A}} = \alpha \;+\; \int_{\mathbb{R}^d} x x^\top \, d\mu(x),
\]
when expressed in the global representation \eqref{LK-rep2}.  
Under the localized representation \eqref{LK-decomp}, the integral is restricted to \(B_1(0)\), so that
\[
Q_{\mathcal{A}} = \alpha \;+\; \int_{B_1(0)} x x^\top \, d\mu(x).
\]
As in the case of the mean vector, the precise decomposition depends on the chosen cutoff function, but the covariance matrix itself is intrinsic to the process.

\subsection{Optimal transport between Markov flows}\label{sec2.3}
Let us return to the abstract general setting. 
Let $\Pi$ be a state space equipped with a cost function \( c: \Pi^2 \to [0, \infty) \). 
Given two probability measures \( \mu, \nu \in \mathcal{P}(\Pi) \), we denote by \( \Gamma(\mu, \nu) \) the set of all \emph{couplings} of \( \mu \) and \( \nu \), that is, the set of probability measures \( \gamma \in \mathcal{P}(\Pi^2) \) such that
\begin{align*}
	\gamma(E \times \Pi) = \mu(E), \qquad \gamma(\Pi \times E) = \nu(E), \qquad \text{for all measurable } E \subset \Pi.
\end{align*}
The \emph{\( c \)-optimal transport cost} between \( \mu \) and \( \nu \) is defined as
\begin{align}\label{c-cost-def}
	\mathcal{C}_c(\mu, \nu) := \inf_{\gamma \in \Gamma(\mu, \nu)} \int_{\Pi^2} c(x, y) \, d\gamma(x, y).
\end{align}
Any coupling \( \gamma \in \Gamma(\mu, \nu) \) that attains this infimum is called a \emph{\( c \)-optimal coupling}.  
The \emph{\( c \)-optimal transport problem} asks whether such an optimal coupling exists. Existence is guaranteed under mild assumptions—for example, if \( c \) is lower semicontinuous.%
\footnote{A finite \( c \)-moment condition on \( \mu \) and \( \nu \) is often imposed to ensure that \( \mathcal{C}_c(\mu, \nu) < \infty \).}

The \emph{Kantorovich dual formulation} of \eqref{c-cost-def} is given by
\begin{align}\label{K-dual1}
	\sup \left[ \int_\Pi \varphi(x)\,d\mu(x) + \int_\Pi \psi(y)\,d\nu(y) \right],
\end{align}
where the supremum is taken over all continuous functions \( \varphi, \psi : \Pi \to \mathbb{R} \) satisfying the pointwise inequality
\[
	\varphi(x) + \psi(y) \le c(x, y) \quad \text{for all } x, y \in \Pi.
\]
The Kantorovich duality theorem asserts that, under appropriate conditions on \( c \), \( \mu \), and \( \nu \), the infimum in \eqref{c-cost-def} coincides with the supremum in \eqref{K-dual}.

In many applications, a central question is how the \( c \)-optimal transport cost between two probability measures evolves under the action of Markov flows. Let \( \mathcal{A}, \mathcal{B} \in \mathcal{G}(\Pi) \) be Feller generators on \( \Pi \), and $\mu,\nu$ be two probability measures. One then studies the evolution of the function
\[
	t \mapsto \rho(t)=\rho(t;\mu,\nu) := \mathcal{C}_c(\mu e^{t\mathcal{A}}, \nu e^{t\mathcal{B}}),
\]
which tracks the \( c \)-optimal transport cost between the time-evolved measures \( \mu e^{t\mathcal{A}} \) and \( \nu e^{t\mathcal{B}} \). A key goal is to bound the growth of \( \rho(t) \) in terms of the generators \( \mathcal{A}, \mathcal{B} \) and the initial measures \( \mu, \nu \). Leveraging the semigroup property of the flows, such bounds can often be derived by analyzing the derivative of \( \rho(t;\mu,\nu) \) at time \( t = 0 \) for $\mu=\de_x,\nu=\de_y,x,y\in\Pi$. This motivates the following definition. 

\begin{definition}[Pointwise transport derivatives between Markov flows]\label{diff-opt-def}
	Let $\mA,\mB \in \mcg(\Pi)$.  
	The \emph{pointwise upper transport derivative} is the map 
	\begin{align}\label{om-def}
		\omega_c^+(x,y;\mA,\mB) := 
		\limsup_{t \searrow 0} \frac{\mcc_c(\delta_x e^{t\mA},\, \delta_y e^{t\mB}) - c(x,y)}{t},
		\qquad (x,y)\in\Pi^2.	
	\end{align}
	Similarly, the \emph{pointwise lower transport derivative} is given by 
	\[
	\omega_c^-(x,y;\mA,\mB) := 
	\liminf_{t \searrow 0} \frac{\mcc_c(\delta_x e^{t\mA},\, \delta_y e^{t\mB}) - c(x,y)}{t},\qquad (x,y)\in \Pi^2.
	\]
\end{definition}

\begin{remark}\label{rem:om-bdd}
	By definition, one always has $\om_c^-(x,y;\mA,\mB) \le \om_c^+(x,y;\mA,\mB)$.  
	A natural question is whether these two bounds coincide, that is, whether the map 
	\(
	t \mapsto \mcc_c(\de_x e^{t\mA}, \de_y e^{t\mB})
	\)
	is differentiable at $t=0$. This is related to the \emph{optimality gap problem}. See Problem \ref{prob:ogp}. 
\end{remark}

The notion of the \emph{pointwise transport derivative} provides a local-in-time measure of how the optimal transport cost evolves under the action of two Markov semigroups. It was first appeared in \cite{alfonsi2018} (in the case $\Pi=\mbr^d$ and $c(x,y)=|x-y|^p$) and further developed in \cite{LimTeoh2025}, where it is used to study the stability of optimal transport costs along Markov flows.
In particular, it is shown that if there exist constants \( \alpha, \beta > 0 \) such that
\[
    \omega^+_c(x, y; \mathcal{A}, \mathcal{B}) \le \alpha + \beta c(x, y) \quad \text{for all } x, y \in \Pi,
\]
then the following exponential stability estimate holds: for any pair of probability measures \( \mu, \nu \in \mathcal{P}(\Pi) \) with finite \( c \)-moment,
\[
    \mathcal{C}_c(\mu e^{t\mathcal{A}}, \nu e^{t\mathcal{B}}) \le \mathcal{C}_c(\mu, \nu) e^{\beta t} + \frac{\alpha}{\beta}(e^{\beta t} - 1).
\]

To estimate or control \( \omega_c^\pm \) in concrete applications, it is often convenient to consider its dual formulation. This leads to the definition of the \emph{dual transport derivative}, defined as follows:

\begin{definition}[Dual transport derivative]\label{diff-dual-def}
	Given \( \mathcal{A}, \mathcal{B} \in \mathcal{G}(\Pi) \) and points \( x_0, y_0 \in \Pi \), define the \emph{dual value} of \eqref{om-def} by
	\begin{align}\label{om-dual-def}
		\omega_c^*(x_0, y_0; \mathcal{A}, \mathcal{B}) := \sup_{(\varphi, \psi)\in \mcl{D}(x_0,y_0)} \left[ \mathcal{A} \varphi(x_0) + \mathcal{B} \psi(y_0) \right],
	\end{align}
	where $\mcl{D}(x_0,y_0)$ is the set of all pairs of functions $\varphi,\psi:\Pi\to\mbr$ such that:
	\begin{itemize}
		\item \( \varphi \in \bar D(\mathcal{A}) \), \( \psi \in \bar D(\mathcal{B}) \);
		\item $(x,y)\mapsto c(x,y)-\varphi(x)-\psi(y)$ achieves a global minimum at $(x_0,y_0)$. 
	\end{itemize}
\end{definition}

\begin{remark}\label{rem:supdom}
	Since both $\mA$ and $\mB$ annihilate constants, the functional
	\[
	(\vphi,\psi)\mapsto (\mA\vphi)(x_0) + (\mB\psi)(y_0)
	\] 
	is invariant under constant shifts of the form $\vphi \mapsto \vphi + b$ and $\psi \mapsto \psi + b'$ with $b,b'\in\mbr$. 
	Consequently, by adding suitable constants, we may assume without loss of generality that $(\vphi,\psi) \in \mcl{D}(x_0,y_0)$ in \eqref{om-dual-def} satisfies the \emph{touching condition}
	\begin{align*}
		\varphi(x) + \psi(y) \,\le\, c(x,y),
		\quad \text{with equality at } (x_0,y_0).
	\end{align*}
\end{remark}

\begin{remark}
	In Definition~\ref{diff-dual-def}, the choice of domains $\bar D(\mA)$ and $\bar D(\mB)$ is not essential. 
	One may instead work with smaller domains such as $D(\mA),D(\mB)$ (as in \cite{LimTeoh2025}), or with other cores or dense subspaces. 
	For example, when $\Pi=\mbr^d$, a natural choice is the space $C^2_2(\mbr^d)$, which is the main domain adopted in the present work. 
	In many situations these domains are dense in one another, so the value of the supremum in \eqref{om-dual-def} remains unchanged.
\end{remark}

The notion of dual transport derivative was also first explored in \cite{alfonsi2018}, and further investigated in \cite{LimTeoh2025}. 
As hinted in the definition, the notion is inspired by the theory of \emph{viscosity solutions}. In particular, the relationship between \eqref{om-dual-def} and \eqref{om-def} closely parallels the classical duality between the Kantorovich formulation \eqref{K-dual1} and the primal optimal transport problem \eqref{c-cost-def}. Just as Kantorovich duality equates the infimum over couplings with a supremum over admissible potentials, the dual formulation \eqref{om-dual-def} offers a variational perspective on the infinitesimal growth of the transport cost in \eqref{om-def}. This analogy suggests that, under suitable conditions, the primal and dual transport derivatives coincide: \( \omega_c^* = \omega_c^\pm \). Determining whether this equality holds is precisely the content of the optimality gap problem, Problem~\ref{prob:ogp}, which will also be addressed in the present work.


\subsection{Optimal Markovian couplings}\label{sec2.4}

In this subsection, we formulate our first main problem—the optimal coupling problem between two Feller processes, and equivalently, between their generators. To set the stage, we first introduce the notion of \emph{couplings between processes and generators}.

\subsubsection{Feller couplings of generators}
Let \( \{X_t\}_{t \ge 0} \) and \( \{Y_t\}_{t \ge 0} \) be two \( \Pi \)-valued stochastic processes defined on a common state space \( \Pi \), and let \( \{\mu_t\}_{t \ge 0} \) and \( \{\nu_t\}_{t \ge 0} \) denote their respective laws:
\[
    \mu_t = \mathrm{law}(X_t), \qquad \nu_t = \mathrm{law}(Y_t).
\]
A \emph{pointwise (in time) coupling} of the two processes is a family \( \{\gamma_t\}_{t \ge 0} \) of probability measures on \( \Pi^2 \) such that
\[
    \gamma_t \in \Gamma(\mu_t, \nu_t) \qquad \text{for all } t \ge 0.
\]
This notion provides a coupling between the marginal distributions at each fixed time, but it does not encode any joint temporal behavior of the processes. In particular, the family \( \{\gamma_t\}_{t \ge 0} \) is not required to arise from the law of a stochastic process on \( \Pi^2 \), and thus lacks the dynamical structure needed to study time-dependent interactions between \( X_t \) and \( Y_t \).

To overcome this limitation, one often considers a \emph{(general) coupling of stochastic processes}, defined as a \( \Pi^2 \)-valued stochastic process \( \{Z_t = (X_t', Y_t')\}_{t \ge 0} \) such that for all \( t \ge 0 \), the marginals of \( Z_t \) agree with those of the original processes:
\begin{align}\label{stoc-marginal}
    \mathrm{law}(X_t') = \mathrm{law}(X_t), \qquad \mathrm{law}(Y_t') = \mathrm{law}(Y_t).
\end{align}
If both \( \{X_t\} \) and \( \{Y_t\} \) are Feller (or more generally, Markov) processes, it is natural to restrict attention to couplings \( \{Z_t\}_{t \ge 0} \) that also preserve the Markovian structure. That is, we seek couplings that themselves form a Markov process on \( \Pi^2 \). In the Feller setting, such processes are fully characterized by their generators, which leads to the following definition of \emph{Feller couplings}. Recall the definition of continuity domain $C(\mA)$ and extended domain $\bar D(\mA)$ of a generator $\mA$ from Definition \ref{def:extdom}.

\begin{definition}[Markovian/Feller couplings]\label{def:feller-couple}
	Let \( \mathcal{A}, \mathcal{B} \in \mathcal{G}(\Pi) \) be two Feller generators. 
	
	(i) A Markov semigroup $\{T_t\}_{t\ge 0}$ is a \emph{Markovian coupling semigroup} of $\mA,\mB$ if it holds for all $\vphi\in C(\mA),\psi\in C(\mB)$ and $t\ge 0$ 
	\begin{align*}
		T_t[\vphi\otimes 1]= e^{t\mA\vphi}\otimes 1,\qquad T_t[1\otimes\psi]= 1\otimes e^{t\mB\psi}. 
	\end{align*}
	Here, \( \varphi \otimes 1 \) and \( 1 \otimes \psi \) denote the functions on \( \Pi^2 \) defined by
	\[
	(\varphi \otimes 1)(x, y) := \varphi(x), \qquad (1 \otimes \psi)(x, y) := \psi(y).
	\]
	
	(ii) 
	A \emph{Feller coupling generator} of $\mA,\mB$ is a Feller generator \( \mathcal{J} \in \mathcal{G}(\Pi^2) \) such that $\{e^{t\mJ}\}_{t\ge0}$ is a Markovian coupling semigroup of $\mA,\mB$. That is, 
	for all \( \varphi \in C(\mA),\psi\in C(\mB) \) and all \( t \ge 0 \), we have
	\[
	e^{t\mathcal{J}}[\varphi \otimes 1] = e^{t\mathcal{A}} \varphi \otimes 1, \qquad
	e^{t\mathcal{J}}[1 \otimes \psi] = 1 \otimes e^{t\mathcal{B}} \psi.
	\]
	We denote by \( \Gamma(\mathcal{A}, \mathcal{B}) \) the set of all Feller coupling generators of \( \mathcal{A} \) and \( \mathcal{B} \). 
	
\end{definition}

\begin{remark}
	This definition coincides with that in \cite{NuskenPavliotis2019}. The notion of a coupling generator can be extended beyond the Feller setting by replacing ``Feller generators'' with their general \emph{Markovian} counterparts. Furthermore, a broader concept of coupling—formulated in terms of martingale problems—is also available, e.g., see \cite{KendallMajkaMijatovic2024}. In the present work, however, we will remain within the Feller framework.
\end{remark}

\begin{remark}
	Given two Feller generators $\mathcal{A}, \mathcal{B}$, the set of couplings $\Gamma(\mathcal{A}, \mathcal{B})$ is always nonempty. For instance, the operator
	\[
	\mathcal{J} := \mathcal{A} \otimes I + I \otimes \mathcal{B}
	\]
	belongs to $\Gamma(\mathcal{A}, \mathcal{B})$ and corresponds to the generator of the \emph{independent coupling} of the processes generated by $\mathcal{A}$ and $\mathcal{B}$.
\end{remark}

We now present an equivalent characterization of Feller coupling generators, stated as a proposition below. The proof is elementary and will be given in Appendix~\ref{Appen-Coupling}.

\begin{proposition}\label{prop:coup-char}
	Let $\mA,\mB\in\mcg(\Pi)$ and $\mJ\in \mcg(\Pi^2)$. The following are equivalent.
	
	(i) $\mcl{J}$ is a Feller coupling generator of $\mA,\mB$;
	
	(ii) for all $\varphi\in \bar D(\mA),\psi\in \bar D(\mB)$, it holds $\varphi\otimes 1 ,1\otimes \psi\in \bar D(\mJ)$ and
	\begin{align*}
		\mJ[\varphi\otimes 1]=\mA\varphi\otimes 1,\quad \mJ[1\otimes \psi]=1\otimes \mB\psi;
	\end{align*}	
	
	(iii) for all $\vphi\in D(\mA),\psi\in D(\mB)$, it holds $\vphi\otimes 1 ,1\otimes \vphi\in \bar D(\mJ)$, and the marginal condition from (ii) holds. 
\end{proposition}

\subsubsection{Optimal Markovian couplings}
Let \( \{X_t\} \) and \( \{Y_t\} \) be Feller processes on a Polish space \( (\Pi, c) \), with generators \( \mathcal{A} \) and \( \mathcal{B} \), respectively. We are interested in constructing a coupling of these processes that is \emph{optimal} with respect to the transport cost \( c \).
A natural approach is to consider a family of \emph{pointwise} couplings \( \{\gamma_t\}_{t \ge 0} \), where each \( \gamma_t \in \mathcal{P}(\Pi^2) \) is a \( c \)-optimal coupling of the marginals \( \mu_t = \mathrm{law}(X_t) \) and \( \nu_t = \mathrm{law}(Y_t) \). While this yields the minimal transport cost at each time, as mentioned earlier it lacks temporal structure for stochastic processes: there is no guarantee that \( \{\gamma_t\}_{t \ge 0} \) corresponds to the time marginals of any joint process \( \{(X_t', Y_t')\}_{t \ge 0} \) on \( \Pi^2 \).

To incorporate temporal dynamics, one instead seeks a coupling at the level of Markov processes: a \( \Pi^2 \)-valued Markov/Feller process \( \{Z_t^* = (X_t^*, Y_t^*)\}_{t \ge 0} \) whose marginals coincide with those of \( \{X_t\} \) and \( \{Y_t\} \). The aim is to construct such a process-level (Markovian) coupling that satisfies an optimality criterion with respect to the cost function \( c \), for example:
\[
    \mathbb{E}^{(x,y)}[c(X_t^*, Y_t^*)] = \inf_{\{(X_t', Y_t')\}} \mathbb{E}^{(x,y)}[c(X_t', Y_t')],
    \qquad \text{for all } (t,x,y) \in [0,\infty) \times \Pi^2,
\]
where the infimum is taken over all Markovian/Feller couplings \( \{(X_t', Y_t')\}_{t \ge 0} \) of \( \{X_t\} \) and \( \{Y_t\} \). 
Alternatively, in the semigroup formulation, we are to identify the Feller generator $\mJ_*\in\Ga(\mA,\mB)$ such that the following holds for all $\mJ\in\Ga(\mA,\mB)$
\begin{align}\label{bdd:glob-opt}
	e^{t\mJ_*}c(x,y)\le  e^{t\mJ}c(x,y),\qquad \forall (t,x,y)\in[0,\infty)\times \Pi^2.
\end{align}

However, this global-in-time optimality may be too strong: the infimum need not be attained or even well-defined at each time \( t \ge 0 \). To circumvent this, one instead considers couplings that minimize the \emph{initial rate of change} of the cost—that is, the time derivative of \( \mathbb{E}[c(X_t', Y_t')] \) at \( t = 0 \):
\begin{align*}
	\drv{}{t} \Big|_{t=0}\mbe^{(x,y)}[c(X_t',Y_t')]=\drv{}{t} \Big|_{t=0} (e^{t\mJ} c)(x,y)= (\mJ c)(x,y),
\end{align*} 
where $\mJ$ is the generator of the coupling process $\{(X'_t,Y'_t)\}$, which is a Feller coupling generator of $\mA,\mB$. 
This leads to the following optimality condition, as formulated by Chen in \cite{Chen1994}:

\begin{definition}[$c$-optimal coupling generator and Markovian transport derivative]\label{c-opt-gen}
Let $\Pi$ be a state space, $c:\Pi^2\to[0,\infty)$ be a cost function, 
and let \( \mathcal{A}, \mathcal{B} \in \mathcal{G}(\Pi) \) be two Feller generators on \( \Pi \). The \emph{Markovian transport derivative} is the function \( \ta_c(\cdot, \cdot; \mathcal{A}, \mathcal{B}) : \Pi^2 \to [-\infty, \infty] \) defined by
\begin{align}\label{ta-def}
    \ta_c(x, y; \mathcal{A}, \mathcal{B}) := \inf_{\mathcal{J} \in \Gamma(\mathcal{A}, \mathcal{B})} (\mathcal{J}^- c)(x, y),
\end{align}
where $\mJ^-$ is the lower generator given in Definition \ref{def:lower-gen}. We say that \( \mathcal{J}_* \in \Gamma(\mathcal{A}, \mathcal{B}) \) is a \emph{\( c \)-optimal coupling} of \( \mathcal{A} \) and \( \mathcal{B} \) if $c\in \bar D(\mJ_*)$ and it holds for all $(x,y)\in\mbr^{2d}$:
\begin{align}\label{inf-attain}
    (\mathcal{J}_* c)(x, y) = \ta_c(x, y; \mathcal{A}, \mathcal{B}).
\end{align}
\end{definition}

\begin{remark}
	A mild technical requirement in the definition of $\ta_c$ is that $e^{t\mJ}c$ be well defined for small times $t\ge 0$.  
	This is guaranteed, for example, if $c$ admits a subadditive bound of the form
	\begin{align}\label{eq:mildcd}
	c(x,y)\;\le\; \Phi(x)+\Psi(y), \qquad \Phi\in C(\mA),\;\Psi\in C(\mB),	
	\end{align}
	where $C(\mA)$ and $C(\mB)$ denote the continuity domains of $\mA$ and $\mB$. 
\end{remark}

\begin{remark}
	In \eqref{ta-def} we use the lower generator $\mJ^-$ to avoid the domain issue of whether $c$ belongs to $\bar D(\mJ)$.  
	In the jump-process setting of \cite{Chen1994}, the situation is simpler: $\mA$ and $\mB$ are bounded Markov jump operators with jump kernels $\{\lambda(x,dx')\}_{x\in\Pi}$, and under mild assumptions, e.g., \eqref{eq:mildcd}, the cost $c$ always lies in the extended domain of every coupling $\mJ\in\Gamma(\mA,\mB)$.  
	There the infimum in \eqref{ta-def} is taken over all Markovian couplings (not necessarily Feller), and the generator itself can be used directly.  
	For unbounded generators, however, $c$ may fail to belong to $\bar D(\mJ)$, so the lower generator $\mJ^-$ provides a robust formulation.
\end{remark}

\begin{remark}\label{rem:2.16}
	An alternative notion of the Markovian transport derivative, denoted by $\tilde{\ta}_c$, can be introduced in integral form, motivated by the semigroup identity \eqref{eq:int-eq}. 
	Specifically, $\tilde{\ta}_c$ is defined as the largest function $\ta:\Pi^2 \to [-\infty,\infty]$ such that for every $t \ge 0$, $(x,y)\in\Pi^2$, and every coupling $\mJ\in\Gamma(\mA,\mB)$,
	\[
	(e^{t\mJ}c)(x,y) \;\;\ge\;\; c(x,y) + \int_0^t (e^{s\mJ}\ta)(x,y)\,ds.
	\]
	Letting $\mathcal{H}$ denote the collection of all such admissible functions $\ta$, one has $\tilde{\ta}_c = \sup_{\ta\in\mathcal{H}}\ta$ pointwise. 
	It is immediate from the definitions that $\ta_c \le \tilde{\ta}_c$. 
	The main advantage of this integral formulation lies in its natural extension to general Markovian (not necessarily Feller) couplings, or even coupling processes. 
	On the other hand, it also introduces certain technical challenges, such as in establishing weak duality (cf. Proposition~\ref{om-ta-prop}). 
	In the L\'evy setting considered here, however, the two notions coincide.
	
\end{remark}

Given two Feller generators \( \mathcal{A}, \mathcal{B} \in \mathcal{G}(\Pi) \), the problem of determining whether there exists a $c$-optimal Markovian coupling, that is, \( \mathcal{J}_* \in \Gamma(\mathcal{A}, \mathcal{B}) \) that attains the infimum in \eqref{inf-attain}, is referred to as the \emph{\( c \)-optimal Markovian coupling problem of $\mA,\mB$}.

\begin{problem}[$c$-optimal Markovian coupling problem]\label{prob:ocp}
	Given $\mA, \mB \in \mcg(\Pi)$, does there exist a $c$-optimal coupling $\mJ_* \in \Ga(\mA, \mB)$ in the sense of Definition~\ref{c-opt-gen}?
\end{problem}

We emphasize that the above is a \emph{global optimization} problem: the quantity $(\mJ_* c_2)(x,y)$ is minimized simultaneously for all $(x,y) \in \Pi^2$, rather than at a fixed point.

\begin{remark}
	In the literature, the \emph{optimal Markovian coupling problem} is often formulated for general \emph{Markov generators}. In this work, we adopt the \emph{Feller generator} setting. This restriction is deliberate: it provides a natural topological structure on the state space, avoids certain technical difficulties arising in the absence of topology, and aligns with the functional analytic approach in our work. Consequently, all ``Markovian couplings'' considered in this paper are, in fact, \emph{Feller} couplings. We retain the term ``Markovian'' in the problem name for consistency with the established literature, but the analysis is carried out entirely within the Feller framework. 
\end{remark}

Before proceeding, we record a simple observation showing how the global-in-time optimality condition \eqref{bdd:glob-opt} immediately yields the pointwise (infinitesimal) optimality condition \eqref{ta-def} at the level of generators. We note it was also proved in \cite{Chen1994} for the case of bounded jump operators. 

\begin{proposition}\label{prop:glob-loc}
	Let $\mathcal{A}, \mathcal{B} \in \mathcal{G}(\Pi)$ and $\mathcal{J}, \mathcal{L} \in \Gamma(\mathcal{A}, \mathcal{B})$.  
	If
	\[
	(e^{t\mathcal{L}} c)(x,y) \le (e^{t\mathcal{J}} c)(x,y), \quad \forall\,x,y \in \Pi,\ \forall\, t \ge 0,
	\]
	then
	\[
	(\mathcal{L}^- c)(x,y) \le (\mathcal{J}^- c)(x,y), \quad \forall\,x,y \in \Pi.
	\]
	In particular, if $\mathcal{J}_*$ satisfies \eqref{bdd:glob-opt} for all $\mathcal{J} \in \Gamma(\mathcal{A}, \mathcal{B})$ and $c\in \bar D(\mJ)$, then $\mathcal{J}_*$ is a $c$-optimal coupling of $\mathcal{A}$ and $\mathcal{B}$.
\end{proposition}

\begin{proof}
	Fix $x,y\in\Pi$ and $t>0$. Subtract $c(x,y)$ from both sides of the assumed inequality, divide by $t$, and let $t\searrow 0$. The semigroup definition of $\mJ$ and $\mcl{L}$ then yields the claim.
\end{proof}

\begin{remark}
	It is natural to ask whether the converse of Proposition~\ref{prop:glob-loc} holds: namely, if $\mathcal{J}_*$ is a $c$-optimal coupling of $\mathcal{A}$ and $\mathcal{B}$, does it necessarily satisfy the global minimality condition \eqref{bdd:glob-opt} for all $\mathcal{J} \in \Gamma(\mathcal{A}, \mathcal{B})$?  
	Preliminary considerations suggest that this implication is \emph{not} valid in general, and constructing an explicit counterexample would be of independent interest.
\end{remark}

\subsection{Comparison of transport derivative: pointwise, dual, and Markovian}\label{sec2.5}
Given two Feller generators \( \mathcal{A}, \mathcal{B} \in \mathcal{G}(\Pi) \), generating Feller processes \( \{X_t\} \) and \( \{Y_t\} \), we have introduced three notions of transport derivative associated with the pair: the \emph{pointwise} $\om_c^\pm$, the \emph{dual} $\om_c^*$, and the \emph{Markovian} transport derivatives $\ta_c$, see Definitions \ref{diff-opt-def}, \ref{diff-dual-def} and \ref{c-opt-gen}. 
Particularly, the pointwise version captures the (least upper bound) rate of change of the optimal transport cost between the marginal laws under \emph{pointwise couplings} of \( \{X_t\} \) and \( \{Y_t\} \). 
In contrast, the Markovian version measures the derivative of the cost under the most efficient \emph{Feller coupling} of the two processes. 
Since every Feller coupling induces a pointwise coupling, the pointwise infinitesimal transport cost provides a lower bound on the Markovian counterpart. Hence we have the following proposition, whose proof is presented in Appendix \ref{Appen-genresult}. 
\begin{proposition}\label{om-ta-prop}
Let \( \Pi \) be a state space, $c:\Pi^2\to\mbr^+$ be a cost function, and let \( \mathcal{A}, \mathcal{B} \in \mathcal{G}(\Pi) \). Let \( \omega_c^\pm,\om_c^* \) and \( \ta_c \) be the transport derivative functions defined by Definitions \ref{diff-opt-def}, \ref{diff-dual-def}, and \ref{c-opt-gen}, respectively. Then for all \( x, y \in \Pi \), it holds that
\[
\omega^*_c(x, y; \mathcal{A}, \mathcal{B})\le \omega_c^-(x, y; \mathcal{A}, \mathcal{B})\le \omega_c^+(x, y; \mathcal{A}, \mathcal{B}) \le \ta_c(x, y; \mathcal{A}, \mathcal{B}).
\]
\end{proposition}


Our second research problem addresses the potential gap between the pointwise and Markovian notions of transport derivative at the level of individual states. This leads to the following question, which we refer to as the \emph{optimality gap problem}.

\begin{problem}[Optimality gap problem]\label{prob:ogp}
	Given a state space $\Pi$, a cost function $c:\Pi^2\to\mbr^+$, and two Feller generators $\mathcal{A}, \mathcal{B} \in \mathcal{G}(\Pi)$, does the following identities hold:
	\[
	\omega_c^*(x, y; \mathcal{A}, \mathcal{B}) = \omega_c^-(x, y; \mathcal{A}, \mathcal{B})=\omega_c^+(x, y; \mathcal{A}, \mathcal{B}) = \theta_c(x, y; \mathcal{A}, \mathcal{B}), 
	\qquad \text{for all } x, y \in \Pi?
	\]
\end{problem}

\begin{remark}
	If the equality holds, then $\om_c^-=\om_c^+$ implies the limit \eqref{om-def}, i.e., the derivative of $t\mapsto \mcc_c(\de_x e^{t\mA},\de_y e^{t\mA})$ at $t=0$, exists. 
\end{remark}

\begin{remark}
	The equality in Problem~\ref{prob:ogp}, if it holds, has several noteworthy implications.  
	In particular, the identity \( \theta_c = \omega_c^\pm \) would mean that for any \(\varepsilon > 0\), one can construct a coupling process \(\{(X_t, Y_t)\}_{t \ge 0}\) between the processes generated by \(\mathcal{A}\) and \(\mathcal{B}\) such that the infinitesimal growth rate of the transportation cost \(\mathbb{E}^{(x,y)}[c(X_t, Y_t)]\) at \(t = 0\) is within \(\varepsilon\) of the optimal value given by \(\omega_c^\pm\), i.e., the derivative of the optimal transport cost between \(\delta_x e^{t \mathcal{A}}\) and \(\delta_y e^{t \mathcal{B}}\), which serves as the lower bound for \(c(X_t, Y_t)\) over all possible (generic) couplings.  
	Moreover, if a \(c\)-optimal coupling process exists, the equality would ensure that the instantaneous cost growth at the level of the coupling process coincides with the pointwise optimal value obtained from the dual formulation.  
	Such a property can be particularly relevant in the analysis of interacting particle systems, where the short-time behavior of couplings between generators plays a crucial role in quantifying convergence rates, stability, or propagation of chaos. 
\end{remark}

\subsection{Existing results for optimal Markovian coupling problem}\label{subsec-history-ocp}\label{sec2.6}
Let us briefly review the existing results for the optimal Markovian coupling problem. 

Prior to the development of the optimal Markovian coupling problem with transport-type cost criteria, coupling methods were extensively studied in the context of maximizing the probability that two processes meet by a fixed time. The study of such maximal couplings dates back at least to the pioneering works of Griffeath~\cite{Griffeath1975} and Goldstein~\cite{Goldstein1978}, who constructed maximal couplings for Markov chains and more general probability spaces, respectively. In the setting of continuous-time diffusions, Lindvall and Rogers~\cite{LindvallRogers1986} established that the reflection coupling of Brownian motion achieves this maximal meeting probability, thus providing an explicit optimal coupling construction. Related problems have also been investigated in the setting of general L\'evy processes; see \cite{SchillingWang2011, BoettcherSchillingWang2011, LiangSchillingWang2020}.

The idea of constructing Markovian couplings to bound the Wasserstein distance between the laws of two Markov processes was, to the best of our knowledge, first introduced by Chen and Li~\cite{ChenLi1989}.  
They showed that one can obtain estimates in the Wasserstein \(p\)-metric, \(p = 1, 2\), between the laws of two processes by constructing a Markovian coupling that satisfies certain appropriate bounds.  
Furthermore, they established an optimality criterion and proved the existence of such couplings for the case \(p=2\), which includes the drift–diffusion setting considered in the present work.

Subsequently, Chen~\cite{Chen1994} formulated this optimality criterion~\eqref{ta-def}, along with the optimal Markovian coupling problem in a general abstract setting of Markov processes on an arbitrary state space.  
He also solved the optimal coupling problem for diffusion processes on \(\mathbb{R}^d\) with a concave cost \(c(x,y) = f(|x-y|)\), where \(f\) satisfies \(f(0) = 0\), \(f' > 0\), and \(f'' \le 0\), together with several applications~\cite{Chen2020OptimalCouplings}.  
Under a similar setting in cost, the optimal Markovian coupling problem was also studied in the L\'evy setting
\cite{LiangSchillingWang2020,KendallMajkaMijatovic2024}. 
In a related development, Zhang~\cite{Zhang2000} proved the existence of \(c\)-optimal Markovian couplings for bounded Markov jump generators on an abstract state space \(\Pi\), where \(c\) is a nonnegative lower semicontinuous cost function.

\subsection{Optimal Markovian coupling and optimality gap problem for L\'evy generators}\label{sec2.7}

We now return to the two research problems posed earlier, and specialize to the setting of a quadratic cost on the Euclidean space:
\[
\Pi = \mathbb{R}^d, 
\qquad c(x,y) = c_2(x,y) = \tfrac{1}{2}|x-y|^2.
\]
In this case, we adopt the subscript ``$2$'' in place of the general cost function $c$ in the notation for transport derivatives, e.g.\ $\theta_c = \theta_2$, $\omega_c^\pm = \omega_2^\pm$, and $\omega_c^* = \omega_2^*$.

In principle, the problems of optimal Markovian coupling and the associated optimality gap can be posed for arbitrary Feller generators $\mA,\mB\in \mcg(\mbr^d)$, or more specific class, \emph{L\'evy-type generators}. 
Such a level of generality, however, appears presently out of reach due to the analytic complexity. 
As a tractable yet structurally rich first step, we therefore focus on the translation-invariant case, namely on L\'evy generators $\mA,\mB\in \mcg_2^\Lambda(\mbr^d)$. 
This restriction preserves many of the essential features of the general problem, while allowing a more transparent development of the ideas introduced above.

Recall the notion of the dual transport derivative $\omega_c$ from Definition~\ref{diff-dual-def}.  
Since $\mathcal{A},\mathcal{B}$ are L\'evy generators, the space $C_2^2(\mathbb{R}^d)$ of quadratically bounded continuous functions serves as a common core for both.  
In this case it is more convenient to work with~\eqref{om-dual-def} restricted to this domain.  
This leads to the following definition.

\begin{definition}[Restricted dual supremum]\label{def:omega-prime}
	Let $\mathcal{A}, \mathcal{B} \in \mathcal{G}_2^\Lambda(\mathbb{R}^d)$ and $(x_0,y_0) \in \mathbb{R}^{2d}$. 
	We define the \emph{restricted dual transport derivative}
	\begin{align}\label{eq:omega-prime}
		\omega_2'(x_0, y_0; \mathcal{A}, \mathcal{B}) 
		:= \sup_{(\varphi,\psi)\in\mcl{D}'(x_0,y_0)} 
		\Big[ \mathcal{A} \varphi(x_0) + \mathcal{B} \psi(y_0) \Big],
	\end{align}
	where $\mcl{D}'(x_0,y_0)$ consists of all pairs 
	$(\varphi,\psi) \in [C^2_2(\mathbb{R}^d)]^2$ satisfying the \emph{touching condition}
	\begin{align}\label{eq:touch}
		\varphi(x) + \psi(y) \,\le\, \tfrac{1}{2}|x-y|^2 
		\quad \text{for all $(x,y)$, with equality at } (x_0,y_0).
	\end{align}
	Equivalently, the definition in~\eqref{eq:omega-prime} is unchanged if condition~\eqref{eq:touch} is replaced by requiring that 
	\(
	c_2 - \varphi \oplus \psi
	\)
	attains a global minimum at $(x_0,y_0)$; see Remark~\ref{rem:supdom}.
\end{definition}

\begin{remark}\label{rem:om'-bdd}
	Since for $\mA,\mB\in\mcg_2^\La(\mbr^d)$, $C_2^2(\mathbb{R}^d)\subset \bar D(\mA)\cap\bar D(\mB)$, clearly it holds for all $x_0,y_0\in\mbr^d$:
	\[
	\omega_2'(x_0, y_0; \mathcal{A}, \mathcal{B}) \le \omega_2^*(x_0, y_0; \mathcal{A}, \mathcal{B}).
	\]
\end{remark}

Before stating our main results, we address a technical question: does the quadratic cost $c_2$ belong to the continuity domain of a coupling generator $\mJ$ associated with two L\'evy generators?  
The following proposition gives an affirmative answer. Its proof is elementary and is included in Appendix \ref{Appen-genresult}.  

\begin{proposition}\label{prop:c2-reg}
	Let $\mA,\mB\in\mcg_2^\Lambda(\mbr^d)$ be L\'evy generators with mean vectors $m_\mA,m_\mB$ and covariance matrices $Q_\mA,Q_\mB$.  
	Define
	\[
	U(t,x,y) := (x-y+t(m_\mA-m_\mB))^\top(m_\mA-m_\mB),
	\]
	and let $\{T_t\}_{t\ge 0}$ be a Markovian coupling semigroup of $\mA,\mB$ and $\mJ\in\Gamma(\mA,\mB)$ be a Feller coupling generator.
	
	(i) It holds for all $t,h\ge 0$ and $(x,y)\in\mbr^{2d}$:
	\begin{align*}
		\frac{h}{2}\,U(t,x,y)
		&\le (T_{t+h}c_2)(x,y)-(T_tc_2)(x,y) \\
		&\le h^2\big(|m_\mA|^2+|m_\mB|^2\big)
		+ h\big[\tr(Q_\mA+Q_\mB)+U(t,x,y)\big].
	\end{align*}
	In particular, the map $t\mapsto (e^{t\mJ}c_2)(x,y)$ is locally Lipschitz. 
	
	(ii) One has $c_2\in C(\mJ)$.
		
	(iii) If $\mJ$ itself is a L\'evy generator, then $c_2\in\bar D(\mJ)$.
\end{proposition}

Our first main theorem, addressing to the optimal Markovian coupling problem between two L\'evy generators, is stated as follows. 

\begin{theorem}[Existence of $2$-optimal Markovian coupling]\label{main1}
	Let \( \mathcal{A}, \mathcal{B} \in \mathcal{G}_2^\Lambda(\mathbb{R}^d) \) be two L\'evy generators with finite second moments, and let \( \theta_2(\cdot, \cdot; \mathcal{A}, \mathcal{B}) \) denote the Markovian transport derivative of $\mA,\mB$ defined in Definition~\ref{c-opt-gen}. 
	
	(i) It holds that
	\[
	\theta_2(x, y; \mathcal{A}, \mathcal{B}) 
	= \theta_2(0, 0; \mathcal{A}, \mathcal{B}) + (m_\mA - m_\mB)^\top (x - y),
	\]
	where \( m_\mathcal{A}, m_\mathcal{B} \in \mathbb{R}^d \) are the mean vectors associated with \( \mathcal{A} \) and \( \mathcal{B} \), respectively.
	
	(ii) A $2$-optimal Markovian coupling \( \mathcal{J}_* \in \Gamma(\mathcal{A}, \mathcal{B}) \) exists. In particular, if the generators are in global form with L\'evy triplets $\mA=\Lag(\ka,\al,\mu)$ and $\mB=\Lag(\zeta,\beta,\nu)$, then $\mJ_*=\Lag(\eta_*,\si_*,\ga_*)\in \mcg_2^\La(\mbr^{2d})$ is a $2$-optimal coupling, where:
	\begin{itemize}
		\item \( \eta_* = \kappa\oplus\zeta \in \mathbb{R}^{2d} \);
		\item \( \sigma_* = \begin{bmatrix} \alpha & K_* \\ K_*^\top & \beta \end{bmatrix} \), where \( K_* \in M(\mathbb{R}^d) \) solves the semidefinite program \eqref{posdef-prog};
		\item \( \gamma_* \in \Lambda_2(\mathbb{R}^{2d}) \) is a $2$-optimal L\'evy coupling of \( \mu \) and \( \nu \), given by Definition \ref{levy-opt-prob}.
	\end{itemize}
	
	(iii) For any $2$-optimal coupling \( \mathcal{J}_* \) of $\mA,\mB$, it holds that
	\begin{align}\label{eq:optgrowth}
		(e^{t\mJ_*}c_2)(x,y)
		= c_2(x,y)+t\ta_2(x,y;\mA,\mB)+ \f {t^2}{2}|m_\mA-m_\mB|^2. 
	\end{align}
	In particular, the semigroup $\{e^{t\mJ_*}\}_{t\ge 0}$ satisfies the following \emph{minimal growth condition}: for every Markovian coupling semigroup $\{T_t\}_{t\ge 0}$ of $\mA,\mB$,  
	\begin{align}\label{eq:mingrowth}
		(e^{t\mJ_*}c_2)(x,y) \;\le\; (T_t c_2)(x,y), 
		\qquad \forall\,x,y\in\mathbb{R}^d,\; t\ge 0.
	\end{align}
\end{theorem}

\begin{remark}
 	We emphasize that the minimal growth property in \eqref{eq:mingrowth} holds among \emph{all Markovian coupling semigroups} of $\mA,\mB$, not only the \emph{Feller} ones. 
\end{remark}

\begin{remark}
	If $\mA,\mB$ are of global form with $\mA=\Lag(\ka,\al,\mu)$, $\mB=\Lag(\zeta,\be,\nu)$, the identity from (i) becomes
	\begin{align*}
		\theta_2(x, y; \mathcal{A}, \mathcal{B}) = \theta_2(0, 0; \mathcal{A}, \mathcal{B}) + (\ka-\zeta)^\top (x - y).
	\end{align*}
\end{remark}

Our second main result addresses the \emph{optimality gap problem}.

\begin{theorem}[Strong duality]\label{main2}
Let \( \mathcal{A}, \mathcal{B} \in \mathcal{G}^\Lambda_2(\mathbb{R}^d) \), and let \( \om_2^\pm,\om_2^*,\theta_2,\om_2'\) be defined as in Definitions \ref{diff-opt-def}, \ref{diff-dual-def}, \ref{c-opt-gen} and \ref{def:omega-prime}. Then, for all \( x, y \in \mathbb{R}^d \), it holds that
\begin{align*}
    \omega_2'(x, y; \mathcal{A}, \mathcal{B}) =\omega_2^*(x, y; \mathcal{A}, \mathcal{B}) =\omega_2^\pm(x, y; \mathcal{A}, \mathcal{B}) = \theta_2(x, y; \mathcal{A}, \mathcal{B}).
\end{align*}
\end{theorem}

As a byproduct of Theorem~\ref{main2}, we obtain an additivity property of the transport derivative with respect to the L\'evy–Khintchine decomposition. 

\begin{theorem}[Additivity w.r.t L\'evy-Khintchine decomposition]\label{main-LK}
	Let $\mA,\mB\in\mcg_2^\La(\mbr^d)$ be in the global form, and $\mA^\bullet,\mB^\bullet$, $\bullet\in\{\nabla,\De,J\}$ be the drift, diffusion,  and jump part of $\mA,\mB$ respectively. It holds for all $x,y\in \mbr^d$: 
	\begin{align*}
		\ta_2(x_0,y_0;\mA,\mB)&= \sum_{\bullet \in \{\nabla,\De,J\}}\ta_2(x,y;\mA^\bullet,\mB^\bullet).  
	\end{align*}
\end{theorem}

%
%

\subsection{Wasserstein-type metric on the space \texorpdfstring{$\mcg_2^\Lambda(\mathbb{R}^d)$}{G\_\{2\}^\textbackslash Lambda(R^d)}}\label{sec2.8}

An interesting application of Theorems~\ref{main1} and~\ref{main2} is the introduction of a Wasserstein-type metric on the space of L\'evy generators \(\mathcal{G}_2^\Lambda(\mathbb{R}^d)\) and the space $\La_2(\mbr^d)$ of L\'evy measures with finite second moments. 

In what follows, we define the corresponding functional, termed the \emph{Wasserstein generator metric}, as
\begin{align}\label{def:Wg}
	\mcW_\mcg(\mA, \mB)^2 := \frac{1}{2} |m_\mA - m_\mB|^2+ \theta_2(0,0;\mA,\mB) ,
\end{align}
where \( m_\mA \) and \( m_\mB \) denote the mean vectors associated with the generators \( \mA, \mB \in\mcg_2^\La(\mbr^d)\), respectively.
Since any L\'evy generator in $\mcg_2^\La(\mbr^d)$ is associated with a L\'evy triplet (in the global form \eqref{LK-rep2}), the metric \eqref{def:Wg} can be expressed in term of the triplet. 
As seen in Theorem \ref{main-LK}, this functional admits a natural decomposition into separate contributions from the drift, diffusion, and jump parts (in the global form). Specifically, we have the drift part $\ta_2(0,0;\mA^\nabla,\mB^\nabla)=0$, and thus
\begin{align*}
	\wg(\mA,\mB)^2=\f 12|m_\mA-m_\mB|^2+\ta_2(0,0;\mA^\De,\mB^\De)+\ta_2(0,0;\mA^J,\mB^J). 
\end{align*}
Moreover, if 
\(
\mA = \Lag(\kappa, \alpha, \mu), \mB = \Lag(\eta, \beta, \nu),
\)
then the above reduces to
\[
\wg(\mA, \mB)^2 = \frac{1}{2} \left( |\kappa - \eta|^2 + \mcl{W}_{\mcl{S}}(\alpha, \beta)^2 + \mcl{W}_\Lambda(\mu, \nu)^2 \right),
\]
where: \( |\kappa - \eta| \) is the standard Euclidean norm,
\( \mcl{W}_{\mcl{S}} \) denotes the \emph{Bures--Wasserstein distance} between symmetric nonnegative definite matrices,
\( \mcl{W}_\Lambda := (\mcc_2^\Lambda)^{1/2} \) is the square root of the \emph{L\'evy transport cost} between Lévy measures, which will be called the \emph{L\'evy-Wasserstein metric}. These functional will be defined in Sections \ref{sec-lotp} and \ref{sec-opt-markov}.

The main result of this subsection can be stated as follows.

\begin{theorem}\label{main3}
	The functional $\mathcal{W}_{\mathcal{G}}$ from \eqref{def:Wg} defines a complete and separable metric on the space of L\'evy generators $\mathcal{G}_2^\Lambda(\mathbb{R}^d)$.
\end{theorem}

As a direct consequence, $\mathcal{W}_\Lambda$ also induces a complete and separable metric on the space $\Lambda_2(\mathbb{R}^d)$ of L\'evy measures with finite second moments.

\begin{corollary}\label{cor:main4}
	Let $\mathcal{W}_\Lambda := (\mathcal{C}_2^\Lambda)^{1/2}$, where $\mathcal{C}_2^\Lambda(\mu, \nu)$, $\mu, \nu \in \Lambda_2(\mathbb{R}^d)$, is the L\'evy optimal transport cost defined in \eqref{levy-cost}.  
	Then $\mathcal{W}_\Lambda$ is a complete and separable metric on $\Lambda_2(\mathbb{R}^d)$.
\end{corollary}

Finally, we state our last theorem in this subsection, which establishes an $L^2$ maximal-type estimate 
on the pathwise discrepancy between two L\'evy processes under an optimal coupling of their generators.

\begin{theorem}\label{main5}
	Let $\mA,\mB \in \mcg_2^\La(\mbr^d)$, and let $\mJ_* \in \Ga(\mA,\mB)$ be a 2-optimal L\'evy coupling of $\mA$ and $\mB$. 
	Denote by $\{Z_t=(X_t,Y_t)\}_{t\ge 0}$ the L\'evy process on $\mbr^{2d}$ with generator $\mJ_*$. 
	It holds for all $T \ge 0$:
	\[
	\mathbb{E}^{(0,0)} \!\left[ \sup_{t \in [0,T]} |X_t - Y_t|^2 \right] 
	\;\leq\; 8 \max\{T,\, T^2\} \; \wg(\mA,\mB)^2.
	\]
	If, in addition, both processes have zero mean (i.e.\ $m_\mA = m_\mB = 0$), then the sharper bound holds:
	\[
	\mathbb{E}^{(0,0)} \!\left[ \sup_{t \in [0,T]} |X_t - Y_t|^2 \right] 
	\;\leq\; 4T \; \wg(\mA,\mB)^2.
	\]
\end{theorem}

We remark that the metrics $\mathcal{W}_{\mathcal{G}}$ and $\mathcal{W}_\Lambda$ were recently introduced in \cite{LimTeoh2025}, while the latter already appeared earlier in \cite{Kolokoltsov2010LevyKhintchine}. In both works, these metrics were used to establish important results, such as propagation of chaos and the generation of L\'evy-type generators, under Lipschitz continuity assumptions w.r.t. these metrics. Although it was suggested in \cite{LimTeoh2025} that $\mathcal{W}_{\mathcal{G}}$ and $\mathcal{W}_\Lambda$ indeed define valid metrics, a formal proof was not given. Theorem~\ref{main3} completes this picture by rigorously verifying their metric properties.

We emphasize that Theorem~\ref{main3} and Corollary~\ref{cor:main4} provide the first rigorous construction of a transport-type metric on the space of L\'evy generators $\mathcal{G}_2^\Lambda(\mathbb{R}^d)$ and the space of L\'evy measures $\La_2(\mbr^d)$. 
While, as mentioned earlier, Wasserstein-type metrics on L\'evy measures have appeared earlier in the literature, our work further establishes their fundamental metric properties. 
This advancement opens new avenues for quantitative analysis in stochastic processes, including stability estimates, propagation of chaos for L\'evy-type mean-field systems, and potential applications in optimal transport and probabilistic numerics.

Further structural and topological properties of the space \( \mcg_2^\Lambda(\mathbb{R}^d) \) under the metric \( \mcW_\mcg \) and the space $\La_2(\mbr^d)$ under the metric $\mcW_\La$, including compactness criteria and characterization of convergence—will be developed in Section~\ref{sec-wass-metric}.

\subsection{Discussion, methods of proof, and organization}\label{sec2.9}
Let us highlight our contributions in this work. As mentioned earlier, the optimal Markovian coupling problem for diffusion processes, including the drift–diffusion generator considered in Theorem~\ref{main1}, was established by \cite{ChenLi1989}. In fact, as we will see later, in this case the optimal coupling can be explicitly constructed using the results of \cite{givens1984class} on the optimal coupling between Gaussian measures. On the other hand, the case of bounded jump operators has been settled by Zhang \cite{Zhang2000}. Consequently, our main contribution lies in addressing the jump component for generators with \emph{unbounded} Lévy measures, where previous techniques do not directly apply.

Let us next discuss the difference in our proof strategy compared to Zhang~\cite{Zhang2000}. There, bounded jump operators are considered, so any coupling $\mathcal{J}$ is also bounded and can be represented via a jump kernel $\lambda(x,y,dx',dy')$:
\begin{align*}
	(\mathcal{J} \Phi)(x,y) &= \int_{\Pi^2} [\Phi(x',y')-\Phi(x,y)] \, d\lambda(x,y,dx',dy').
\end{align*}
Zhang constructed the jump kernel corresponding to an optimal Markovian coupling by taking the pointwise (in $(x,y) \in \Pi$) limit along a subsequence as $t \searrow 0$:
\begin{align*}
	\lambda(x,y,dx',dy') &= \lim_{t \searrow 0} \frac{1}{t} \big[ \gamma_t - \delta_{(x,y)} \big],
\end{align*}
where $\gamma_t$ is a $c$-optimal coupling between the laws $\delta_x e^{t\mathcal{A}}$ and $\delta_y e^{t\mathcal{B}}$.  
The existence of this limit along a subsequence follows from the compactness of measures. A key technical challenge is establishing the measurability of the kernel $\lambda(x,y)$ with respect to $(x,y)$, which is achieved using a measurable selection theorem.

Our construction strategy is to tackle the optimal Markovian coupling problem for L\'evy generators by considering each component—drift, diffusion, and jump—separately, and then combining the results using the additivity property of the problem and the strong duality (to be established later).  
The main challenge, of course, lies in constructing the optimal coupling for the jump components. While the pointwise limit approach used by Zhang~\cite{Zhang2000} could potentially be adapted to this setting, we adopt a different strategy.  
Specifically, we construct an optimal coupling by solving an \emph{optimal transport problem between L\'evy measures}, analogous to the classical optimal transport problem for probability measures.  
Given two L\'evy measures $\mu, \nu$ on $\mathbb{R}^d$, we seek a coupling (in an appropriate sense) $\gamma$ of $\mu$ and $\nu$, which is itself a L\'evy measure on $\mathbb{R}^{2d}$, that minimizes the transport cost between these two measures.  
To the best of our knowledge, this is the first time an optimal transport problem has been studied systematically for L\'evy measures.  
We refer to this as the \emph{L\'evy optimal transport problem}. Once a minimizer is identified, it directly yields the optimal Markovian jump operator for the pair of processes.

Let us remark on the connection between the optimal Markovian coupling problem and the optimality gap problem. At first sight these may appear to be independent questions, but in fact they are closely intertwined. Specifically, once the optimal coupling generators have been constructed for each component --- drift, diffusion, and jump --- it is in fact straightforward to verify that their superposition yields an optimal coupling, though only within the class of \emph{L\'evy} (or even L\'evy-type) coupling generators. To extend this optimality to the larger class of all Feller coupling generators, or even Markovian coupling semigroups, one needs to appeal to the duality: the dual problem provides the verification principle that confirms the superposition is indeed minimal.

\subsubsection*{Future directions}
Let us also outline some possible future directions of this research: 
\begin{enumerate}
	\item \textbf{General \(p \in [1, \infty)\).}  
	The optimal Markovian coupling problem for L\'evy generators can be formulated for the general \(p\)-cost
	\(
	c_p(x,y) = \frac{1}{p} |x-y|^p.
	\)  
	In this work, we focus on the case \(p = 2\), primarily because of the self-duality of \(\frac{1}{2}|x|^2\). This property simplifies certain arguments (e.g., Proposition \ref{0isenough}), since the Legendre transform of \(\frac{1}{2}|x|^2\) is itself. For general \(p\), this convenience is lost, and the problem may become substantially more challenging. For instance, an optimal Markovian coupling between $\mA,\mB$, if exists, may not be translation-invariant (i.e., a L\'evy generator).
	
	\item \textbf{General L\'evy-type generators.}  
	One may pose the optimal coupling problem for a pair of L\'evy-type generators, i.e., operators of the form \eqref{LK-decomp} where the triplet \((\kappa, \alpha, \mu)\) depends on \(x \in \mathbb{R}^d\).  
	A L\'evy-type generator with bounded second moment can be viewed as a family of L\'evy generators \(\{\mathcal{A}(x)\}_x\) with \(\mathcal{A}(x) \in \mathcal{G}_2^\Lambda(\mathbb{R}^d)\). A natural candidate for an optimal coupling is then \(\{\mathcal{J}_*(x,y)\}_{x,y}\), where \(\mathcal{J}_*(x,y)\) is an optimal coupling of \(\mathcal{A}(x)\) and \(\mathcal{B}(y)\) guaranteed by Theorem \ref{main1}.  
	However, several challenges remain: the optimal coupling may not be unique, which can create measurability issues, and the family \(\{\mathcal{J}_*(x,y)\}_{x,y}\) may not generate a Feller semigroup. These problems warrant further investigation.
	
	\item \textbf{Applications of the metric \(\mathcal{W}_\mathcal{G}\).}  
	Given the Wasserstein generator metric \(\mathcal{W}_\mathcal{G}\) defined in \eqref{def:Wg}, one can explore a variety of potential applications, as illustrated in \cite{Kolokoltsov2010LevyKhintchine,LimTeoh2025}.  
	For instance, consider a L\'evy-type generator \(\{\mathcal{A}(x)\}_x\) that is Lipschitz with respect to \(\mathcal{W}_\mathcal{G}\), meaning that
	\[
	\mathcal{W}_\mathcal{G}(\mathcal{A}(x), \mathcal{A}(y)) \le C |x-y| \quad \text{for some constant } C \ge 0.
	\]  
	A natural question arising from this observation is whether such a Lipschitz condition guarantees that \(\mathcal{A}\) indeed generates a Feller semigroup.
	
	\item \textbf{Optimal Markovian coupling problem for generalized couplings}. 
	In the present work, the optimal coupling is sought within the class of \emph{Feller generators/ processes}. 
	In more general settings, however, the Feller property may be too restrictive. 
	It is therefore natural to enlarge the framework and search for optimal couplings within broader notions, such as general Markov semigroups, Markov processes, or solutions to martingale problems. 
	This calls for a suitably generalized formulation of the coupling problem; see Remark~\ref{rem:2.16}.
	
\end{enumerate}

\subsubsection*{Organization}
The remainder of this article is organized as follows. Section~\ref{sec-lotp} introduces and solves the L\'evy optimal transport problem between L\'evy measures. The development closely parallels the classical theory of optimal transport, including the existence of minimizers (optimal couplings), fundamental structural properties, and a Kantorovich-type duality. Section~\ref{sec-opt-markov} addresses the optimal Markovian coupling problem and the associated optimality gap problem for each component of a L\'evy generator --- drift, diffusion, and jump --- as an intermediate step. Building on this, Section~\ref{sec-dual-gap} establishes the additivity property, which allows us to combine the componentwise results and thereby prove the main theorems stated in Section~\ref{sec2.7} (Theorems~\ref{main1}, \ref{main2}, and \ref{main-LK}). Section~\ref{sec-wass-metric} is devoted to the study of the Wasserstein-type metrics $\mathcal{W}_\mathcal{G},\mcW_\La$, where we establish Theorem~\ref{main3}, \ref{main5}, Corollary \ref{cor:main4}, and further investigate their topological properties. Finally, the appendix collects the proofs of several technical lemmas used throughout Sections~\ref{sec-lotp}--\ref{sec-wass-metric}.

%% file: S3.tex
\section{L\'evy Optimal Transport Problem}\label{sec-lotp}

In this section, we formulate and solve a variation of the classical optimal transport problem between two L\'evy measures \(\mu\) and \(\nu\), including cases where \(\mu\) and \(\nu\) have different, or even infinite, total mass. We refer to this formulation as the \emph{L\'evy optimal transport problem}. 
As will be seen in the next section, this problem is closely related to the optimal coupling problem of pure jump L\'evy processes.

We begin this section with a brief overview of the classical optimal transport problem on \(\mathbb{R}^d\) with quadratic cost. The L\'evy optimal transport problem is then introduced in the subsequent subsection, along with motivating examples and preliminary results. In the final two subsections, we establish the existence of minimizers for the L\'evy transport problem and develop its dual formulation, drawing parallels with the classical theory.

In the proof, to simplify notation for computations we set $c_2(x,y) := \frac12 |x-y|^2$ for the squared cost, and use the natural pairing
\[
\langle \mu, f \rangle := \int f \, d\mu,
\]
where the integral is taken over the space ($\mathbb{R}^d$ or $\mathbb{R}^{2d}$) on which the (L\'evy) measure $\mu$ is defined. This notation can also be extended to certain unbounded functions $f$ provided that the integral is well-defined.

\subsection{Classical optimal transport problem}
Let us begin our discussion with the \emph{classical optimal transport problem} on $\mbr^{d}$ with respect to the squared cost. 

\begin{definition}[Coupling of measures and 2-optimality]\label{coupling_optimality_definition}
    Let $\mu,\nu\in \mcl{M}(\mbr^d)$ be a pair of bounded measures with equal mass: $\mu(\mbr^d)=\nu(\mbr^d)$. 
    
    (i) A \emph{coupling measure} is a (bounded) measure $\ga\in\mcl{M}(\mbr^{2d})$ such that it holds for all measurable $E\subset \mbr^{d}$
    \begin{align*}
        \gamma(E\times \mbr^d)=\mu(E),\qquad \gamma(\mbr^d\times E)=\nu(E). 
    \end{align*}
    Denote $\Ga(\mu,\nu)$ \emph{the family of all coupling}. 

    (ii) The \emph{2-optimal transport cost between $\mu,\nu$} is defined by
    \begin{align*}
        \mcc_2(\mu,\nu):=\inf_{\ga\in\Ga(\mu,\nu)} \inn{\ga,c_2}=\inf_{\ga\in\Ga(\mu,\nu)} \int_{\mbr^{2d}}\frac 12 |x-y|^2 d\gamma(x,y).
    \end{align*}
    
    (iii) A \emph{2-optimal coupling} of the pair $\mu,\nu$ is a coupling measure $\ga_*\in\Ga(\mu,\nu)$ such that 
    \begin{align*}
        \inn{\ga_*,c_2}=\int_{\mbr^{2d}} \frac 12 |x-y|^2 d\ga_*(x,y)&= \mcc_2(\mu,\nu). 
    \end{align*}

    (iv) A bounded measure $\ga\in \mcl{M}(\mbr^{2d})$ is \emph{2-optimal} if it is it is 2-optimal w.r.t. its marginal.
\end{definition}

\begin{remark}
	In the literature, (e.g., \cite{MR3050280,MR2459454}), the optimal transport problem is typically studied in the setting of \emph{probability measures}, that is, when \( \mu(\mathbb{R}^d) = \nu(\mathbb{R}^d) = 1 \). The problem and many of the associated results to be presented later, e.g., existence of minimizer, duality, can be naturally extended to the more general setting of bounded measures with equal mass, i.e., \( \mu(\mathbb{R}^d) = \nu(\mathbb{R}^d) < \infty \).
\end{remark}

In the classical theory of optimal transport, the fundamental theorem of optimal transport establishes a deep connection between optimal coupling measures, cyclical monotonicity, and Kantorovich duality. The following definition and theorem formalize this relationship. While they can be stated in the general setting of a Polish space \( \Pi \) and cost $c$, we restrict our attention to the case \( \Pi = \mathbb{R}^d \) and $c(x,y)=c_2(x,y)=\f 12 |x-y|^2$. 

\begin{definition}[$2$-cyclical monotonicity]
	A measurable set $E\subset \mbr^{d}\times \mbr^{d}$ is \emph{$2$-cyclically monotone} if for any finite sequences $\{(x_k,y_k)\}_{k=1}^N\subset E$, and any permutations $\si$ on $\{1,2,\cdots, N\}$, it holds
	\begin{align*}
		\sum_{k=1}^N \f 12  |x_k-y_k|^2 &\le \sum_{k=1}^N \frac 12 |x_k-y_{\si(k)}|^2.
	\end{align*}
\end{definition}

\begin{theorem}[The fundamental theorem of optimal transport \cite{MR3050280}]\label{FTOT}
	Let $E\subset \mbr^{2d}$ be measurable. The following are equivalent:
	
	(i) $E$ is $2$-cyclically monotone;
	
	(ii) there are lower semicontinuous functions $\varphi,\psi:\mbr^d\to[-\infty,\infty)$ with $\varphi\oplus\psi\le c_2$ such that
	\begin{align*}
		\varphi(x)+\psi(y)= c_2(x,y),\qquad \mbox{ for all }(x,y)\in E;
	\end{align*}
	
	(iii) for any bounded measures $\ga$ with $\mathrm{supp}(\ga)\subset E $, $\ga$ is classically 2-optimal. 
\end{theorem}

\subsection{L\'evy optimal transport problem}

Let \( \mu, \nu \in \La(\mathbb{R}^d) \) be two L\'evy measures. In this work, we are interested in a \emph{coupling} L\'evy measure \( \gamma \in \La(\mathbb{R}^{2d}) \) in the following sense: if \( \{X_t\} \) and \( \{Y_t\} \) are pure jump L\'evy processes on \( \mathbb{R}^d \) with L\'evy measures \( \mu \) and \( \nu \), respectively, and \( \{Z_t\} \) is a pure jump L\'evy process on \( \mathbb{R}^{2d} \) with L\'evy measure \( \gamma \), then \( \{Z_t\} \) is a \emph{Feller coupling process} of \( \{X_t\} \) and \( \{Y_t\} \). This motivates the following definition.

\begin{definition}[Coupling of L\'evy measures]\label{Levy_Coupling}
    Let $\mu,\nu\in\La(\mbr^d)$ be two Lévy measures on $\mathbb{R}^d$. A \emph{Lévy coupling} between $\mu,\nu$ is a Lévy measure $\gamma$ on $\mathbb{R}^{2d}$ such that for all measurable sets $E\subset\mathbb{R}^d$ that \emph{does not contain a neighbourhood of $0$}, it holds
    \begin{equation}\label{levy-marg-cond}
        \gamma(E\times\mathbb{R}^d)=\mu(E), \qquad \gamma(\mathbb{R}^d\times E)=\nu(E).
    \end{equation}
    We denote the family of all Lévy couplings of $\mu,\nu$ as $\Gamma^\Lambda(\mu,\nu)$
\end{definition}

\begin{remark}
	Unlike the classical coupling of probability measures, the marginal condition \eqref{levy-marg-cond} for Lévy couplings is only required to hold on measurable sets that \emph{do not} contain a neighborhood of the origin. At first glance, this restriction may seem unusual, but it is the natural choice in the context of Lévy measures, see Proposition \ref{coup-equiv-prop} below.
\end{remark}

\begin{remark}\label{rem:ganonempty}
    For \(\mu, \nu \in \La(\mathbb{R}^d)\), the set \(\Gamma^\La(\mu, \nu)\) is nonempty. For instance, let \(\gamma \in \La(\mathbb{R}^{2d})\) be defined by
    \begin{align*}
        \gamma = \mu \otimes \delta_0 + \delta_0 \otimes \nu,
    \end{align*}
    where \(\otimes\) denotes the product (tensor) of measures. In particular, \(\gamma\) is a Lévy measure supported on \(\mathbb{R}^d \times \{0\} \cup \{0\} \times \mathbb{R}^d\). One can readily verify that \(\gamma\) is a Lévy coupling of \(\mu\) and \(\nu\). 
    Indeed, if \(\mu\) and \(\nu\) are the Lévy measures associated with processes \(\{X_t\}\) and \(\{Y_t\}\), respectively, then \(\gamma\) corresponds to the Lévy measure of the process \(\{Z_t\} = \{(X_t, Y_t)\}\), assuming that \(\{X_t\}\) and \(\{Y_t\}\) are independent.
\end{remark}

In the coming proposition and lemma, let us provide some equivalent conditions for L\'evy couplings of $\mu,\nu$.

\begin{proposition}\label{equivalentLevyCoupling}
    Let $\mu,\nu\in \La(\mbr^d)$ and $\ga\in \La(\mbr^{2d})$. The following are equivalent:

    (i) $\ga\in \Ga^\La(\mu,\nu)$;

    (ii) for all bounded measurable functions $f$ that vanish on a neighbourhood of $0$, we have
    \begin{equation}\label{marg-int}
        \int_{\mbr^{2d}} f(x)d\gamma(x,y)= \int_{\mbr^d} f(x)d\mu(x),\qquad \int_{\mbr^{2d}} f(y)d\gamma(x,y)=\int_{\mbr^d} f(y)d\nu(x);
    \end{equation}
    
    (iii) for all bounded twice differentiable functions $f\in C_b^2(\mathbb{R}^d)$ that vanish on a neighbourhood of $0$, \eqref{marg-int} holds;

    (iv) for all bounded continuous functions $f\in C_b(\mathbb{R}^d)$ that vanish on a neighbourhood of $0$, \eqref{marg-int} holds;

    (v) for all bounded continuous functions $f\in C_b(\mbr^{d})$ such that $|f(x)|\le C|x|^2$ for some $C\ge 0$, \eqref{marg-int} holds.
\end{proposition}

\begin{remark}
	It is instructive to contrast this with the classical setting of probability measure couplings. In that case, the equivalence stated in Proposition \ref{equivalentLevyCoupling} holds without the restriction that the functions vanish on a neighborhood of the origin.
\end{remark}

\begin{remark}
	Using the natural pairing notation, the marginal integral condition \eqref{marg-int} can be rewritten in the more compact form:
	\begin{align}\label{marg-pairing}
		\inn{\ga, f\otimes 1}=\inn{\mu,f},\qquad \inn{\ga,1\otimes f}=\inn{\nu,f}. 
	\end{align}
\end{remark}

\begin{proof}
(i) $\to$ (ii). It is straightforward from Definition \ref{Levy_Coupling} that \eqref{marg-int} holds for all characteristic functions $f$ of sets that does not contain a neighborhood of $0$. One may then approximate any bounded measurable functions vanishing around 0 as a sequence of simple functions. The dominated convergence theorem implies \eqref{marg-int} holds for all such functions.

(ii) $\to$ (iii). Trivial.

(iii) $\to$ (iv). Let \( f \in C_b(\mathbb{R}^d) \) vanish on \( B_r(0) \) for some \( r>0 \). Then there exists a sequence \( \{f_n\}_{n} \subset C_b^2(\mathbb{R}^d) \), each vanishing on a smaller ball (e.g., \( B_{r/2}(0) \)), such that \( f_n \to f \) locally uniformly. Since \eqref{marg-int} holds for all \( f_n \), letting \( n \to \infty \) and applying the bounded convergence theorem yields the desired equality for \( f \).

(iv) $\to$ (v). 
Given a bounded continuous functions $f\in C_b(\mbr^{d})$ such that $|f(x)|\le C|x|^2$ for some $C\ge 0$, we may define a sequence of bounded continuous functions $\{f_n\}$ such that $f_n$ vanishes in $B_{1/n}(0)$, the sequence converges pointwise to $f$, and $|f_n|\le C|x|^2$.
Since \eqref{marg-int} holds for all $f_n$ and $\min\cb{1,|x|^2}\in L^1(\mu),L^1(\nu)$, invoking the dominated convergence theorem, \eqref{marg-int} holds for $f$ as well.  

(v) $\to$ (i). Given an open set $E\subset \mbr^d\setminus B_\ep(0)$ for some $\ep>0$, by Urysohn lemma we may find a sequence of continuous functions $\{f_n\}$ vanishing on $B_{\ep/2}(0)$ such that $f_n\searrow \chi_E$ pointwise. We note that for each $n$, $f_n$ satisfies Condition (iii) and thus \eqref{marg-int} holds. Applying the dominated convergence theorem, passing $n\to\infty$, \eqref{marg-int} holds for $f=\chi_E$, that is, \eqref{levy-marg-cond} holds for all open sets that does not contain a neighborhood of origin. Applying a standard argument involving $\pi$-$\lambda$ theorem, \eqref{levy-marg-cond} extends to all measurable sets $E$ that does not contain a neighborhood of origin. This proves (i). 
\end{proof}

Let us record the following simple lemma, which will be useful later.

\begin{lemma}\label{FTOLOT3to1lemma}
    Suppose $\gamma \in \Gamma^{\Lambda}(\mu, \nu)$ and let $f:\mathbb{R}^d \to \mathbb{R}$ be measurable such that $f^+ \in L^1(\mbr^d,d\mu)$. Then $f^+ \otimes 1 \in L^1(\mbr^{2d},d\gamma)$ and
    \begin{align}\label{marg2}
        \int_{\mathbb{R}^{2d}} f(x)\,d\gamma(x,y) = \int_{\mathbb{R}^d} f(x)\,d\mu(x).
    \end{align}
    The integrals above are understood to take values in $[-\infty, \infty)$.
\end{lemma}

\begin{proof}
	For $\ep>0$, define the cutoff functions $f_\ep := f\chi_{B_{\ep}^c(0)}$. We then have
    \[
        f_\ep^\pm=\max\cb{\pm f,0} = f^\pm\chi_{B_{\ep}^c(0)},
    \]
    and the sequences $\{f_\ep^\pm\}$ increase pointwise to $f^\pm$ as $\ep \nearrow 0$.

    Since $\gamma \in \Gamma^{\Lambda}(\mu, \nu)$, by Proposition \ref{equivalentLevyCoupling}, \eqref{marg2} holds with $f^\pm_\ep$ in place of $f$.
    By the monotone convergence theorem, \eqref{marg2} also holds with $f^\pm$ in place of $f$.
    In particular, since $f^+ \in L^1(\mu)$, the integral of $f^+$ with respect to $\gamma$ is finite, i.e., $f^+ \otimes 1 \in L^1(\gamma)$.
    Finally, since $f^+ - f^- = f$, \eqref{marg2} follows by subtracting the integrals of positive and negative parts.
\end{proof}

As mentioned earlier, the motivation behind the above definition is to relate a coupling of pure jump Lévy processes with their corresponding Lévy measures. The proposition below establishes this connection. 


\begin{proposition}\label{coup-equiv-prop}
	Let $\mu,\nu\in \La_2(\mbr^d)$, $\ga\in \La_2(\mbr^{2d})$,
	$\mA=\Lag(0,0,\mu),\mB=\Lag(0,0,\nu)$ and $\mJ=\Lag (0,0,\ga)$. Then $\mJ\in \Ga(\mA,\mB)$ if and only if $\ga\in \Ga^\La(\mu,\nu)$. 
\end{proposition}

\emph{Remark.} The assumption that \( \mathcal{A}, \mathcal{B} \) are given in the \emph{global form} is essential for this proposition. In particular, the identification \( \mathcal{J} = \Lag(0, 0, \gamma) \) relies on the compensated jump operator being expressed without a cutoff function. In the localized form (e.g., with a cutoff such as \( \chi_{B_1(0)} \)), the generator depends on the choice of cutoff, and the marginal condition would no longer correspond directly to \( \gamma \in \Gamma^\Lambda(\mu, \nu) \).

\begin{proof}
	Before the proof, let us introduce some notations to simplify the computation. 
	Fix $f\in C^2(\mathbb{R}^d)$ and $x_0\in \mathbb{R}^d$. Write the (second-order) Taylor remainder in the direction $x$ as
	\[
	\mathcal{R}_{x_0} f(x)\;:=\; f(x_0+x)-f(x_0)-x\cdot \nabla f(x_0).
	\]
	For $F\in C^2(\mathbb{R}^{2d})$ and $(x_0,y_0)\in \mathbb{R}^{2d}$ define analogously
	\[
	\mathcal{R}_{(x_0,y_0)} F(x,y)\;:=\; F(x_0+x,y_0+y)-F(x_0,y_0)
	-(x,y)^\top \cdot \nabla F(x_0,y_0).
	\]
	With this notation, the jump-type L\'evy generators act as
	\begin{align}\label{id:temp1}
		\mathcal{A} f(x_0)=\int_{\mathbb{R}^d} \mathcal{R}_{x_0} f(x)\, \mu(dx)=\langle \mu, \mathcal{R}_{x_0} f\rangle,\quad
		\mathcal{B} f(y_0)=\int_{\mathbb{R}^d} \mathcal{R}_{y_0} f(y)\, \nu(dy)=\langle \nu, \mathcal{R}_{y_0} f\rangle,	
	\end{align}
	and for $\mathcal{J}$,
	\begin{align}\label{id:temp2}
		\mathcal{J} F(x_0,y_0)=\int_{\mathbb{R}^{2d}} \mathcal{R}_{(x_0,y_0)} F(x,y)\, \gamma(dx,dy)
		=\langle \gamma, \mathcal{R}_{(x_0,y_0)} F\rangle.	
	\end{align}
	The remainder operator tensorizes in the following way:
	\begin{equation}\label{eq:tensor-rem}
		\mathcal{R}_{(x_0,y_0)}[f\otimes 1]\;=\;(\mathcal{R}_{x_0} f)\otimes 1,\qquad
		\mathcal{R}_{(x_0,y_0)}[1\otimes f]\;=\;1\otimes (\mathcal{R}_{y_0} f).
	\end{equation}

	$\to$. 
	Suppose $\mJ \in \Ga(\mA,\mB)$.  
	This means that for any $f \in C_b^2(\mathbb{R}^d)$ we have  
	\[
	\mJ[f \otimes 1] = \mA f \otimes 1, 
	\qquad 
	\mJ[1 \otimes f] = 1 \otimes \mB f.
	\]  
	Fix $f \in C_b^2(\mathbb{R}^d)$ vanishing on a neighbourhood of $0$.  
	Then $f(0) = 0$ and $\nabla f(0) = 0$, hence $\mathcal{R}_0 f = f$ and by \eqref{eq:tensor-rem} $\mcl{R}_{(0,0)}[f\otimes 1]=f\otimes 1$.  Combining these with \eqref{id:temp1}, \eqref{id:temp2}, we find 
	\[
	\langle \gamma, f \otimes 1 \rangle
	= \langle \gamma, \mcl{R}_{(0,0)}[f \otimes 1] \rangle
	= \mJ[f \otimes 1](0,0)
	= \mA f(0)
	= \langle \mu, \mathcal{R}_0 f \rangle
	= \langle \mu, f \rangle.
	\]  
	Similarly,  we have $\langle \gamma, 1 \otimes f \rangle
	= \langle \nu, f \rangle.$
	By Proposition \ref{equivalentLevyCoupling} (see \eqref{marg-pairing}), we conclude that 
	$\gamma \in \Ga^\Lambda(\mu,\nu)$.
	
	$\leftarrow$.  
	Suppose $\gamma \in \Ga^\Lambda(\mu,\nu)$.  
	Let $f \in C_0^2(\Pi)$ and fix $x_0, y_0 \in \mathbb{R}^d$.  
	By Taylor's theorem, the remainder $\mathcal{R}_{x_0} f$ satisfies  
	\(
	|\mathcal{R}_{x_0} f(x)| \le C |x|^2
	\)
	for some $C \ge 0$ depending on the $C^2$-norm of $f$.  
	Invoking Proposition \ref{equivalentLevyCoupling} and using \eqref{id:temp1}--\eqref{eq:tensor-rem}, we obtain
	\[
	\mathcal{J}[f \otimes 1](x_0, y_0)
	= \langle \gamma, \mathcal{R}_{(x_0, y_0)}(f \otimes 1) \rangle
	= \langle \mu, \mathcal{R}_{x_0} f \rangle
	= \mathcal{A} f(x_0).
	\]
	By a similar argument,  
	\(
	\mathcal{J}[1 \otimes f](x_0, y_0)
	= \mathcal{B} f(y_0),
	\)
	and hence by Proposition \ref{prop:coup-char} $\mathcal{J} \in \Ga(\mathcal{A}, \mathcal{B})$. Note we use the fact that $C_0^2(\mbr^d)$ is dense in $D(\mA),D(\mB)$ respectively. 	
\end{proof}

We now formulate the optimality condition for L\'evy couplings with respect to the squared cost, interpreted in the sense of L\'evy transport.

\begin{definition}[2-Optimal L\'evy transport cost and couplings]\label{levy-opt-prob}
	Let $\mu, \nu \in \Lambda_2(\mathbb{R}^d)$ be two L\'evy measures on $\mathbb{R}^d$ with finite second moments. The \emph{2-optimal L\'evy transport cost} is defined by
	\begin{align}\label{levy-cost}
		\mathcal{C}_2^\Lambda(\mu, \nu) := \inf_{\gamma \in \Gamma^\Lambda(\mu, \nu)} \int_{\mathbb{R}^{2d}} \frac{1}{2} |x - y|^2 \, d\gamma(x, y).
	\end{align}
	A L\'evy measure $\gamma_* \in \Lambda(\mathbb{R}^{2d})$ is called a \emph{2-optimal L\'evy coupling} of $\mu$ and $\nu$ if
	\begin{align}\label{LOTP}
		\int_{\mathbb{R}^{2d}} \frac{1}{2} |x - y|^2 \, d\gamma_*(x, y) = \mathcal{C}_2^\Lambda(\mu, \nu).
	\end{align}
\end{definition}

As seen in the definition above, the objective functional \eqref{levy-cost} coincides formally with that of the classical optimal transport problem. However, the key distinction lies in the admissible set of couplings. The classical formulation considers all couplings between probability measures of equal mass, while the L\'evy formulation restricts to \emph{L\'evy couplings} between possibly unequal L\'evy measures. This change in admissibility conditions fundamentally alters the structure and interpretation of the problem. 

This formulation will later be shown to coincide with the optimal coupling problem for L\'evy jump processes under suitable conditions, a connection that will be explored in the next section. In that context, the L\'evy measure associated with the optimal Markovian coupling of two jump processes with L\'evy measures \(\mu\) and \(\nu\) corresponds precisely to a 2-optimal L\'evy coupling.

We remark that several variants of optimal transport, in which the cost functional is optimized under alternative constraints, have also been extensively studied. Notable examples include unbalanced optimal transport \cite{chizat2018unbalanced}, partial optimal transport \cite{Caffarelli2010Partial,DavilaKim2016}, and martingale optimal transport \cite{BackhoffVeraguas2022,Beiglbock2013Martingale}, among others.

\begin{convention}
	Throughout this section, we will use the terms \emph{classical} and \emph{L\'evy} to distinguish between the two parallel frameworks of optimal transport: the classical formulation (Definition~\ref{coupling_optimality_definition}) and the L\'evy-based formulation (Definition~\ref{levy-opt-prob}). Accordingly, we will refer to notions such as the \emph{classical} optimal transport problem, \emph{classical} 2-optimal coupling, and, in contrast, the \emph{L\'evy} optimal transport problem and \emph{L\'evy} 2-optimal coupling. We encourage the reader to be mindful of this distinction, as many definitions and results will appear formally similar but belong to distinct theoretical settings.
\end{convention}

\subsection{Some preliminary results of the L\'evy optimal transport cost}

Before we proceed with solving the L\'evy optimal transport problem, we collect some elementary properties of the transport cost.

\begin{proposition}\label{prop-prelim}
	Let $\mu,\nu,\mu',\nu'\in \La(\mathbb{R}^d)$, and let $\mcc_2^\La$ be the L\'evy transport cost defined in \eqref{levy-cost}. Then the following properties hold:
	
	\begin{enumerate}[label=(\roman*)]
		\item $\mcc_2^\La(\mu,\nu)\in [0,\infty]$. 
		\item $\mcc_2^\La(\mu,\nu)=\mcc_2^\La(\nu,\mu)$.
		\item $\mcc_2^\La(\mu,\mu)=0$.
		\item If $\mu,\nu\in \La_2(\mathbb{R}^d)$, then $\mcc_2^\La(\mu,\nu)<\infty$.
		\item $\mcc_2^\La(\mu,0)=\frac{1}{2}\int_{\mathbb{R}^d} |x|^2 \, d\mu(x)$.
		\item $\mcc_2^\La(\mu+\mu',\nu+\nu') \le \mcc_2^\La(\mu,\nu) + \mcc_2^\La(\mu',\nu')$.
		\item For any $\alpha\ge 0$, $\mcc_2^\La(\alpha\mu,\alpha\nu) = \alpha \, \mcc_2^\La(\mu,\nu)$.
	\end{enumerate}
\end{proposition}

\begin{proof}
	Recall $c_2(x,y)=\f 12|x-y|^2$ and the notation $\inn{\cdot,\cdot}$ of natural pairings between L\'evy measures and functions. 
	
	(i) The nonnegativity $\mcc_2^\La(\mu,\nu) \ge 0$ is immediate from the definition.
	
	(ii) Let $T:(\mathbb{R}^d)^2 \to (\mathbb{R}^d)^2$ be the coordinate swap map. Then $\gamma \in \Ga^\La(\mu,\nu)$ if and only if $T_\sharp \gamma \in \Ga^\La(\nu,\mu)$, which gives symmetry.
	
	(iii) Nonnegativity follows from (i). The diagonally supported measure 
	\(
	\gamma_*(E) := \mu(\{x:(x,x)\in E\})
	\) 
	lies in $\Ga^\La(\mu,\mu)$ and satisfies $\langle \gamma_*, c_2 \rangle = 0$, hence $\mcc_2^\La(\mu,\mu)=0$.
	
	(iv) Using $c_2(x,y)\le |x|^2+|y|^2$, for any $\gamma \in \Ga^\La(\mu,\nu)$ (non-empty by Remark \ref{rem:ganonempty}),
	\[
	\langle \gamma, c_2 \rangle \le \int_{\mathbb{R}^{2d}} (|x|^2 + |y|^2) \, d\gamma(x,y) = \int_{\mathbb{R}^d} |x|^2 \, d\mu(x) + \int_{\mathbb{R}^d} |y|^2 \, d\nu(y) < \infty.
	\]
	This follows $\mcc_2^\La(\mu,\nu)<\infty$. 
	
	(v) There is only one L\'evy coupling between $\mu$ and $0$, given by $\gamma = \mu \otimes \delta_0$, supported on $\mathbb{R}^d \times \{0\}$.
	
	(vi) Fix $\epsilon>0$ and choose $\gamma \in \Ga^\La(\mu,\nu)$, $\gamma' \in \Ga^\La(\mu',\nu')$ such that $\langle \gamma, c_2\rangle \le \mcc_2^\La(\mu,\nu)+\epsilon$, $\langle \gamma', c_2 \rangle \le \mcc_2^\La(\mu',\nu')+\epsilon$. Then $\gamma + \gamma' \in \Ga^\La(\mu+\mu',\nu+\nu')$, and by linearity,
	\[
	\mcc_2^\La(\mu+\mu',\nu+\nu') \le \langle \gamma+\gamma', c_2 \rangle = \langle \gamma, c_2 \rangle + \langle \gamma', c_2 \rangle \le \mcc_2^\La(\mu,\nu) + \mcc_2^\La(\mu',\nu') + 2\epsilon.
	\]
	Letting $\epsilon \to 0$ gives the desired inequality.
	
	(vii) The identity trivially holds for $\alpha=0$. For $\alpha>0$, note that $\gamma \in \Ga^\La(\mu,\nu)$ if and only if $\alpha \gamma \in \Ga^\La(\alpha \mu, \alpha \nu)$, which yields the identity.
\end{proof}

The next result provides a comparison between the classical and L\'evy transport costs, under the assumption that the two L\'evy measures have equal (finite) total mass.

\begin{proposition}\label{prop:prelim2}
	Let $\mu,\nu \in \La(\mathbb{R}^d)$ satisfy $\mu(\mathbb{R}^d) = \nu(\mathbb{R}^d) < \infty$. Then
	\[
	\mcc_2^\La(\mu,\nu) \le \mcc_2(\mu,\nu).
	\]
\end{proposition}

\begin{proof}
	Every classical coupling is also a L\'evy coupling, i.e., $\Gamma(\mu,\nu) \subset \Gamma^{\Lambda}(\mu,\nu)$. The inequality follows immediately from the definitions of $\mcc_2$ and $\mcc_2^\La$.
\end{proof}

An immediate question naturally arises: if two L\'evy measures have equal total mass, does the L\'evy transport cost necessarily coincide with the classical optimal transport cost? Surprisingly, the answer is \emph{no}. The following example illustrates this phenomenon. Later, the fundamental theorem of L\'evy optimal transport (Theorem \ref{FTOLOT}) will provide a necessary and sufficient condition under which the two costs indeed coincide, see Remark \ref{rem:cost-equal}.

\begin{example}
    Consider a probability measure $\mu\in \mcp_2(\mbr^d)$ with finite second moment, and assume the first moment is nonzero:
    \begin{align*}
        m_1(\mu):=\int_{\mathbb{R}^d}x d\mu(x)\neq0,\qquad m_2(\mu):= \int_{\mbr^d} |x|^2 d\mu(x). 
    \end{align*}
    Fix $h\in \mbr^d$ and let $\mu_h$ be the $h$-translation of $\mu$, namely, $\mu_h(E)= \mu(E+h)$ for all measurable $E$.
    It is well-known here that the classical optimal transport cost is given by $\mcc_2(\mu,\mu_h)= \frac 12 |h|^2$. 
    We now show that the L\'evy transport cost $\mcc_2^\La(\mu,\mu_h)$ is strictly less than $\mcc_2(\mu,\mu_h)$ for some $h\in \mbr^d$. 
    
    Let $\gamma$ be a classical 2-optimal coupling of $\mu$ and $\mu_h$, so that $\inn{\ga,c_2}=\mcc_2(\mu,\mu_h)=\f 12 |h|^2$. For $\ta\in (0,1]$ we consider the following L\'evy coupling of $\mu$ and $\mu_h$:
    \begin{align*}
        \tilde{\gamma}=(1-\theta)\gamma+\theta(\mu\otimes\delta_0+\delta_0\otimes\mu_h).
    \end{align*}
    That is, $\tilde\ga$ is a convex combination of a classical, and hence L\'evy coupilng $\ga$ of $\mu,\nu$ and their trivial L\'evy coupling (see Remark \ref{rem:ganonempty}). Hence it is also a L\'evy coupling of $\mu,\nu$. 
    We then compute the following integral:
    \begin{align*}
    	\inn{\tilde \ga,c_2}&= (1-\ta)\inn{\ga,c_2}+ \ta\inn{\mu\otimes \de_0,c_2} +\ta \inn{\de_0\otimes \mu_h,c_2} \\
    	&=(1-\ta)\inn{\ga,c_2}+ \f{\ta}2 \int_{\mbr^d} |x|^2 d\mu(x)+\f{\ta}2 \int_{\mbr^d} |y|^2 d\mu_h(y)\\
    	&=\frac{1-\theta}{2}\lvert h\rvert^2+\theta m_2(\mu)+\theta h\cdot m_1(\mu)+\frac{\ta\lvert h\rvert^2}{2}
    	\\&=\frac{\lvert h\rvert^2}{2}+\theta (m_2(\mu)+ h\cdot m_1(\mu)),
    \end{align*}
    where we use the fact:
    \begin{align*}
        \int_{\mbr^d} |y|^2 d\mu_h(y) = \int_{\mbr^d} |y+h|^2 d\mu(y) = m_2(\mu)+ 2h\cdot m_1(\mu)+ h^2. 
    \end{align*}
    Since $m_1(\mu)\neq 0$, we may choose $h\in\mbr^d$ such that $m_2(\mu)+h\cdot m_1(\mu)<0$. This provides a L\'evy coupling with transport cost $<\f{|h|^2}{2}$. Thus, $\mcc_2^\La(\mu,\mu_h)<\mcc_2(\mu,\mu_h)$. 
\end{example}

\subsection{Construction of optimal L\'evy couplings}
The main theorem of this section is the following, which states the existence of a 2-optimal L\'evy coupling for any pairs of L\'evy measures $\mu,\nu\in\La_2(\mbr^d)$ with finite second moment. 

\begin{theorem}[Existence of optimal L\'evy couplings]
	\label{lotp-thm}
	Given $\mu,\nu\in\La_2(\mbr^d)$, a 2-optimal L\'evy coupling $\ga_*$ exists. Moreover, $\ga_*\in \La_2(\mbr^{2d})$ 
\end{theorem}

\subsubsection{Weak convergence and tightness in the space of L\'evy measures}\label{sec-La-conv}

The main idea behind the construction is to obtain a minimizer as the weak limit of a subsequence of minimizing L\'evy couplings \( \gamma_n \). A technical challenge arises from the fact that the sequence \( \{\gamma_n\}_n \) may lack tightness in the classical sense due to potential singularity near the origin. This issue is well known in the literature and is typically handled by introducing a topology—comprising both convergence and compactness criteria—tailored to the integrability structure of L\'evy measures.

Let us now briefly review the notion of weak convergence and tightness that are suitable for the space of L\'evy measures. A more detailed treatment will be presented in Appendix~\ref{appen-weaktop}.
We say that a sequence of L\'evy measures \( \{ \mu_n \}_n \subset \Lambda(\mathbb{R}^d) \) \emph{converges L\'evy-weakly}, or simply \emph{\( \Lambda \)-weakly}, to \( \mu \in \Lambda(\mathbb{R}^d) \) if
\[
\lim_{n \to \infty}\inn{\mu_n,\varphi}=
\lim_{n \to \infty} \int_{\mathbb{R}^d} \varphi(x) \, d\mu_n(x) = \int_{\mathbb{R}^d} \varphi(x) \, d\mu(x)=\inn{\mu,\varphi}.
\]
for all bounded continuous test functions \( \varphi \in C_b(\mathbb{R}^d) \) satisfying a quadratic bound for some constant \( C \ge 0 \):
\[
|\varphi(x)| \le C \min\{1, |x|^2\}.
\]
Next, a family \( \mathcal{M} \subset \Lambda(\mathbb{R}^d) \) of L\'evy measures is said to be \emph{L\'evy-tight} (or simply \emph{\( \Lambda \)-tight}) if it satisfies the following two conditions:
\begin{itemize}
	\item[(i)] Uniform moment bound:
	\[
	\sup_{\mu \in \mathcal{M}} \int_{\mathbb{R}^d} \min\{1, |x|^2\} \, d\mu(x) < \infty;
	\]
	\item[(ii)] Tightness away from the origin: for every \( \varepsilon > 0 \), there exists a compact set \( K \subset \mathbb{R}^d \setminus \{0\} \) such that
	\[
	\sup_{\mu \in \mathcal{M}} \int_{K^c} \min\{1, |x|^2\} \, d\mu(x) < \varepsilon.
	\]
\end{itemize}

\medskip

In the construction of minimizers for the L\'evy optimal transport problem, the following two foundational results concerning \( \Lambda \)-weak convergence and \( \Lambda \)-tightness are essential:
\begin{itemize}
	\item \emph{Prokhorov-type compactness:} Every \( \Lambda \)-tight sequence of L\'evy measures admits a \( \Lambda \)-weakly convergent subsequence.
	\item \emph{Portmanteau-type theorem:} Let \( \mu_n \to \mu \) \( \Lambda \)-weakly. Then for every function \( \varphi \in C_b(\mathbb{R}^d) \) satisfying \( |\varphi(x)| \le C \min\{1, |x|^2\} \), it holds that
	\[
	\liminf_{n \to \infty} \inn{\mu_n,\varphi}= \liminf_{n \to \infty} \int \varphi \, d\mu_n \ge \int \varphi \, d\mu = \inn{\mu,\varphi},
	\]
	provided \( \varphi \) is lower semicontinuous; and similarly, the \( \limsup \) inequality holds for upper semicontinuous functions.
\end{itemize}
Though these results are standard and often used implicitly in the literature, we include detailed proofs in Appendix~\ref{appen-weaktop} for completeness.

\subsubsection{Construction of minimizers}

One of the key steps in the construction of a minimizer to the optimal transport problem is to establish the tightness of the family of couplings. Let us prove this in a slightly more general setting, whose proof is similar to that of the classical counterpart.

\begin{lemma}\label{lem:couplingtight}
	Let $\mcl{M},\mcl{N}\subset \La(\mbr^d)$ be two $\La$-tight families, and define
	\begin{align*}
		\Ga^\La(\mcl{M},\mcl{N}) = \cb{\ga\in\La(\mbr^{2d}) : \ga \in\Ga^\La(\mu,\nu),\ \mu\in\mcl{M},\ \nu\in\mcl{N} }.
	\end{align*}
	That is, $\Ga^\La(\mcl{M},\mcl{N})$ is the family of all L\'evy couplings $\ga$ of $\mu\in\mcl{M},\ \nu\in\mcl{N}$. Then $\Ga(\mcl{M},\mcl{N})$ is $\La$-tight in $\La(\mbr^{2d})$.
\end{lemma}

\begin{proof}
	Let us abbreviate $\Ga = \Ga^\La(\mcl{M},\mcl{N})$ and introduce the minimum operator $a\wedge b = \min\cb{a,b}$. We aim to show: (1) it holds:
	\begin{align*}
		\sup_{\ga\in\Ga} \int_{\mbr^{2d}}[1\wedge(|x|^2+|y|^2)] \, d\ga(x,y) < \infty;
	\end{align*}
	(2) for every $\ep>0$, there exists a compact set $K \subset \mbr^{2d}\setminus \{(0,0)\}$ such that
	\begin{align*}
		\sup_{\ga\in\Ga} \int_{K^c} [1\wedge(|x|^2+|y|^2)] \, d\ga(x,y) < \ep.
	\end{align*}
		
	To prove (1), observe the bound
	\(
	[1\wedge(|x|^2+|y|^2)]\le (1\wedge|x|^2) + (1\wedge |y|^2).
	\)
	Using the marginal condition, for any $\ga \in \Ga^\La(\mu,\nu)$ with $\mu \in \mcl{M},\ \nu \in \mcl{N}$, we have
	\begin{align*}
		\int_{\mbr^{2d}} 1\wedge(|x|^2+|y|^2) \, d\ga(x,y)
		&\le \int_{\mbr^d} 1\wedge |x|^2 \, d\mu(x) + \int_{\mbr^d} 1\wedge |y|^2 \, d\nu(y).
	\end{align*}
	Since $\mcl{M}$ and $\mcl{N}$ are $\La$-tight, the right-hand side is uniformly bounded. This proves (1).
	
	To prove (2), by $\La$-tightness of $\mcl{M}$ and $\mcl{N}$, there exists a compact set $\tilde K \subset \mbr^d \setminus \{0\}$ such that
	\[
	\sup_{\mu \in \mcl{M} \cup \mcl{N}} \int_{\tilde K^c} 1\wedge |x|^2 \, d\mu(x) < \f{\ep}{2}.
	\]
	Let $K = \tilde K \times \tilde K$, which is a compact subset of $\mbr^{2d} \setminus \{(0,0)\}$. Then for any $\ga \in \Ga^\La(\mu,\nu)$ with $\mu \in \mcl{M},\ \nu \in \mcl{N}$,
	\begin{align*}
		\int_{K^c} 1\wedge(|x|^2+|y|^2) \, d\ga(x,y)
		&\le \int_{(\tilde K^c \times \mbr^d) \cup (\mbr^d \times \tilde K^c)} [(1\wedge |x|^2)+(1\wedge |y|^2)] d\ga(x,y) \\
		&\le \int_{\tilde K^c} 1\wedge |x|^2 \, d\mu(x) + \int_{\tilde K^c} 1\wedge |y|^2\, d\nu(y) \\
		&< \ep.
	\end{align*}
	This proves (2), hence $\Ga$ is $\La$-tight.
\end{proof}

We may now proceed to the proof of Theorem \ref{lotp-thm}

\begin{proof}[Proof of Theorem \ref{lotp-thm}]
	We first observe that for any $\mu, \nu \in \Lambda_2(\mathbb{R}^d)$, the set of L\'evy couplings $\Gamma^\Lambda(\mu, \nu)$ is nonempty, see Remark \ref{rem:ganonempty}. 
	Hence, the L\'evy cost $\mathcal{C}_2^\Lambda(\mu, \nu)$ is well-defined and finite.
	
	Let $\{\gamma_n\} \subset \Gamma^\Lambda(\mu, \nu)$ be a minimizing sequence, i.e.,
	\[
	\lim_{n\to\infty} \inn{\ga_n,c_2}
	= \mathcal{C}_2^\Lambda(\mu, \nu).
	\]
	By Lemma \ref{lem:couplingtight}, $\Ga^\La(\mu,\nu)$ is $\La$-tight. Hence by the Prokhorov theorem (see Theorem \ref{thm:prokhorov}), there is $\ga_*\in \La(\mbr^{2d})$ and a subsequence $\{\ga_{n_k}\}_{k}$ such that $\ga_{n_k}\to\ga_*$ $\La$-weakly. 

	Let us first show that $\ga\in \La_2(\mbr^d)$. Indeed, it follows by Portmanteau lemma (Lemma \ref{lem:port}):
	\begin{align*}
		\inn{\ga_*,|x|^2+|y|^2}\le \liminf_{n\to\infty} \inn{\ga_n,|x|^2+|y|^2} = \inn{\mu,|x|^2}+ \inn{\nu,|y|^2}<\infty.
	\end{align*}
	Hence $\ga_*$ has a finite second moment.
	
	Next we show $\ga_*\in \Ga^{\La}(\mu,\nu)$. 	
    For any $f\in C_b(\mathbb{R}^d)$ that vanishes on some neighbourhood of $0$, clearly $|f(x)|\le C\min\cb{1,|x|^2}$ for some constant $C\ge 0$. Since $\ga_n\to\ga$ $\La$-weakly we have
    \begin{align*}
    	\inn{\ga_*,f\otimes 1} = \lim_{n\to\infty} \inn{\ga_n,f\otimes 1} = \inn{\mu,f}.
    \end{align*}
    A similar computation also leads to $\inn{\ga_*,1\otimes f}=\inn{\nu,f}$. 
    The above holds for all $f\in C_b(\mbr^d)$ that vanishes on a neighborhood of 0. 
    By Proposition \ref{equivalentLevyCoupling} (see \eqref{marg-pairing}) $\ga_*\in\Ga^\La(\mu,\nu)$.

    Finally let us prove $\ga_*$ is L\'evy 2-optimal. Indeed, by Portmanteau lemma we find 
    \begin{align*}
    	\inn{\ga_*,c_2}&\le \liminf_{n\to\infty} \inn{\ga_n,c_2}=\mcc_2^\La(\mu,\nu). 
    \end{align*}
    This shows $\ga_*$ is 2-optimal. 
    The construction of 2-optimal coupling is completed.
\end{proof}

\subsection{The fundamental theorem of Lévy optimal transport}






In the L\'evy optimal transport problem, a similar connection between optimal couplings, cyclical monotonicity, and duality also holds. We present the corresponding result below, which we refer to as the \emph{fundamental theorem of L\'evy optimal transport}.

\begin{theorem}[The fundamental theorem of L\'evy optimal transport] \label{FTOLOT}
	Let $\mu,\nu\in \La_2(\mbr^d)$ and $\ga\in \La_2(\mbr^{2d})$. The following are equivalent. 
	
	(i) $\ga$ is a 2-optimal L\'evy coupling of $\mu,\nu$.
	
	(ii) The set $\supp(\ga)\cup\{(0,0)\}$ is $2$-cyclically monotone.
	
	(iii) There exists a pair of lower semicontinuous functions $\varphi,\psi:\mbr^d\to[-\infty,\infty)$ such that $\varphi\in L^1(\mu),\psi\in L^1(\nu)$, $\varphi\oplus \psi\le c_2$, $\varphi(0)=\psi(0)=0$, and $\varphi(x)+\psi(y)= c_2(x,y)$ for all $x,y\in \supp(\ga)$.
\end{theorem}

\begin{remark}\label{rem:cost-equal}
	As established in Proposition~\ref{prop:prelim2}, if $\mu,\nu$ have finite equal mass, the L\'evy transport cost is always bounded above by the classical transport cost, that is,
	\(
	\mcc_2^\Lambda(\mu,\nu) \;\le\; \mcc_2(\mu,\nu).
	\)
	As a consequence of the preceding theorem, the two costs coincide if and only if the point $(0,0)$ lies in the support of a classical optimal coupling of $\mu$ and $\nu$.
\end{remark}

\begin{remark}\label{dual-equal}
	The pair of functions $(\varphi,\psi)$ from (iii) is called \emph{a Kantorovich potential} for the pair $\mu,\nu$. Particularly it holds
	\begin{align*}
		\int_{\mbr^d} \varphi(x)d\mu(x)+\int_{\mbr^d} \psi(y)d\mu(y)=\inn{\mu,\varphi}+\inn{\nu,\psi}&= \inn{\ga,\varphi\oplus\psi}= \inn{\ga,c_2} =\mcc_2^\La(\mu,\nu).
	\end{align*}
	The second equality is due to Proposition \ref{FTOLOT3to1lemma} and  that $\varphi\in L^1(\mu),\psi\in L^1(\nu)$. 
\end{remark}

\begin{proof}
    (i) $\to$ (ii). Suppose $\ga$ is a 2-optimal L\'evy coupling of $\mu,\nu$. 
    Consider the decomposition $\ga = \ga_0+ \zeta$, where 
    \begin{align*}
        d\ga_0(x,y)&= \min\cb{1,|x|^2+|y|^2}d\ga(x,y), \qquad d\zeta(x,y) = (1-\min\cb{1,|x|^2+|y|^2})d\ga(x,y). 
    \end{align*}
    Note that $\ga_0$ is a bounded Borel measure on $\mbr^{2d}$ with $\supp(\ga_0)=\supp(\ga). $    

    Fix $\al>0$ (e.g., $\al=1$) and consider the bounded measure $\tilde \ga_0 = \ga_0+ \al\de_{(0,0)}$, for then $\supp(\tilde \ga_0)= \supp(\ga_0)\cup\{(0,0)\}$. To establish (ii), we will use Theorem \ref{FTOT}, (iii)$\to$(i). So it reduces to show that $\tilde \ga_0$ is classically 2-optimal with respect to its marginals:
%
    \begin{align*}
        \tilde \mu_0 &= \#_1 (\ga_0)+ \al \de_0,\qquad \tilde \nu_0  = \#_2(\ga_0) +\al \de_0. 
    \end{align*}
    (Note: $\tilde \mu_0,\tilde \nu_0$ have equal mass, that is, $\tilde \mu_0(\mbr^d)=\tilde \nu_0(\mbr^d)=\ga_0(\mbr^{2d})+\al$.)
    Suppose $\tilde \ga_0$ is not classically 2-optimal, that is, there exists $\tilde \ga_0'\in \Ga(\tilde \mu_0,\tilde \nu_0)$ such that 
    \begin{align*}
    	\inn{\tilde\ga_0',c_2}<\inn{\tilde \ga_0,c_2}. 
    \end{align*}    
    Consider the measure $\ga' = \tilde \ga_0'+\zeta$. 
    We verify that $\ga'$ is a L\'evy coupling of $\mu,\nu$. For any continuous function $\varphi\in C_b(\mbr^d)$ vanishing on a neighborhood of $0$, we have $\inn{\de_0,\varphi}=0$. Therefore,
    \begin{align*}
    	\inn{\ga',\varphi\otimes 1}&= \inn{\tilde\ga_0',\varphi\otimes 1}+\inn{\zeta,\varphi\otimes 1} = \inn{\#_1(\ga_0)+\al \de_0,\varphi } + \inn{\zeta, \varphi\otimes 1}\\
    	&= \inn{\ga_0,\varphi\otimes 1}+ \inn{\zeta,\varphi\otimes 1} = \inn{\ga,\varphi\otimes 1} = \inn{\mu,\varphi}. 
    \end{align*}
%
    The last second equality is due to that $\ga_0+\zeta = \ga$. 
    A similar argument also shows $\inn{\ga',1\otimes \varphi}=\inn{\nu,\varphi}$. 
    Hence $\ga'$ is a L\'evy coupling of $\mu,\nu$, i.e., 
    $\ga' \in \Ga^\La(\mu,\nu)$. 
    Moreover, we find
	\begin{align*}
		\inn{\ga',c_2}&= \inn{\tilde\ga_0',c_2}+ \inn{\zeta,c_2} < \inn{\tilde\ga_0,c_2} + \inn{\zeta,c_2} = \inn{\ga_0,c_2}+\inn{\zeta,c_2} = \inn{\ga,c_2}.
	\end{align*}    
    This contradicts with the L\'evy 2-optimality of $\ga$. Hence $\tilde \ga_0$ is classically 2-optimal.

    (ii) $\to$ (iii).  
    By the fundamental theorem of (classical) optimal transport (Theorem \ref{FTOT}(ii)),
    there is a pair of lower semicontinous potentials $(\varphi,\psi)$ such that $\varphi\oplus\psi\le c_2$ and $\varphi(x)+\psi(y)=c_2(x,y)$ for all $(x,y)\in \supp(\gamma)\cup\{(0,0)\}$. Particularly we have $\varphi(0)+\psi(0)=0$. By adding and subtracting a constant if necessary, we may normalize so that \(\varphi(0) = \psi(0) = 0\). 
    
    It remains to verify that \( \varphi \in L^1(\mu),\psi\in L^1(\nu) \). Note that \( \varphi^+\in L^1(\mu), \psi^+ \in L^1(\nu) \) since  
    \(
    0 \le \varphi^+(x), \psi^+(x) \le \tfrac{1}{2}|x|^2,
    \)
    and both \( \mu \) and \( \nu \) have finite second moments. By Lemma \ref{FTOLOT3to1lemma},
    \begin{align*}
    	\inn{\mu,\varphi}+\inn{\nu,\psi}= \inn{\ga,\varphi\otimes 1}+\inn{\ga,1\otimes \psi}= \inn{\ga,\varphi\oplus\psi}= \inn{\ga,c_2}.
    \end{align*}
%
    where the integrals are understood in the extended real line \( [-\infty, \infty) \). The linearity used above is valid in this context. Since \( \gamma \) also has finite second moment, the integral $\inn{\ga,c_2}$ is finite nonnegative, implying
    \[
    \inn{\mu,\varphi}+\inn{\nu,\psi} \in(0,\infty).
    \]
    As \( \varphi = \varphi^+ - \varphi^- \) and \( \varphi^+ \in L^1(\mu) \), it follows that \( \varphi^- \in L^1(\mu) \), and similarly \( \psi^- \in L^1(\nu) \). Hence, \( \varphi \in L^1(\mu) \) and \( \psi \in L^1(\nu) \).

    (iii) $\to$ (i).  We have $\varphi(0)=\psi(0)=0$, $\varphi\oplus\psi \le c_2$, with the equality holds when $(x,y)\in\supp(\ga)\cup\{(0,0)\}$. Hence, for any $\tilde{\gamma}\in\Gamma^{\Lambda}(\mu,\nu)$
    \begin{align*}
    	\inn{\ga,c_2}&= \inn{\ga,\varphi\oplus\psi} = \inn{\mu,\varphi}+\inn{\nu,\psi} = \inn{\tilde \ga,\varphi\oplus\psi}\le \inn{\tilde \ga,c_2}. 
    \end{align*}
    Here, the second and third equalities  use Lemma \ref{FTOLOT3to1lemma}. 
     Therefore (i) is established.
\end{proof}

\input{S3.1}

%% file: S3.1.tex
\subsection{L\'evy-Kantorovich duality}
A consequence of Theorem \ref{FTOLOT} is a Kantorovich-type duality for the L\'evy optimal transport problem. In classical optimal transport theory, the Kantorovich duality states that for any \( \mu, \nu \in \mathcal{P}(\mathbb{R}^d) \),
\begin{align*}
	\mathcal{C}_2(\mu, \nu) &= \sup_{\varphi, \psi} \left[ \int_{\mathbb{R}^d} \varphi(x) \, d\mu(x) + \int_{\mathbb{R}^d} \psi(y) \, d\nu(y) \right],
\end{align*}
where the supremum is taken over all continuous functions \( \varphi, \psi \) such that \( \varphi \oplus \psi \leq c_2 \), with \( c_2(x, y) := \frac{1}{2} |x - y|^2 \). By a density argument, one can show that this identity continues to hold when the supremum is restricted to smaller but dense classes of functions, such as \( C_b(\mathbb{R}^d) \), \( C_0(\mathbb{R}^d) \), \( C_c(\mathbb{R}^d) \), or \( C_0^2(\mathbb{R}^d) \).

We now turn to the duality principle for the L\'evy optimal transport problem, which we refer to as the \emph{L\'evy–Kantorovich duality}.

\begin{theorem}[L\'evy-Kantorovich duality]\label{LK-duality-thm}
	Let \(\mu, \nu \in \Lambda_2(\mathbb{R}^d)\). Then
	\begin{align}\label{K-dual}
		\mathcal{C}^\Lambda_2(\mu, \nu) = \sup_{\varphi, \psi} \left[ \int_{\mathbb{R}^d} \varphi(x) \, d\mu(x) + \int_{\mathbb{R}^d} \psi(y) \, d\nu(y) \right],
	\end{align}
	where the supremum is taken over all pairs \((\varphi, \psi)\) satisfying: \(\varphi \oplus \psi \le c_2\), \(\varphi(0) = \psi(0)=0\), and any of the following regularity conditions:
	\begin{enumerate}[label=(\roman*)]
		\item \(\varphi, \psi\) are lower semicontinuous;
		\item \(\varphi, \psi \in C_2^2(\mathbb{R}^d)\).
		\item \(\varphi, \psi\in C_2^2(\mbr^d)\), and vanish on a neighborhood of the origin; 
		\item \(\varphi, \psi\in C_b^2(\mbr^d)\), and vanish on a neighborhood of the origin; 
	\end{enumerate}
\end{theorem}

Before proving the theorem, we first record the following consequence.  

\begin{corollary}\label{cor:lopt-dual}
	Let $\mu,\nu \in \Lambda_2(\mbr^d)$ and set $\mA=\La(0,0,\mu)$, $\mB=\La(0,0,\nu)$.  
	Then, for the restricted dual transport derivative $\om_2'$ from Definition~\ref{def:omega-prime}, one has
	\[
	\mcc_2^\La(\mu,\nu) = \om_2'(0,0;\mA,\mB).
	\]
\end{corollary}

\begin{proof}
	Take $\vphi,\psi \in C_2^2(\mbr^d)$ with $\vphi(0)=\psi(0)=0$ and $\vphi \oplus \psi \le c_2$.  
	In particular, $\vphi$ and $\psi$ touch $\tfrac12|x|^2$ at $0$ from below, which forces $\nabla \vphi(0)=\nabla \psi(0)=0$.  
	Thus,
	\begin{align}\label{eq:temp3.11}
		\mA\vphi(0) 
		= \int_{\mbr^d} \big[\vphi(x+x')-\vphi(x)-(x')^\top\nabla\vphi(x)\big]\Big|_{x=0}\, d\mu(x') 
		= \int_{\mbr^d}\vphi(x')\,d\mu(x'),	
	\end{align}
	and similarly, $\mB\psi(0) = \int_{\mbr^d}\psi(y')\,d\nu(y')$.  
	Hence, for all admissible $\vphi,\psi$,
	\[
	\mA\vphi(0)+\mB\psi(0)
	= \int_{\mbr^d}\vphi(x)\,d\mu(x) + \int_{\mbr^d}\psi(y)\,d\nu(y).
	\]
	Taking the supremum over such pairs $(\vphi,\psi)$, the definition of $\om_2'$ together with Theorem~\ref{LK-duality-thm}(ii) yields the claim.
\end{proof}

Let us now return to the proof of Theorem \ref{LK-duality-thm}. The formula \eqref{K-dual} holds with the supremum taken over lower semicontinuous pair is contained in Theorem \ref{FTOLOT}(iii). To extend the formula to classes of functions with stronger regularity conditions, we need the following lemmas. 

\begin{lemma}\label{lem:supgap}
	Let $X$ be a topological space and $F:X\to \mbr\cup\{-\infty\}$ be an upper semicontinuous function. Suppose $Y\subset X$ is dense. Then it holds
	\begin{align*}
		\sup_{x\in X} F(x)=\sup_{x\in Y} F(x). 
	\end{align*}
\end{lemma}

\begin{lemma}\label{lem:denseL1}
	Let $\mu,\nu\in \La_2(\mbr^d)$ be two L\'evy measures. For any pairs $(\varphi,\psi)\in L^1(\mu)\oplus L^1(\nu)$ satisfying $\varphi\oplus \psi\le c_2$, there is a sequence of pairs $\{(\varphi_n,\psi_n)\}_n$ such that 
	\begin{itemize}
		\item $\varphi_n,\psi_n\in C_b^2(\mbr^d)$ and vanish at a neighborhood of the origin;
		\item $\varphi_n\oplus \psi_n\le c_2$;
		\item $\varphi_n\to\varphi$ in $L^1(\mu)$, $\psi_n\to\psi\in L^1(\nu)$. 
	\end{itemize}
\end{lemma}

The proof of Lemma~\ref{lem:supgap} is elementary and omitted. In contrast, the proof of Lemma~\ref{lem:denseL1} is more technical, relying on the construction of an approximating sequence of smooth function pairs that converge in \( L^1 \) to a given admissible pair \((\varphi, \psi)\). To avoid disrupting the main flow of the presentation, we defer the proof to the appendix.
Let us now return to the proof of Proposition \ref{LK-duality-thm}.

\begin{proof}[Proof of Proposition \ref{LK-duality-thm}]
	(i) Let $\ga$ be a 2-L\'evy optimal coupling of $\mu,\nu$, whose existence is established in Theorem \ref{lotp-thm}.
	Consider any pair $(\varphi,\psi)$ of lower semicontinuous with $\varphi\oplus\psi\le c_2$ and $\varphi(0)=\psi(0)=0$. Particularly we have $\varphi(x),\psi(x)\le \frac 12 |x|^2$, and so $\varphi^+\in L^1(\mu),\psi^+\in L^1(\nu)$. By Proposition \ref{FTOLOT3to1lemma}, we then have
	\begin{align*}
			\int_{\mathbb{R}^d} \varphi(x) \, d\mu(x) + \int_{\mathbb{R}^d} \psi(y) \, d\nu(y)&= \int_{\mbr^{2d}} \sqb{\varphi(x)+\psi(y)}d\ga(x,y) \\
			&\le \int_{\mbr^{2d}}c_2(x,y)d\gamma(x,y)=\mcc_2^\La(\mu,\nu). 
		\end{align*}
	Taking the supremum over all such pairs we find \eqref{K-dual} holds with ``$\ge$" in place of ``=." The equality in fact holds, as the maximizer exists, as proved in Theorem \ref{FTOLOT} (see Remark \ref{dual-equal}). 
	
	For the remaining of the proof, denote $\mcl{F}_\bullet$ for $\bullet=$(i), (ii),  (iii) , (iv) 
	the family of all pairs $(\varphi,\psi)$ that satisfies the condition given by (i)--(iv) of Proposition \ref{LK-duality-thm} respectively, and $\mcc_{\bullet}$ for the supremum of 
	\begin{align*}
		\mcc_\bullet = \sup_{(\varphi,\psi)\in\mcl{F}_\bullet }F(\varphi,\psi):= \sup_{(\varphi,\psi)\in\mcl{F}_\bullet }\sqb{\int_{\mbr^d} \varphi(x)d\mu(x)+\int_{\mbr^d} \psi(y)d\nu(y)}.
	\end{align*}
	 Note we have the inclusion $\mcl{F}_{(i)}\supset \mcl{F}_{(ii)}\supset \mcl{F}_{(iii)}\supset \mcl{F}_{(iv)}, $
	 which then implies
	 \begin{align*}
	 	\mcc_{(i)}\ge \mcc_{(ii)}\ge \mcc_{(iii)}\ge \mcc_{(iv)}.
	 \end{align*}
	 The proof is complete if we prove $\mcc_{(i)}=\mcc_{(iv)}$. By Lemma \ref{lem:denseL1}, the space $\mcl{F}_{(iv)}$ is dense in $\mcl{F}_{(i)}$ w.r.t the metric topology of the product space $L^1(\mu)\oplus L^1(\nu)$. Lemma \ref{lem:supgap} then implies $\mcc_{(i)}=\mcc_{(iv)}$. This completes the proof. 
\end{proof}

%% file: S4.tex
\section{Optimal Markovian Couplings and Strong Duality for Drift, Diffusion and Jump Parts} \label{sec-opt-markov}

As an intermediate step toward proving the existence of optimal Markovian couplings (Theorem \ref{main1})
and establishing strong duality (Theorem \ref{main2}) for a full L\'evy generator, 
we begin by treating the drift, diffusion, and jump components separately. 
For each component, we construct an explicit $2$-optimal Markovian coupling 
and verify the corresponding strong duality relation. 
This decomposition will then serve as the basis for combining the components 
and addressing the general case in the coming section. 

Throughout this and the following sections, we simplify notation by suppressing the dependence on the fixed generators $\mA,\mB$ whenever the context is clear. In particular, for the transport derivatives $\ta_2, \om_2^\pm, \om_2^*, \om_2'$, etc., we write
\[
\ta_2 = \ta_2(\cdot,\cdot;\mA,\mB), \qquad 
\ta_2(x,y) = \ta_2(x,y;\mA,\mB).
\]

\subsection{Affine structure of $\ta_2,\om_2'$}
We begin by establishing the affine identity for $\ta_2$ from Theorem \ref{main1}(i), along with the corresponding identity for $\om_2'$.

\subsubsection{Affine structure of $\ta_2$}
Let $\mA,\mB\in \mcg_2^\La(\mbr^d)$. 
Recall the definition of Markovian transport derivative $\ta_2$ from Definition \ref{c-opt-gen}:
\begin{align*}
	\ta_2(x,y;\mA,\mB)&= \inf_{\mJ\in\Ga(\mA,\mB)} (\mJ^- c_2)(x,y),\qquad c_2(x,y)=\frac 12 |x-y|^2,
\end{align*}
where the infimum is taken over all Feller couplings $\mJ$ of $\mA,\mB$, and $\mJ^-$ is the lower generator of $\mJ$, as given in Definition \ref{def:lower-gen}. We note the lower generator $\mJ^-$ has the following additive property: if $\Phi\in C(\mJ), \Psi\in \bar D(\mJ)$, then 
\begin{align}\label{eq:liminf-add}
	\mJ^-(\Phi+\Psi)=\mJ^-\Phi + \mJ\Psi. 
\end{align}

To prove Theorem \ref{main1}(i), we first introduce a technical lemma. 
For $h \in \mbr^d$, denote by $\si_h : C(\mbr^d) \to C(\mbr^d)$ the translation operator defined by
\begin{align*}
	(\si_h \varphi)(x) := \varphi(x+h).
\end{align*}
Given a Feller generator $\mJ \in \mcg(\mbr^d)$, we define its \emph{$h$-translated generator} by
\begin{align*}
	\mJ^h := \si_h^{-1} \mJ \si_h = \si_{-h} \mJ \si_h.
\end{align*}
It is straightforward to check that $\mJ^h$ is again a Feller generator, and it generates the Feller semigroup
\(
e^{t \mJ^h} = \si_h^{-1} e^{t\mJ} \si_h,
\)
see \cite[Chapter~II, Section~2.1]{engel2000one}.
It is also easy to verify that the lower generator of $\mJ^h$ satisfies the same relation:
\begin{align}\label{eq:lg-trans}
	\mJ^{h-}:=(\mJ^{h})^-=\si_h^{-1}\mJ^-\si_h. 
\end{align}

\begin{lemma}\label{trans-lem}
	Let $\mA,\mB\in \mcg^\La_2(\mbr^d)$, $m_\mA,m_\mB\in\mbr^d$ be their mean vectors, and $\mJ\in\Ga(\mA,\mB)$ be a Feller coupling of two L\'evy generators. For any $h=(\tau,\tau')\in\mbr^{2d}$, it holds 
	\begin{align*}
		[\mJ^- \si_h c_2](x,y)&= \mJ^- c_2(x,y)+ (m_\mA-m_\mB)^\top (\tau-\tau'). 
	\end{align*}
\end{lemma}

\begin{proof}
	Observe the following basic identity for squared cost: 
	\begin{align*}
		(\si_h c_2)(x,y)= c_2(x+\tau,y+\tau')&= c_2(x,y) + c_2(\tau,\tau') + (x-y)^\top(\tau-\tau').
	\end{align*}
	Write $\Theta(x) = x^\top (\tau-\tau')$. Note that $\Ta\otimes 1,1\otimes \Ta\in \bar D(\mJ)$, by the definition of coupling generators.
	 Since $c_2(\tau,\tau')$ is a constant (w.r.t. $x,y$), acting $\mJ$ to it becomes zero.
	This follows from \eqref{eq:liminf-add}:
	\begin{align*}
		[\mJ^- \si_h c_2](x,y)&= (\mJ^- c_2)(x,y)+ \mJ (\Theta\otimes 1)(x,y)-\mJ(1\otimes \Theta)(x,y)\\
		&= (\mJ^- c_2)(x,y)+ m_\mA^\top (\tau-\tau') -m_\mB^\top(\tau-\tau')\\
		&= (\mJ^- c_2)(x,y)+ (m_\mA-m_\mB)^\top (\tau-\tau'). \qedhere
	\end{align*}
\end{proof}

Let us first establish Theorem \ref{main1}(i). 

\begin{proof}[Proof of Theorem \ref{main1}(i)]
	Let $x_1,y_1,x_2,y_2\in\mathbb{R}^{d}$, and $h=(\tau,\tau')=(x_2-x_1,y_2-y_1)$. Fix $\epsilon>0$ and let $\mathcal{J}_{\epsilon}\in\Gamma(\mathcal{A,B})$ be a coupling generator (depending on $x_1,y_1$) such that
	\begin{equation}\label{temp:4.3}
		\mathcal{J}_{\epsilon}^-(c_2)(x_1,y_1)\leq\inf_{\mathcal{J}\in\Gamma(\mathcal{A,B})}\mathcal{J}^-(c_2)(x_1,y_1)+\epsilon = \ta_2(x_1,y_1)+\ep,
	\end{equation}
	where $\mJ_{\ep}^-$ is the lower generator of $\mJ_\ep$. 
	Consider the $h$-translated generator $\mathcal{J}_{\epsilon}^h=\sigma_h^{-1}\mathcal{J}_{\epsilon}\sigma_h$. Note we have $\mJ^h_\ep\in \Ga(\mA,\mB)$.     
	To see why, notice for any $\varphi\in \bar D(\mA)$ we have $\si_h [\varphi\otimes 1]=(\si_\tau \varphi)\otimes 1$. Thus, by the translation invariance of L\'evy generators, particularly $\si_\tau \mA = \mA\si_\tau$, we have
	\begin{align*}
		\mJ_\ep^h[\varphi\otimes 1]=(\sigma_h^{-1}\mJ_\ep \sigma_h)[\varphi\otimes 1]=\sigma_h^{-1}[(\mA\sigma_\tau\varphi)\otimes 1]=(\si_\tau^{-1}\mA\sigma_\tau\varphi)\otimes 1=(\mathcal{A}\varphi)\otimes 1. 
	\end{align*}
	Similarly we can argue that $\mJ_\ep^h[1\otimes \psi]=\mB\psi$ for any $\psi \in \bar D(\mB)$. 
	Hence $\mJ_\ep^h \in \Ga(\mA,\mB)$, and we therefore conclude
	\begin{align}\label{temp:4.4}
		\ta_2(x_2,y_2)=\inf_{\mJ\in\Ga(\mA,\mB)}\mJ^-(c_2)(x_2,y_2)\le (\mJ_{\ep}^{h-} c_2)(x_2,y_2).
	\end{align}
	
	We now calculate $(\mcl{J}_{\ep-}^h c_2)(x_2,y_2)$. 
	First by Lemma \ref{trans-lem}, we have
	\begin{align*}
		(\mJ_\ep^- \si_h c_2)(x_2,y_2) 
		&= (\mJ_\ep^- c_2)(x_2,y_2) + (m_\mA-m_\mB)^\top(\tau-\tau'). 
	\end{align*}
	Applying the inverse translation $\si_{h}^{-1}$ and the identity \eqref{eq:lg-trans}, this follows 
	\begin{align*}
		(\mJ_\ep^{h-}c_2)(x_2,y_2)&=(\si_h^{-1}\mJ^-_\ep \si_h c_2)(x_2,y_2) =(\mJ^-_\ep c_2)(x_2-\tau,y_2-\tau') + (m_\mA-m_\mB)^\top(\tau-\tau')\\
		&= (\mJ^-_\ep c_2)(x_1,y_1) + (m_\mA-m_\mB)^\top((x_2-x_1)-(y_2-y_1)).
	\end{align*}
	By \eqref{temp:4.3} and \eqref{temp:4.4}, this follows
	\begin{align*}
		\ta_2(x_2,y_2)\le \ta_2 (x_1,y_1) - (m_\mA-m_\mB)^\top((x_2-x_1)-(y_2-y_1)) +\ep. 
	\end{align*}
	Passing $\ep\searrow 0,$
	\begin{align*}
		\ta_2(x_2,y_2)\le \ta_2 (x_1,y_1) + (m_\mA-m_\mB)^\top((x_2-x_1)-(y_2-y_1)). 
	\end{align*}
	Switching the role of $(x_1,y_1)$ and $(x_2,y_2)$, the above holds with $\leq$ replaced by $=$. The proposition follows as the special case $(x_1,y_1)=(0,0)$.
\end{proof}

\subsubsection{Affine structure of $\om'_2$}
We next prove the affine identity for $\om_2'$. 
Let us recall the definition of the \emph{restricted dual transport derivative}.  
For $\mA,\mB \in \mcg^\La_2(\mathbb{R}^d)$, it is given by 
\[
\omega_2'(x_0,y_0;\mA,\mB) 
:= \sup_{(\vphi,\psi)\in \mcD'(x_0,y_0)}[\mA\vphi(x_0)+\mB\psi(y_0)],
\]
where 
$\mathcal{D}'(x_0, y_0)$ denotes the set of admissible pairs $(\varphi, \psi)\in[ C_2^2(\mbr^d)]^2$ satisfying the \emph{touching condition}
\begin{align*}
	\varphi(x) + \psi(y) \,\le\, \tfrac{1}{2}|x-y|^2, 
	\quad \text{with equality at } (x_0,y_0).
\end{align*}

The first step towards proving the chain of equalities in \eqref{eq:om2'-chain} is to show that $\omega_2'$ admits the same affine structure as $\theta_2$, established in Theorem~\ref{main1}(i).

\begin{proposition}\label{om*-linear}
	Let $\mA,\mB\in \mcg_2^\La(\mbr^d)$. 
	It holds for all $x_0,y_0\in \mbr^d$:
	\begin{align*}
		\om_2'(x_0,y_0;\mA,\mB)&= \om_2'(0,0;\mA,\mB)+ (m_\mA-m_\mB)^\top (x_0-y_0). 
	\end{align*}
\end{proposition}

\begin{proof}
	Let us assume that $\mA,\mB$ admit the global form $\mA = \Lag(\ka,\al,\mu),\ \mB=\Lag(\zeta,\be,\nu)$. Suppressing the generators $\mA,\mB$ in $\omega_2'$, we are to prove: for all $x_0,y_0\in\mbr^d$,
	\begin{align*}
		\omega_2'(x_0,y_0) = \omega_2'(0,0) + (\ka-\zeta)^\top (x_0-y_0). 
	\end{align*}
	
	Fix $(\varphi,\psi)\in \mcl{D}'(x_0,y_0)$ and consider the second-order Taylor remainder of $\varphi,\psi$ at $x_0,y_0$ respectively:
	\begin{align*}
		\tilde \varphi(x) := \varphi(x+x_0)-\varphi(x_0)-x^\top \nabla\varphi(x_0),\\
		\tilde \psi(y) := \psi(y+y_0)-\psi(y_0)-y^\top \nabla\psi(y_0).
	\end{align*}
	We claim that $(\tilde \varphi,\tilde \psi)\in \mcl{D}'(0,0)$. To see this, first note that since $\varphi\oplus\psi$ touches $c_2$ from below at $(x_0,y_0)$, their gradients agree at $(x_0,y_0)$, i.e.,
	\begin{align*}
		(\nabla \varphi(x_0),\nabla \psi(y_0))
		= (\nabla_{(x,y)}(\varphi\oplus \psi))(x_0,y_0) 
		= (\nabla_{(x,y)} c_2)(x_0,y_0) 
		= (x_0-y_0,\ y_0-x_0).
	\end{align*}
	This yields
	\begin{align*}
		\tilde \varphi(x)+\tilde \psi(y)
		&= \varphi(x+x_0)+\psi(y+y_0)-\big(\varphi(x_0)+\psi(y_0)\big) - x^\top \nabla\varphi(x_0) - y^\top \nabla\psi(y_0)\\
		&\le \frac12 |(x+x_0)-(y+y_0)|^2 - \frac12 |x_0-y_0|^2 - (x-y)^\top (x_0-y_0)\\
		&= \frac12 |x-y|^2 = c_2(x,y). 
	\end{align*}
	Clearly the pair is twice differentiable and satisfies the quadratic bound \eqref{eq:quad-bound}. 
	Since $\tilde \varphi(0)=\tilde \psi(0)=0$, this shows $(\tilde \varphi,\tilde\psi) \in \mcl{D}'(0,0)$. 
	
	We next observe that 
	\begin{align*}
		D^2\varphi(x_0)= D^2\tilde\varphi(0), \quad D^2\psi(y_0)= D^2\tilde\psi(0), \quad \nabla \tilde \varphi(0)=0=\nabla \tilde \psi(0). 
	\end{align*}
	If we decompose the generators $\mA,\mB$ into their diffusion, drift, and jump parts, we find 
	\begin{align*}
		\mA^\De\varphi(x_0)=\mA^\De \tilde \varphi(0),\qquad& \mA^\nabla \varphi(x_0)=\ka^\top(x_0-y_0),\\
		\mB^\De\psi(y_0)=\mB^\De \tilde \psi(0), \qquad &\mB^\nabla \psi(y_0)=- \zeta^\top (x_0-y_0). 
	\end{align*}
	Moreover, since $\tilde \varphi(0)=0$ and $\nabla \tilde \varphi(0)=0$, it follows that
	\begin{align*}
		\mA^J\varphi(x_0) 
		&= \int_{\mbr^d} \big[\varphi(x+x_0)-\varphi(x_0) - x\cdot \nabla \varphi(x_0)\big]\,d\mu(x) 
		= \int_{\mbr^d} \tilde \varphi(x)\,d\mu(x) = \mA^J\tilde \varphi(0). 
	\end{align*}
	Likewise, $\mB^J\psi(y_0)=\mB^J\tilde\psi(0)$. 
	Combining all these observations, we have
	\begin{align*}
		\mA \varphi(x_0)+ \mB \psi(y_0)
		&= \mA^\De \tilde \varphi(0)+\mB^\De \tilde \psi(0) + \mA^J \tilde \varphi(0)+\mB^J \tilde \psi(0) + (\kappa-\zeta)^\top (x_0-y_0)\\
		&= \mA\tilde \varphi(0)+\mB\tilde \psi(0) + (\kappa-\zeta)^\top (x_0-y_0)\\
		&\le \omega_2'(0,0)+(\kappa-\zeta)^\top (x_0-y_0). 
	\end{align*}
	Taking the supremum of the left-hand side over all $(\varphi,\psi)\in \mcl{D}'(x_0,y_0)$ we find the bound 
	\begin{align*}
		\omega_2'(x_0,y_0;\mA,\mB) \le \omega_2'(0,0;\mA,\mB)+(\kappa-\zeta)^\top (x_0-y_0). 
	\end{align*}
	
	The reverse inequality follows from the same argument above, together with the following observation: if $(\tilde\varphi,\tilde \psi)\in \mcl{D}'(0,0)$, and 
	\begin{align*}
		\varphi(x) &= \tilde\varphi(x-x_0)+ x^\top (x_0-y_0),\qquad 
		\psi(y) = \tilde\psi(y-y_0)+ y^\top (y_0-x_0),
	\end{align*}
	then $(\varphi,\psi)\in \mcl{D}'(x_0,y_0)$. 
\end{proof}

\subsection{Reduction of global to local minimization}

The general $c$-optimal Markovian coupling problem for a pair of Feller generators 
\(\mathcal{A}, \mathcal{B} \in \mcg(\Pi)\) is, by default, a \emph{global} optimization problem. 
That is, the aim is to find a \emph{Feller coupling}—a Feller generator on \(\Pi^2\)—denoted by \(\mathcal{J}_*\), 
such that the minimum is attained pointwise:
\[
(\mathcal{J}_* c)(x,y) = \theta_c(x,y;\mathcal{A},\mathcal{B}), 
\qquad (x,y)\in \Pi^2.
\]

In the case of optimal Markovian couplings between L\'evy generators with respect to the quadratic cost 
\(c_2(x,y) = \tfrac{1}{2}|x-y|^2\), the problem admits a significant simplification. 
Owing to the translation invariance of L\'evy generators—and, in particular, the affine structure of 
\(\ta_2\) and \(\om_2'\)—the global optimization problem reduces to a \emph{local} one. 
In fact, it suffices to minimize the pointwise expression \((\mathcal{J} c_2)(x,y)\) at a single 
reference point, such as \((x,y)=(0,0)\), among all \emph{L\'evy} (rather than merely \emph{Feller}) coupling generators.

In what follows, for $\mA,\mB \in \mcg_2^\La(\mbr^d)$, we denote by $\Ga^\La(\mA,\mB) \subset \Ga(\mA,\mB)$ 
the set of all L\'evy coupling generators of $\mA$ and $\mB$.

\begin{proposition}\label{0isenough}
	Let $\mA,\mB \in \mcg_2^\Lambda(\mbr^d)$.
	
	(i) Suppose there exists a L\'evy generator $\mJ_* \in \Ga^\Lambda(\mA,\mB)$ such that 
	\[
	(\mJ_* c_2)(0,0) = \ta_2(0,0;\mA,\mB).
	\]
	Then $\mJ_*$ is a $2$-optimal Markovian coupling of $\mA$ and $\mB$. 
	
	(ii) The same conclusion as in (i) holds if $\ta_2$ is replaced by $\om_2'$. 
	In this case, the strong duality relation is satisfied:
	\[
	\om_2'(x,y;\mA,\mB) = \ta_2(x,y;\mA,\mB).
	\]
\end{proposition}

\begin{proof}
	Let us again suppress the generators $\mA,\mB$ in $\ta_2,\om_2'$. 
	
    (i) Let $\mJ_*$ be given in the corollary. To show $\mJ_*$ is optimal, we need to establish the identity $(\mJ_*c_2)(x_0,y_0)=\ta_2(x_0,y_0)$ for all $(x_0,y_0)\in\mbr^{2d}$. Let $h=(x_0,y_0)$, for then 
    \begin{align*}
        (\mJ_*c_2)(x_0,y_0)=(\si_h \mJ_* c_2)(0,0). 
    \end{align*}
    Since $\mJ_*$ is a L\'evy generator, it commutes with any translation operator, that is $\si_h \mJ_*=\mJ_*\si_h$. By Lemma \ref{trans-lem} we find 
    \begin{align*}
        (\mJ_*c_2)(x_0,y_0)&=(\si_h \mJ_* c_2)(0,0) = (\mJ_* \si_h c_2)(0,0) = (\mJ_*c_2)(0,0)+ (m_\mA-m_\mB)^\top(x_0-y_0)\\
        &= \ta_2(0,0)+(m_\mA-m_\mB)^\top(x_0-y_0) = \ta_2(x_0,y_0).
    \end{align*}
    The proof is complete.
    
    (ii) We note the inequality holds: $\om_2'(0,0)\le \ta_2(0,0)\le (\mJ_*c_2)(0,0)$. The first is due to Proposition \ref{om-ta-prop} and the second is from the definition of $\ta_2$. If $(\mJ_*c_2)(0,0)=\om_2'(0,0)$, then it forces the chain of inequality holds. Since $\om_2'(0,0)=\ta_2(0,0)$, the affine structure of $\ta_2,\om_2'$, particularly Theorem \ref{main1}(i) and Proposition \ref{om*-linear} imply $\ta_2=\om_2'$. 
\end{proof}

%

\subsection{Optimal Markovian couplings and strong duality of drift-diffusion parts}\label{sec:omc-dd}
Let us begin with considering the optimal Markovian couplings between a pair of drift-diffusion operators. Let $\al,\be\in \mcl{S}_{\ge 0}(\mbr^d)$ be two nonnegative definite $(d\times d)$-matrices, $\ka,\zeta\in \mbr^d$ be two drift vectors, and $\mA=\La_2 (\ka,\al,0),\mB=\La_2(\zeta,\be,0)$, namely,
\begin{align*}
	(\mA\varphi)(x)= \ka\cdot\nabla\varphi(x)+\f 12\tr[\al D^2\varphi](x),\qquad (\mB\psi)(y)=\zeta\cdot \nabla \varphi(y)+\f 12\tr[\be D^2\psi](y). 
\end{align*}
As discussed, the more general case of drift-diffusion generators with $x$-dependent coefficients was treated in \cite{ChenLi1989}. We present the construction here for the sake of completeness.

A prototypical coupling of the drift-diffusion operators \(\mathcal{A}\) and \(\mathcal{B}\) above is itself a diffusion operator  \(\mathcal{J} =\La_2 (\eta,\sigma, 0)\) on \(\mathbb{R}^{2d}\), where \(\sigma \in \mathcal{S}_{\ge 0}(\mathbb{R}^{2d})\) takes the block form
\begin{align}\label{si-def}
	\sigma = \begin{bmatrix}
		\alpha & K \\ K^\top & \beta
	\end{bmatrix}\ge 0,	
\end{align}
and $\eta =\ka\oplus \zeta\in \mbr^{2d}$. 
The off-diagonal blocks are transposes of each other to ensure that \(\sigma\) is symmetric. Furthermore, the matrix \(K\) is chosen so that \(\sigma\) remains nonnegative semidefinite.

The optimal Markovian coupling problem for two drift-diffusion generators is deeply related to the classical optimal transport problem between two \emph{Gaussian measures}, which has been thoroughly studied in the literature. Given a probability measure \(\mu \in \mathcal{P}(\mathbb{R}^d)\), we write
\[
\mu \sim \mathcal{N}(m, \sigma), \qquad m \in \mathbb{R}^d, \ \sigma \in \mathcal{S}_{\geq 0}(\mathbb{R}^d),
\]
to indicate that \(\mu\) is a $d$-dimensional Gaussian measure with mean \(m\) and covariance matrix \(\sigma\).
In \cite{givens1984class}, Givens and Shortt showed that if \(\mu \sim \mathcal{N}(m_1, \alpha)\) and \(\nu \sim \mathcal{N}(m_2, \beta)\), a (classical) 2-optimal transport plan \(\gamma_*\) between \(\mu\) and \(\nu\) is given by a Gaussian measure \(\gamma_* \sim \mathcal{N}(m, \sigma)\) on \(\mathbb{R}^{2d}\), where
\[
m = (m_1, m_2), \qquad \sigma = \begin{bmatrix}
\alpha & K_* \\
K_*^\top & \beta
\end{bmatrix},
\]
and \(K_*\) is a matrix that solves the following positive semidefinite programming problem:
\begin{align}\label{posdef-prog}
\text{minimize} \quad K \mapsto \tfrac{1}{2} \operatorname{tr}(\alpha + \beta - 2K), \qquad \text{subject to} \quad \begin{bmatrix}
\alpha & K \\
K^\top & \beta
\end{bmatrix} \geq 0.
\end{align}
We also note the uniqueness of such coupling holds if the covariance matrices $\al,\be$ are non-degenerate, i.e., strictly postive definite.

As a result, the 2-optimal transport cost is given by
\[
\mathcal{C}_2(\mu, \nu) = \tfrac{1}{2} \int_{\mathbb{R}^{2d}} |x - y|^2 \, d\gamma_*(x, y)
= \tfrac{1}{2} |m_1 - m_2|^2 + \mcW_{\mcS}(\alpha, \beta)^2,
\]
where
\begin{align}\label{D-def}
    \mcW_{\mcS}(\alpha, \beta)^2 := \frac{1}{2} \operatorname{tr}(\alpha + \beta - 2K_*) 
= \min_{K} \frac{1}{2} \operatorname{tr}(\alpha + \beta - 2K),    
\end{align}
with the minimum taken over all \(K\) satisfying the constraint in \eqref{posdef-prog}.
Moreover, \(\mcW_{\mcS}(\alpha, \beta)\) admits the following explicit formula:
\[
\mcW_\mcS(\al,\be)^2 :=\tfrac{1}{2} \operatorname{tr} \left[ \alpha + \beta - 2\left( \alpha^{1/2} \beta \alpha^{1/2} \right)^{1/2} \right].
\]
The functional $\mcW_\mcS$ defines on the space $\mcS_{\ge 0}(\mbr^d)$ of nonnegative definite matrices is called the \emph{Bures-Wasserstein distance}, which is a metric on the space.

With these results, we may now address the optimal Markovian coupling problem for drift-diffusion operators.

\begin{proposition}\label{dd-opt-prop}
Let \(\mathcal{A} =\Lag (\ka,\alpha, 0)\) and \(\mathcal{B} =\Lag (\zeta,\beta, 0)\) be drift-diffusion generators on \(\mathbb{R}^d\). Then the Markovian transport derivative is given by
\begin{align}\label{diff-min}
    \ta_2(x, y; \mathcal{A}, \mathcal{B}) =  (\ka-\zeta)^\top (x-y)+ \mcW_{\mcS}(\alpha, \beta)^2,
\end{align}
where \(\mcW_{\mcS}(\alpha, \beta)\) is as defined in \eqref{D-def}. Moreover, an optimal Markovian coupling of \(\mathcal{A}\) and \(\mathcal{B}\) is given by the drift diffusion generator \(\mathcal{J}=\Lag (\eta,\si, 0)\), where \(\sigma\) is defined as in \eqref{si-def} with \(K\) being a minimizer of the semidefinite program \eqref{posdef-prog}, and $\eta=(\ka,\zeta)\in\mbr^{2d}$.
\end{proposition}

\begin{proof}
    By Theorem \ref{main1}(i), it suffices to prove \eqref{diff-min} for $x=y=0$, that is, 
    \begin{align}\label{diff-min2}
    	\ta_2(0,0;\mA,\mB)&= \mcW_{\mcS}(\alpha, \beta)^2. 
    \end{align}

	Let $\mu_t =\de_0 e^{t\mA},\nu_t = \de_0 e^{t\mB}$. We have $\mu_t \sim \mcl{N}(t\ka, t\al),\nu_t\sim \mcl{N}(t\zeta,t\be)$. From the discussion earlier, $\ga_t\sim \mcl{N}(t\eta,t\si_*)$, $\eta = (\ka,\zeta)\in\mbr^{2d}$, is a 2-optimal coupling of $\mu_t,\nu_t$, where $\si_* \in \mcl{S}_{\ge 0}(\mbr^{2d})$ is given by \eqref{si-def}, \eqref{posdef-prog}. If we let $\mJ_*=\Lag(\eta,\si_*,0)$, then $\mJ_*\in\Ga(\mA,\mB)$ and we have $\ga_t=\de_{(0,0)} e^{t\mJ_*}$. 
    We specifically point out that the semidefinite programming problem \eqref{posdef-prog} has the following linear property:
    \begin{align*}
    	\mcW_{\mcS}(t\alpha, t\beta)^2=t\mcW_{\mcS}(\alpha, \beta)^2.
    \end{align*}
    Hence we have 
	\begin{align*}
		(e^{t\mJ_*}c_2)(0,0)=\int_{\mbr^{2d}} c_2(x,y) d\gamma_t=\mcc_2(\de_0 e^{t\mA},\de_0 e^{t\mB})&= \mcc_2(\mu_t,\nu_t)=\frac 12 t^2|\ka-\zeta|^2+ t\mcW_{\mcS}(\alpha, \beta)^2, 
	\end{align*}
    Subtracting $c_2(0,0)=0$ and dividing $t$ then passing $t\searrow 0$, it leads to
    \begin{align*}
        (   \mJ_*c_2)(0,0)=\om_2^+(0,0;\mA,\mB)&= \limsup_{t\searrow 0} \frac{\mcc_2(\de_0 e^{t\mA},\de_0 e^{t\mB})}{t}= \mcW_{\mcS}(\alpha, \beta)^2,
    \end{align*}
where we recall the pointwise upper transport derivative $\om_2^+$ from Definition \ref{diff-opt-def}. We point out the above also holds if $\om_2^+$ is replaced by the lower derivative $\om_2^-$ (and hence $\om_2^+=\om_2^-$). 
    Since $\ta_2(0,0;\mA,\mB)$ is the infimum of $(\mJ c_2)(0,0)$ over all $\mJ\in \Ga(\mA,\mB)$, it follows
    \begin{align*}
        \ta_2(0,0;\mA,\mB)\le \om_2^\pm(0,0;\mA,\mB)=\mcl{W}_\mcS(\al,\be)^2. 
    \end{align*}
    Recall the bound $\om_2^\pm\le \ta_2$ was established in Proposition \ref{om-ta-prop}. This proves \eqref{diff-min2}. Moreover, $\mJ_*$ is the 2-optimal coupling generator of $\mA,\mB$, because $(\mJ_* c_2)(0,0)=\ta_2(0,0;\mA,\mB)$. 
\end{proof}

We now present a duality formula for the minimization problem \eqref{diff-min}, which will play a key role in the next chapter. This duality arises from the theory of positive semidefinite linear programming. For completeness, we provide a proof using tools from optimal transport.

\begin{lemma}\label{lem:diff-dual}
	Let $\alpha,\beta \in \mathcal{S}_{\ge 0}(\mathbb{R}^d)$, and let $\mathcal{W}_{\mcS}(\alpha,\beta)^2$ denote the minimum value of the primal problem \eqref{posdef-prog}. 
	\begin{align}
		\mathcal{W}_{\mcS}(\alpha,\beta)^2 = \sup_{A,B} \operatorname{tr}[\alpha A + \beta B], \label{dual-id}
	\end{align}
	where the supremum is taken over all pairs of positive semidefinite matrices $(A,B)$ satisfying the matrix inequality
	\begin{align}\label{eq:mat-cond}
		\begin{bmatrix}
			A & 0 \\ 0 & B
		\end{bmatrix}
		\le \frac{1}{2}
		\begin{bmatrix}
			I & -I \\ -I & I
		\end{bmatrix}.
	\end{align}
\end{lemma}

\begin{proof}
	As discussed above, $\mathcal{W}_\mcS(\alpha,\beta)$ is the $2$-Wasserstein distance between two centered Gaussian measures $\mu \sim \mathcal{N}(0,\alpha)$ and $\nu \sim \mathcal{N}(0,\beta)$. By classical Kantorovich duality, we have
	\begin{align*}
		\mathcal{W}_2^2(\mu,\nu)
		= \sup_{\varphi\oplus\psi\le c_2} \left\{ \int_{\mathbb{R}^d} \varphi(x) \, d\mu(x) + \int_{\mathbb{R}^d} \psi(y) \, d\nu(y) \right\}.
	\end{align*}
	
	For Gaussian marginals, the optimal Kantorovich potentials $(\varphi,\psi)$ are known to be generalized quadratic functions. When $\alpha$ and $\beta$ are positive definite, the optimal potentials take the form
	\[
	\varphi(x) = x^\top A x, \quad \psi(y) = y^\top B y,
	\]
	for some symmetric matrices $A,B \ge 0$ satisfying the constraint
	\[
	\varphi(x) + \psi(y) \le \frac{1}{2}|x - y|^2 \quad \text{for all } x,y.
	\]
	This pointwise inequality is equivalent to the matrix inequality \eqref{eq:mat-cond}.
	
	In the case where $\alpha$ or $\beta$ is singular, the optimal potentials are limits of such quadratic forms supported on subspaces, and the dual value is still attained as the supremum over such quadratic pairs. Therefore, it suffices to restrict the dual formulation to quadratic potentials satisfying the cost constraint. In this case, using the property of Gaussian measures, we compute
	\begin{align*}
		\int_{\mathbb{R}^d} \varphi(x) \, d\mu(x) + \int_{\mathbb{R}^d} \psi(y) \, d\nu(y)
		&= \int_{\mathbb{R}^d} x^\top A x \, d\mu(x) + \int_{\mathbb{R}^d} y^\top B y \, d\nu(y) 
		= \operatorname{tr}(\alpha A) + \operatorname{tr}(\beta B).
	\end{align*}
	This proves the identity \eqref{dual-id}.
\end{proof}

\begin{corollary}\label{cor:sd-dd}
	Assume the setting of Proposition \ref{dd-opt-prop}. The strong duality holds: for all $x,y\in \mbr^d$,
	\begin{align*}
		\om_2'(x,y;\mA,\mB)=\ta_2(x,y;\mA,\mB)= (\ka-\zeta)^\top (x-y)+\mcW_\mcl{S}(\al,\be)^2. 
	\end{align*}
\end{corollary}

\begin{proof}
	The second equality is established in Proposition \ref{dd-opt-prop}. To prove the first, by the affine structure of $\ta_2,\om_2'$, it suffices to verify for $x=y=0$, that is, 
	\begin{align*}
		\om_2'(0,0)=\ta_2(0,0)= \mcW_{\mcl S}(\al,\be)^2. 
	\end{align*}
	
	Since \( \omega_2'(0, 0) \le \ta_2(0, 0) \) is established in Proposition~\ref{om-ta-prop}, it suffices to prove the reverse inequality.
	By Lemma~\ref{lem:diff-dual}, for any \( \varepsilon > 0 \), there exist positive definite matrices \( A, B \in \mathcal{S}_{> 0}(\mathbb{R}^d) \) satisfying the matrix condition~\eqref{eq:mat-cond} such that
	\[
	\mathcal{W}_{\mathcal{S}}(\alpha, \beta)^2 \le \operatorname{tr}[\alpha A + \beta B] + \varepsilon.
	\]
	Define test functions \( \varphi(x) = \frac{1}{2} x^\top A x \) and \( \psi(y) = \frac{1}{2} y^\top B y \). Then \( \varphi \oplus \psi \le c_2 \) and \( \varphi(0) = \psi(0) = 0 \), so \( (\varphi, \psi) \in \mcl{D}'(0,0) \). Hence,
	\[
	\omega_2'(0, 0) \ge \mA \varphi(0) + \mB \psi(0) = \operatorname{tr}[\alpha A + \beta B] \ge \mathcal{W}_{\mathcal{S}}(\alpha, \beta)^2 - \varepsilon.
	\]
	Taking \( \varepsilon \searrow 0 \) gives the required inequality.
\end{proof}

\subsection{Optimal Markovian coupling and strong dualty of jump parts}

We now solve the optimal Markovian coupling problem for pure jump generators. As hinted in the previous section, the solution is closely related to the Lévy optimal transport problem between the underlying Lévy measures.

\begin{proposition}\label{prop:J-opt}
	Let $\mu,\nu\in \La_2(\mbr^d)$ and set $\mA=\Lag(0,0,\mu)$, $\mB=\Lag(0,0,\nu)$. 
	Let $\ga_*$ be a L\'evy 2-optimal coupling of $\mu,\nu$. Then the operator $\mJ_*:=\Lag(0,0,\ga_*)$ is an optimal Markovian coupling of $\mA,\mB$. Moreover, the strong duality holds: 
	\begin{align*}
		\om_2'(x,y;\mA,\mB)=\ta_2(x,y;\mA,\mB)= \mcc_2^\La(\mu,\nu). 
	\end{align*}
\end{proposition}

\begin{proof}
	To show $\mJ_*$ is an optimal Markovian coupling and the strong duality (the first equality),
	by Proposition \ref{0isenough}(ii) it suffices to show $(\mJ_*c_2)(0,0)=\om_2'(0,0)$ . 
	Indeed, 
	\begin{align*}
		(\mJ_* c_2)(0,0)
		&=\int_{\mathbb{R}^{2d}} \Big( c_2(x+x',y+y') - c_2(x,y) - (x',y')^\top \nabla_{x,y}c_2(x,y) \Big)\big|_{(x,y)=(0,0)}\, d\ga_*(x',y').
	\end{align*}
	Since $c_2(0,0)=0$ and $\nabla_{x,y}c_2(0,0)=(0,0)$, the integrand reduces to $c_2(x',y')=\tfrac{1}{2}|x'-y'|^2$. Hence, by the L\'evy 2-optimality of $\ga_*$,
	\[
	(\mJ_*c_2)(0,0)= \int_{\mathbb{R}^{2d}} \tfrac{1}{2}|x'-y'|^2 \, d\ga_*(x',y')=\mcc_2^\La(\mu,\nu).
	\]
	The identity $\mcc_2^\La(\mu,\nu)=\om_2'(0,0)$ is established in Corollary \ref{cor:lopt-dual}. This follows $(\mJ_*c_2)(0,0)=\om_2'(0,0)$. 
	
	Now to show the second equality, we note in the global form, any pure jump operators have zero mean, i.e., $m_\mA=m_\mB=0$. We have shown $\ta_2(0,0)=\mcc_2^\La(\mu,\nu)$. The affine structure from Theorem \ref{main1}(i) implies the equality also holds for all $x,y\in \mbr^d$.
\end{proof}

%% file: S5.tex
\section{Strong Duality and Optimal Markovian Couplings for L\'evy Generators} 
\label{sec-dual-gap}

In this section, we establish strong duality and the existence of an optimal Markovian coupling 
for a given pair of full L\'evy generators $\mA,\mB \in \mcg_2^\La(\mbr^d)$. 
In the previous section, we verified these results separately for the drift, diffusion, and jump components. 
Here, we combine those componentwise constructions to handle the full L\'evy generator, 
thereby completing the proof of the general case.

\subsection{Strong duality: proofs of Theorems \ref{main2} and \ref{main-LK}}
Let us start with proving the strong duality Theorem \ref{main2}, that is, 
establishing the equality
\begin{align}\label{eq:om2'-chain}
	\om'_2(x,y;\mA,\mB)=\om^*_2(x,y;\mA,\mB)=\om_2^\pm(x,y;\mA,\mB) = \ta_2(x,y;\mA,\mB)
\end{align}
for all L\'evy generators $\mA, \mB \in \mcg_2^\La(\mbr^d)$ with bounded second moment and $x,y\in \mbr^d$. We remind the readers that the inequality $\om_2'\le \om_2^*\le \om_2^\pm\le \ta_2$ holds, see Proposition \ref{om-ta-prop} (and its proof in Appendix \ref{Appen-genresult}) and Remark \ref{rem:om'-bdd}. 

The main tool in the proof of the strong duality is the following lemma, which establishes the additivity of the functional \( \omega_2'(\cdot, \cdot; \mA, \mB) \) with respect to the Lévy–Khintchine decomposition of the operators \( \mA \) and \( \mB \). 
Due to the technical nature of the proof, we will state the lemma now but postpone its proof to the end of this section.

\begin{lemma}\label{om-add-lem}
	Let $\mA,\mB\in\mcg_2^\La(\mbr^d)$ be in the global form, and $\mA^\bullet,\mB^\bullet$, $\bullet\in\{\nabla,\De,J\}$ be the diffusion, drift, and jump part of $\mA,\mB$ respectively. It holds for all $x_0,y_0\in \mbr^d$: 
	\begin{align*}
		\om_2'(x_0,y_0;\mA,\mB)&= \sum_{\bullet \in \{\nabla,\De,J\}}\om_2'(x_0,y_0;\mA^\bullet,\mB^\bullet).  
	\end{align*}
\end{lemma}

We may now establish the strong duality.

\begin{proof}[Proof of Theorem \ref{main2}]
	As pointed out earlier, the chain of inequalities holds:
	\[
	\om_2'(x,y)\le \omega_2^*(x,y) \le \omega_2^-(x,y)\le \om_2^+(x,y) \le \theta_2(x,y), \qquad \text{for all } (x,y) \in \mathbb{R}^{2d}.
	\]
	To complete the proof, it suffices to show that $\theta_2(x,y) \le \omega_2'(x,y)$. By the affine structure of these functions, namely, Lemma~\ref{om*-linear} and Theorem~\ref{main1}(i), it is enough to verify the identity at the origin:
	\begin{equation} \label{goal1.5}
		\theta_2(0,0) = \omega_2'(0,0).
	\end{equation}
	
	As in Lemma \ref{om-add-lem}, for $\bullet \in\{\nabla,\De,J\}$ denote $\mA^\bullet,\mB^\bullet$ the drift, diffusion and jump part of $\mA,\mB$. 
	Let $\mJ_*^\bullet\in \Ga^\La(\mA^\bullet,\mB^\bullet)$ be an optimal L\'evy coupling generator of $\mA^\bullet,\mB^\bullet$, whose existence is guaranteed by Propositions \ref{dd-opt-prop} and \ref{prop:J-opt}. Specifically it holds for $\bullet = \nabla,\De,J$:
	\begin{align*}
		(\mJ^\bullet_*c_2)(0,0)=\ta_2(0,0;\mA^\bullet,\mB^\bullet). 
	\end{align*}
	Consider the superposition operator $\mJ_* = \mJ_*^\nabla+\mJ_*^\De+\mJ_*^J$.
	It is easy to verify that $\mJ_*$ is a L\'evy coupling generator of $\mA,\mB$, i.e., $\mJ_*\in\Ga^\La(\mA,\mB)$. Hence, by the definition of $\ta_2$, it holds
	\begin{align*}
		\ta_2(0,0)&\le (\mJ_* c_2)(0,0)= \sum_{\bullet \in\{\nabla,\De,J\}} \ta_2(0,0;\mA^\bullet,\mB^\bullet). 
	\end{align*}
	As verify earlier in Corollary \ref{cor:sd-dd} and Proposition \ref{prop:J-opt}, the strong duality $\ta_2=\om_2'$ holds for each part of generators. It then follows from Lemma \ref{om-add-lem} that 
	\begin{align*}
		\ta_2(0,0)\le \sum_{\bullet \in\{\nabla,\De,J\}} \ta_2(0,0;\mA^\bullet,\mB^\bullet)= \sum_{\bullet \in\{\nabla,\De,J\}} \om'_2(0,0;\mA^\bullet,\mB^\bullet)=\om_2'(0,0). 
	\end{align*}
	This is \eqref{goal1.5}. The strong duality is established.
\end{proof}

\begin{proof}[Proof of Theorem \ref{main-LK}]
	It follows immediately from Lemma \ref{om-add-lem} and Theorem \ref{main2}, particularly $\om'_2=\ta_2$. 
\end{proof}

\subsection{Optimal Markovian couplings: Proofs of Theorems \ref{main1}(ii, iii)}

With the strong duality and additivity of $\ta_2$ in place, we may now complete the construction of an optimal Markovian coupling for any given pair $\mA,\mB\in\mcg_2^\La(\mbr^d)$, and the proof of Theorem \ref{main1}(ii). 

\begin{proof}[Proof of Theorem \ref{main1}, (ii)]
	As in the theorem, let $\mA = \Lag(\ka,\al,\mu)$, $\mB = \Lag(\zeta,\be,\nu)$, and 
	$\mJ_* = \Lag(\eta_*,\si_*,\ga_*)$, where the triplet $(\eta_*,\si_*,\ga_*)$ is given as in the statement. 
	Clearly, $\mJ_* \in \Ga^\Lambda(\mA,\mB)$. 
	
	Decompose $\mJ_*$ into its drift, diffusion, and jump components,
	\(
	\mJ_* = \mJ_*^\nabla + \mJ_*^\Delta + \mJ_*^J.
	\)
	By construction of the triplet, each component $\mJ_*^\bullet,\bullet=\nabla,\De,J$, is an optimal Markovian coupling of 
	$\mA^\bullet$ and $\mB^\bullet$. In particular, it holds
	\[
	(\mJ_*^\bullet c_2)(0,0) = \ta_2(0,0;\mA^\bullet,\mB^\bullet).
	\]
	It then follows from Theorem \ref{main-LK} that
	\begin{align*}
		(\mJ_* c_2)(0,0) 
		&= \sum_{\bullet \in \{\nabla,\Delta,J\}} (\mJ_*^\bullet c_2)(0,0) 
		= \sum_{\bullet \in \{\nabla,\Delta,J\}} \ta_2(0,0;\mA^\bullet,\mB^\bullet) 
		= \ta_2(0,0;\mA,\mB). 
	\end{align*}
	By Proposition \ref{0isenough}, $\mJ_*$ is therefore a $2$-optimal Markovian coupling of $\mA$ and $\mB$, 
	which completes the construction.
\end{proof}

We now present the proof of Theorem~\ref{main1}(iii). The next lemma provides a lower bound on the right-hand lower Dini derivative of $t\mapsto (T_t c_2)(x,y)$, which is the only remaining ingredient.

\begin{lemma}\label{lem:dini-lbbd}
	Let $\mA,\mB\in\mcg_2^\La(\mbr^d)$, and let $\om_2^+=\om_2^+(\cdot,\cdot;\mA,\mB)$ denote the upper pointwise transport derivative.  
	Let $\{T_t\}_{t\ge 0}$ be a Markovian coupling semigroup of $\mA,\mB$.  
	For $(D^-T_t c_2)(x,y)$ denoting the right-hand lower Dini derivative of the map $t\mapsto (T_t c_2)(x,y)$, we have
	\[
	(D^- T_t c_2)(x,y)\;\ge\; (T_t\om_2^+)(x,y).
	\]
\end{lemma}

\begin{proof}
	We begin with the case $t=0$. Observe that
	\begin{align}\label{eq:temp5.2}
		D^-\Big|_{t=0} T_t c_2(x,y)
		=\liminf_{h\searrow 0} \frac{T_h c_2(x,y)-c_2(x,y)}{h}
		\;\ge\; \om_2^+(x,y).
	\end{align}
	The proof of \eqref{eq:temp5.2} is identical to the argument establishing $\ta_c \ge \om_c^+$ in Proposition~\ref{om-ta-prop}; see Appendix~\ref{Appen-genresult}.
	
	For general $t\ge 0$, the semigroup property yields
	\begin{align*}
		(D^-T_t c_2)(x,y)
		&=\liminf_{h\searrow 0} \frac{T_{t+h}c_2 - T_t c_2}{h}(x,y) 
		=\liminf_{h\searrow 0} T_t\!\left(\frac{T_h c_2 - c_2}{h}\right)(x,y).
	\end{align*}
	By Fatou's lemma applied with respect to the transition kernel of $T_t$, and using that Proposition~\ref{prop:c2-reg}(i) provides a uniform lower bound in $h$, we obtain
	\[
	(D^-T_t c_2)(x,y)
	\;\ge\; T_t\!\left(\liminf_{h\searrow 0}\frac{T_h c_2 - c_2}{h}\right)(x,y).
	\]
	Finally, combining this with \eqref{eq:temp5.2} and the order-preserving property of Markov semigroups gives the desired inequality.
\end{proof}

\begin{proof}[Proof of Theorem \ref{main1}(iii)]
Fix a Markovian coupling semigroup $\{T_t\}_{t\ge0}$ of $\mA,\mB$.  
By Proposition~\ref{prop:c2-reg}(i), the map $t \mapsto (T_t c_2)(x,y)$ is locally Lipschitz, hence absolutely continuous.  
By the fundamental theorem of calculus for absolutely continuous functions, its derivative exists for almost every $t \ge 0$ and coincides with the lower Dini derivative $(D^-T_t c_2)(x,y)$.  
Invoking Lemma~\ref{lem:dini-lbbd} together with strong duality (namely $\om_2^+=\ta_2$), we obtain
\begin{align*}
	(T_t c_2)(x,y)=c_2(x,y)+\int_0^t (D^-T_sc_2)(x,y)ds \ge c_2(x,y) + \int_0^t (T_s \ta_2)(x,y)\,ds.
\end{align*}

By the affine structure of $\ta_2$ established in Theorem~\ref{main1}(i), and using the marginal property of $\{T_t\}$, we compute for any $t \ge 0$ and $(x,y)\in\mathbb{R}^{2d}$:
\begin{align*}
	(T_s \ta_2)(x,y)
	&= \ta_2(0,0) + (m_\mA - m_\mB)^\top \bigl[(x-y) + s(m_\mA - m_\mB)\bigr].
\end{align*}
Substituting this into the integral yields
\begin{align*}
	(T_t c_2)(x,y)
	&\ge c_2(x,y) + \int_0^t \left[ \ta_2(0,0) + (m_\mA - m_\mB)^\top \bigl[(x-y) + s(m_\mA - m_\mB)\bigr] \right] ds \\
	&= c_2(x,y) + t \bigl[ \ta_2(0,0) + (m_\mA - m_\mB)^\top (x-y) \bigr] + \frac{t^2}{2} |m_\mA - m_\mB|^2.
\end{align*}
Moreover, equality holds when $T_t = e^{t\mJ_*}$ for a $2$-optimal coupling $\mJ_*$, giving precisely the identity \eqref{eq:optgrowth}.  
Consequently, we conclude $T_tc_2\ge e^{t\mJ_*}c_2$. The proof is complete.
\end{proof}

\input{S5.1}

%% file: S5.1.tex
\subsection{Additivity of $\om_2'$ (proof of Lemma \ref{om-add-lem})}

We now return to the proof of Lemma \ref{om-add-lem}, which asserts the additivity of \(\om_2'\) with respect to the L\'evy–Khintchine decomposition of the generators.
The key idea is to observe how each component of the generators interacts with the test functions \(\varphi, \psi\). When \(\mA\) and \(\mB\) are diffusion or drift operators, the expression \(\mA\varphi(x_0) + \mB\psi(y_0)\) depends only on the local behavior of \(\varphi\) and \(\psi\) near \(x_0\) and \(y_0\), respectively. In contrast, when \(\mA\) and \(\mB\) are pure jump operators, the same quantity depends on the global behavior of \(\varphi\) and \(\psi\), due to the nonlocal nature of jumps. This separation of influence—local for drift and diffusion, global for jumps—implies that the contributions from each component can be optimized independently. Consequently, the total value \(\om_2'(x_0, y_0; \mA, \mB)\) decomposes additively across the three types of components.

Although the underlying idea is straightforward, the proof involves some technical subtleties. The strategy is to consider, for each part of the L\'evy–Khintchine decomposition, an optimal (or approximately optimal) pair of test functions. 
The key idea is to combine these pairs in a suitable way to construct an admissible test function pair that approximates a near-maximizer for the full generators \(\mA, \mB\). The main technical difficulty lies in this construction: due to the differing local and nonlocal dependencies of the generator components, several careful modifications are required to ensure that the combined test functions satisfy the admissibility constraints while preserving near-optimality. These modifications form the core of the technical part of the proof.

Given \( \ep > 0 \), define the truncated cost functions \( c_2^\ep \) and the remainder \( \delta_2^\ep \) by
\begin{align*}
	c_2^\ep(x,y) = \f 12\max\cb{0,|x-y|-\ep}^2, \qquad \delta_2^\ep(x,y) = c_2(x,y) - c_2^\ep(x,y). 
\end{align*}
Specifically, \( \delta_2^\varepsilon \) can be computed explicitly: \( \delta_2^\varepsilon(x, y) =  q_2^\varepsilon(|x - y|) \), where $q_2^\ep:[0,\infty)\to[0,\infty)$ is given by
\begin{align}\label{def:q}
	q_2^\varepsilon(r) = \frac{1}{2} r^2 \chi_{\{ r \le \varepsilon \}}(r) + \left( \varepsilon r - \frac{1}{2} \varepsilon^2 \right) \chi_{\{ r > \varepsilon \}}(r).
\end{align}
In particular, \( \delta_2^\varepsilon \) is a \( C^1 \) function such that \( \delta_2^\varepsilon = c_2 \) for all \( |x - y| \le \varepsilon \), and is affine (i.e., linear) for \( |x - y| > \varepsilon \). Clearly we have the decomposition
\begin{align}\label{eq:c_2-decomp}
	c_2(x, y) = c_2^\varepsilon(x, y) + \delta_2^\varepsilon(x, y).
\end{align}

The key idea is to construct a pair of test functions of the form
\[
(\varphi_\varepsilon, \psi_\varepsilon) := (\varphi_\varepsilon^\Delta + \varphi_\varepsilon^J, \, \psi_\varepsilon^\Delta + \psi_\varepsilon^J),
\]
where each pair \( (\varphi_\varepsilon^\bullet, \psi_\varepsilon^\bullet) \) is a near-optimal pair for \( \omega_2'(0, 0; \mathcal{A}^\bullet, \mathcal{B}^\bullet) \), with \( \bullet \in \{ \Delta, J \} \). These are chosen to satisfy
\[
\varphi_\varepsilon^J \oplus \psi_\varepsilon^J \le c_2^\varepsilon, \qquad 
\varphi_\varepsilon^\Delta \oplus \psi_\varepsilon^\Delta \le \delta_2^\varepsilon.  
\]
Summing the two components, we obtain \( \varphi_\varepsilon \oplus \psi_\varepsilon \le c_2 \),
which ensures the admissibility of the combined pair in the dual formulation.

To proceed, we require two approximation lemmas. The first extends Lemma~\ref{lem:denseL1} by constructing a sequence of $C_2^2$-pairs suitable for the jump part. The second handles the diffusion component by building a localized approximation near the origin.

\begin{lemma}\label{lem:approx-C3}
	Let \( \mu, \nu \in \Lambda_2(\mathbb{R}^d) \) be L\'evy measures. For any pair \( (\varphi, \psi) \in L^1(\mu) \oplus L^1(\nu) \) satisfying \( \varphi \oplus \psi \le c_2 \), there exists a family \( \{(\varphi_\varepsilon, \psi_\varepsilon)\}_{\varepsilon > 0} \) such that for all sufficiently small \( \varepsilon > 0 \), the following hold: 
	\begin{enumerate}[label=(\roman*)]
		\item \( \varphi_\varepsilon, \psi_\varepsilon \in C_2^2(\mathbb{R}^d) \), and both vanish in a neighborhood of the origin;
		\item \( \varphi_\varepsilon \oplus \psi_\varepsilon \le c_2^\varepsilon \);
		\item \( \varphi_\varepsilon \to \varphi \) in \( L^1(\mu) \), and \( \psi_\varepsilon \to \psi \) in \( L^1(\nu) \) as \( \varepsilon \searrow 0 \).
	\end{enumerate}
\end{lemma}

\begin{lemma}\label{lem:approx-C4}
	Let \( \mu, \nu \in \Lambda_2(\mathbb{R}^d) \), and suppose \( \varphi, \psi \in C^2_2(\mathbb{R}^d) \) satisfy the quadratic growth condition \( |\varphi(x)|, |\psi(x)| \le C |x|^2 \) for some \( C \ge 0 \), along with \( \varphi \oplus \psi \le c_2 \). Then there exists a family \( \{(\varphi_\varepsilon, \psi_\varepsilon)\}_{\varepsilon > 0} \) such that for all \( \varepsilon > 0 \), the following hold:
	\begin{enumerate}[label=(\roman*)]
		\item \( \varphi_\varepsilon, \psi_\varepsilon \in C_b^2(\mathbb{R}^d) \), and \( \varphi_\varepsilon(x) = \varphi(x) \), \( \psi_\varepsilon(x) = \psi(x) \) for all \( |x| < \frac{\varepsilon}{4} \);
		\item \( \varphi_\varepsilon \oplus \psi_\varepsilon \le \delta_2^\varepsilon \);
		\item \( \varphi_\varepsilon \to 0 \) in \( L^1(\mu) \), and \( \psi_\varepsilon \to 0 \) in \( L^1(\nu) \) as \( \varepsilon \searrow 0 \).
	\end{enumerate}
\end{lemma}

The proof of Lemma~\ref{lem:approx-C3}, similar to that of Lemma~\ref{lem:denseL1}, is technical and thus deferred to the appendix. The proof of Lemma~\ref{lem:approx-C4} is presented below.

\begin{proof}[Proof of Lemma~\ref{lem:approx-C4}]
	Fix a $C^\infty$ function \( \rho : \mathbb{R}^+ \to \mathbb{R}^+ \) such that
	\[
	\rho(r) = r \text{ for } r \in [0, \tfrac{1}{2}], \quad 0 \le \rho(r) \le 1,\quad \rho(\infty)=1, \quad |\rho'(r)| \le 1 \text{ for all } r \ge 0.
	\]
	For each \( \ep > 0 \), define \( \rho_\ep : [0, \infty) \to [0, \infty) \) and the cutoof map $T_\ep:\mbr^d\to\mbr^d$ by
	\[
	\rho_\ep(r) := \ep \rho\left(\tfrac{r}{\ep}\right),\qquad T_\ep(x) := \frac{x}{|x|} \rho_\ep(|x|), \quad \text{with } T_\ep(0) := 0.
	\]
	One can easily verify the following properties:
	\begin{enumerate}
			\item \( T_\ep \) is smooth and 1-Lipschitz on \( \mathbb{R}^d \);
			\item \( T_\ep(x) = x \) for all \( x \in B_{\ep/2}(0) \);
			\item \( T_\ep(x) \in B_\ep(0) \) for all \( x \in \mathbb{R}^d \);
			\item \( T_\ep(x) = 0 \) for all sufficiently large \( |x| \).
		\end{enumerate}

	Following this, given $\vphi,\psi$ from the lemma, let us define for $\ep>0$:
	\[
	\varphi_\varepsilon(x) := \varphi(T_{\f \varepsilon 2}(x)), \qquad \psi_\varepsilon(x) := \psi(T_{\f \varepsilon 2}(x)).
	\]
	We now verify Properties (i)--(iii) holds for the family $\{(\vphi_\ep,\psi_\ep)\}_{\ep>0}$. 
	Property (i) follows directly from the construction.
	To verify (ii), note that \( T_{\f \varepsilon 2} \) is 1-Lipschitz and satisfies \( |T_{\f \varepsilon 2}(x)| \le \f \varepsilon 2 \). Since \( \varphi \oplus \psi \le c_2 \), we estimate:
	\[
	\varphi_\varepsilon(x) + \psi_\varepsilon(y) 
	= \varphi(T_{\f \varepsilon 2}(x)) + \psi(T_{\f \varepsilon 2}(y)) 
	\le \frac{1}{2} |T_{\f \varepsilon 2}(x) - T_{\f \varepsilon 2}(y)|^2 
	\le \tilde{q}_2^\varepsilon(|x - y|),
	\]
	where \( \tilde{q}_2^\varepsilon(r) := \frac{1}{2} \min\{r, \varepsilon\}^2 \). Since \( \tilde{q}_2^\varepsilon \le q_2^\varepsilon \) from \eqref{def:q}, it follows that \( \varphi_\varepsilon \oplus \psi_\varepsilon \le \delta_2^\varepsilon \), as desired.
	
	For (iii), observe that since \( T_{\f \varepsilon 2} \) is 1-Lipschitz and satisfies \( |T_{\f \varepsilon 2}(x)| \le \f \varepsilon 2 \), we have
	\[
	|\varphi_\varepsilon(x)| = |\varphi(T_{\f \varepsilon 2}(x))| \le C |T_{\f \varepsilon 2}(x)|^2 \le C \min\cb{|x|, \f \varepsilon 2}^2 \le C |x|^2.
	\]
	As \( |x|^2 \in L^1(\mu) \), the dominated convergence theorem implies \( \varphi_\varepsilon \to 0 \) in \( L^1(\mu) \). An identical argument shows \( \psi_\varepsilon \to 0 \) in \( L^1(\nu) \).
\end{proof}

We may now proceed to prove Lemma \ref{om-add-lem}.

\begin{proof}[Proof of Lemma~\ref{om-add-lem}]
	For notational simplicity, we suppress the subscript $2$ and the generator arguments $\mA,\mB$ in \( \om'=\omega_2' \) and write
	\[
	\omega_\bullet'(x, y) := \omega_2'(x, y; \mathcal{A}^\bullet, \mathcal{B}^\bullet), \qquad \bullet \in \{ \Delta, \nabla, J \}.
	\]
	By Proposition~\ref{om*-linear}, it suffices to verify the identity at the origin, i.e., \( x = y = 0 \). Since \( \omega_\nabla'(0, 0) = 0 \) by Corollary \ref{cor:sd-dd}, the goal reduces to proving
	\[
	\omega'(0, 0) = \omega_\Delta'(0, 0) + \omega_J'(0, 0),
	\]
	where we set \( \omega' := \omega_2'(0, 0; \mathcal{A}, \mathcal{B}) \).
	The inequality \( \omega'(0,0) \le \omega_\Delta'(0,0) + \omega_J'(0,0) \) follows immediately from the definition and subadditivity of the supremum. We now prove the reverse inequality. Fix \( \eta > 0 \), and we will show that
	\begin{equation}\label{eq:goal4}
		\omega'(0, 0) \ge \omega_\Delta'(0, 0) + \omega_J'(0, 0) - \eta.
	\end{equation}
	
	Let us introduce the shorthand for $\vphi,\psi\in C_2^2(\mbr^d)$:
	\begin{align*}
		\mathcal{K}(\varphi, \psi) := \mathcal{A} \varphi(0) + \mathcal{B} \psi(0),\qquad \mathcal{K}_\bullet(\varphi, \psi) := \mathcal{A}^\bullet \varphi(0) + \mathcal{B}^\bullet \psi(0)\mbox{ for }\bullet \in \{ \nabla, \Delta, J \}.
	\end{align*}
	Clearly, \( \mathcal{K} = \mathcal{K}_\nabla + \mathcal{K}_\Delta + \mathcal{K}_J \), and by definition
	\[
	\omega_\bullet' = \sup_{(\varphi, \psi) \in \mathcal{D}'(0, 0)} \mathcal{K}_\bullet(\varphi, \psi).
	\]
	We first make a key observation. For any \( (\varphi, \psi) \in \mathcal{D}'(0, 0) \), we may normalize by adding a constant so that \( \varphi(0) = \psi(0) = 0 \). By the same computation from \eqref{eq:temp3.11}, we have
	\begin{equation}\label{eq:Kj-prop}
		\mathcal{K}_J(\varphi, \psi)= \mA^J\varphi(0)+\mB^J\psi(0) = \int_{\mathbb{R}^d} \varphi(x) \, d\mu(x) + \int_{\mathbb{R}^d} \psi(y) \, d\nu(y).
	\end{equation}
	
	Let \( (\varphi^\bullet, \psi^\bullet) \in \mathcal{D}'(0, 0) \), \( \bullet \in \{ \Delta, J \} \), be such that
	\begin{equation}\label{eq:bound2}
		\mathcal{K}_\bullet(\varphi^\bullet, \psi^\bullet) \ge \omega_\bullet'(0, 0) - \frac{\eta}{2}.
	\end{equation}
	Let \( \{ (\varphi^\bullet_\varepsilon, \psi^\bullet_\varepsilon) \}_{\varepsilon > 0} \) be the approximating families constructed in Lemmas~\ref{lem:approx-C3} and~\ref{lem:approx-C4}, respectively. 
	For the jump part, the pair \( (\varphi^J_\varepsilon, \psi^J_\varepsilon) \in \mathcal{D}'(0, 0) \) vanishes near the origin and converges to \( (\varphi^J, \psi^J) \) in \( L^1(\mu) \oplus L^1(\nu) \). By \eqref{eq:Kj-prop}, we conclude:
	\begin{align}\label{eq:tem3}
		\mathcal{K}_\nabla(\varphi^J, \psi^J) = \mathcal{K}_\Delta(\varphi^J, \psi^J) = 0, \quad \text{and} \quad \lim_{\varepsilon \to 0} \mathcal{K}_J(\varphi^J_\varepsilon, \psi^J_\varepsilon) = \mathcal{K}_J(\varphi^J, \psi^J).	
	\end{align}
	For the diffusion part, the approximants satisfy
	\begin{align*}
		(\varphi^\Delta_\varepsilon, \psi^\Delta_\varepsilon) = (\varphi^\Delta, \psi^\Delta) \quad \text{on } B_{\varepsilon/4}(0), \quad \text{and} \quad (\varphi^\Delta_\varepsilon, \psi^\Delta_\varepsilon) \to (0, 0) \text{ in } L^1(\mu) \oplus L^1(\nu).
	\end{align*}
	Since \( \mathcal{A}^\Delta \), \( \mathcal{B}^\Delta \) are local operators, we have
	\begin{align}\label{eq:tem4}
		\mathcal{K}_\nabla(\varphi^\Delta_\varepsilon, \psi^\Delta_\varepsilon) = 0, \quad \mathcal{K}_\Delta(\varphi^\Delta_\varepsilon, \psi^\Delta_\varepsilon) = \mathcal{K}_\Delta(\varphi^\Delta, \psi^\Delta), \quad \text{and} \quad \lim_{\varepsilon \to 0} \mathcal{K}_J(\varphi^\Delta_\varepsilon, \psi^\Delta_\varepsilon) = 0.	
	\end{align}
	
	Now define the combined test pair:
	\[
	\varphi_\varepsilon := \varphi^\Delta_\varepsilon + \varphi^J_\varepsilon, \qquad \psi_\varepsilon := \psi^\Delta_\varepsilon + \psi^J_\varepsilon.
	\]
	By construction and Lemmas~\ref{lem:approx-C3} and~\ref{lem:approx-C4}, we have
	\[
	\varphi_\varepsilon \oplus \psi_\varepsilon \le \delta_2^\varepsilon + c_2^\varepsilon = c_2,
	\]
	and \( \varphi_\varepsilon(0) + \psi_\varepsilon(0) = 0 \), so \( (\varphi_\varepsilon, \psi_\varepsilon) \in \mathcal{D}'(0, 0) \).
	Using \eqref{eq:tem3} and \eqref{eq:tem4}, we compute
	\begin{align*}
		\mathcal{K}(\varphi_\varepsilon, \psi_\varepsilon)
		&= \sum_{\bullet \in \{ \nabla, \Delta, J \}} \sum_{\ast \in \{ \Delta, J \}} \mathcal{K}_\bullet(\varphi^\ast_\varepsilon, \psi^\ast_\varepsilon) \\
		&= \mathcal{K}_\Delta(\varphi^\Delta, \psi^\Delta) + \mathcal{K}_J(\varphi^J_\varepsilon, \psi^J_\varepsilon) + \mathcal{K}_J(\varphi^\Delta_\varepsilon, \psi^\Delta_\varepsilon).
	\end{align*}
	Passing to the limit as \( \varepsilon \searrow 0 \) and using \eqref{eq:tem3}, \eqref{eq:tem4} and \eqref{eq:bound2}, we obtain
	\[
	\omega'(0, 0) \ge \limsup_{\varepsilon \to 0} \mathcal{K}(\varphi_\varepsilon, \psi_\varepsilon) \ge \omega_\Delta'(0, 0) + \omega_J'(0, 0) - \eta,
	\]
	which proves \eqref{eq:goal4}. Since \( \eta > 0 \) was arbitrary, the result follows.
\end{proof}

%% file: S6.tex
\section{Wasserstein-type Metric on the Space $\mcg_2^\La(\mathbb{R}^d)$ and $\La_2(\mathbb{R}^d)$} \label{sec-wass-metric}

With the groundwork laid in the preceding sections, we now turn to the metric on the space $\mcg_2^\Lambda(\mathbb{R}^d)$ of L\'evy-type generators with finite second moments, induced by optimal couplings. The metric functional is given in \eqref{def:Wg}, namely  
\begin{align}\label{def:W_gmetric}
	\mcW_\mcg(\mA, \mB)^2 := \tfrac{1}{2} |m_\mA - m_\mB|^2 + \theta_2(0,0; \mA, \mB),
\end{align}
where $m_\mA, m_\mB \in \mathbb{R}^d$ denote the mean vectors associated with the generators $\mA$ and $\mB$. 
We will call this metric \emph{Wasserstein generator metric} on $\mcg_2^\La(\mbr^d)$. 
Since Theorem~\ref{main2} established the identity $\omega_2^* = \omega_2 = \theta_2$, the term $\theta_2$ in \eqref{def:W_gmetric} can be equivalently replaced by any of these expressions.  

The contribution of $\theta_2(0,0;\mA,\mB)$ depends only on the \emph{centered parts} of $\mA$ and $\mB$. Indeed, writing  
\[
\mA = \bar \mA + m_{\mA}\cdot \nabla, \qquad \mB = \bar\mB+m_{\mB}\cdot \nabla,
\]  
we see that $\bar \mA$ and $\bar\mB$ have vanishing mean vectors, and hence are referred to as centered. In particular, if $\mA,\mB$ are given in the global L\'evy--Khintchine form, the centered generator consists precisely of the diffusion and jump parts, i.e.\ $\bar \mA=\mA^\Delta+\mA^J$.  

By Theorem~\ref{main-LK}, the term $\theta_2$ decomposes naturally into drift, diffusion, and jump contributions. Moreover, Lemma~\ref{cor:sd-dd} shows that the drift contribution (evaluating at $x=y=0$) vanishes. The remaining contributions correspond to the squared Bures--Wasserstein distance in the diffusion part, and to the L\'evy optimal transport cost $(\mcW_\Lambda)^2 := \mcc_2^\Lambda$ (see Definition~\ref{levy-opt-prob}) in the jump part. Consequently, we arrive at the explicit expression  
\begin{align}\label{eq:wass-LK}
	\mcW_\mcg(\mA, \mB)^2 \;=\; \tfrac{1}{2} |\ka-\zeta|^2 \;+\; \mcW_\mcS(\alpha, \beta)^2 \;+\; \mcW_\Lambda(\mu, \nu)^2,
\end{align}
where $\mA=\Lag(\ka,\alpha,\mu)$ and $\mB=\Lag(\zeta,\beta,\nu)$.  The metric $\mcW_\La$ on $\La_2(\mbr^d)$ will be called \emph{L\'evy-Wasserstein metric}.

In the subsequent subsections, we establish several fundamental properties of the metric $\mcW_\mcg$ and $\mcW_\La$, 
contained in Theorems \ref{main3}, \ref{main5}, and Corollary \ref{cor:main4}, including:  
\begin{itemize}
	\item that $\mcW_\mcg,\mcW_\La$ define metrics;
	\item lower semicontinuity and weak compactness;
	\item completeness;
	\item separability;
	\item a maximal-type bound on the path space.
\end{itemize}

\subsection{Wasserstein generator metric}

We begin by establishing that the functional $\mcW_\mcg$ indeed defines a metric on the space $\mcg_2^\La(\mathbb{R}^d)$.  
This fact is a direct consequence of the strong duality result in Theorem~\ref{main2}.

\begin{proof}[Proof of Theorem \ref{main3}, $\mcW_\mcg$ is a metric]
	We first show that the functional is symmetric in $(\mA,\mB)$. By \eqref{def:W_gmetric} and $\ta_2=\om_2^+$, it suffices to verify that
	\(
	\omega_2^+(0,0;\mA,\mB) = \omega_2^+(0,0;\mB,\mA).
	\)
	Indeed, from the definition of the pointwise transport derivative, 
	\[
	\omega^+_2(0,0;\mA,\mB) = \limsup_{t \searrow 0} \frac{1}{t}\mcc_2(\delta_0 e^{t\mA}, \delta_0 e^{t\mB}) = \limsup_{t \searrow 0} \frac 1t{\mcc_2(\delta_0 e^{t\mB}, \delta_0 e^{t\mA})} = \omega^+_2(0,0;\mB,\mA).
	\]
	Hence, $\mcW_\mcg^\La$ is symmetric and clearly nonnegative.
	
	We now show that $\mcW_\mcg(\mA,\mB) = 0$ if and only if $\mA = \mB$.
	The reverse direction is immediate, since if $\mA = \mB$ then $m_\mA = m_\mB$, and
	\[
	\omega^+_2(0,0;\mA,\mA) = \limsup_{t \searrow 0} \frac{1}{t} \mcc_2(\delta_0 e^{t\mA}, \delta_0 e^{t\mA}) = 0.
	\]
	Conversely, suppose $\mcW_\mcg(\mA,\mB) = 0$. Then $m_\mA = m_\mB$ and $\ta_2(0,0;\mA,\mB) = 0$. By Theorem \ref{main1}(i), this implies $\ta_2(x,x;\mA,\mB) = 0$ for all $x\in \mbr^d$. 
	From Theorem \ref{main1}(iii), any 2-optimal Lévy coupling generator $\mJ_* \in \Gamma^\La(\mA,\mB)$ satisfies
	\(
	e^{t\mJ_*} c_2(x,x) = c_2(x,x) = 0.
	\)
	Since $\delta_{(x,x)} e^{t\mJ_*}$ is a coupling of $\delta_x e^{t\mA}$ and $\delta_x e^{t\mB}$, we obtain
	\[
	\mcc_2(\delta_x e^{t\mA}, \delta_x e^{t\mB}) \le \inn{\de_{(x,x)}e^{t\mJ_*},c_2}= (e^{t\mJ_*}c_2)(x,x)=0. 
	\]
	Hence, for all $x \in \mathbb{R}^d$ and $t \ge 0$, we have $\delta_x e^{t\mA} = \delta_x e^{t\mB}$, which implies $e^{t\mA} = e^{t\mB}$ and therefore $\mA = \mB$.
	
	It remains to show that $\mcW_\mcg$ satisfies the triangle inequality. By \eqref{def:W_gmetric}, it suffices to verify that the map $(\mA,\mB) \mapsto \omega^+_2(0,0;\mA,\mB)^{1/2}$ satisfies the triangle inequality.
	Take any $\mA, \mB, \mC \in \mcg_2^\La(\mathbb{R}^d)$. Then using the triangle inequality for Wasserstein-2 metric, we have
	\begin{align*}
		\omega^+_2(0,0;\mA,\mC)^{1/2} &= \limsup_{t \searrow 0}t^{-1/2} \mcW_2(\delta_0 e^{t\mA}, \delta_0 e^{t\mC}) \\
		&\le \limsup_{t \searrow 0} t^{-1/2}[\mcW_2(\delta_0 e^{t\mA}, \delta_0 e^{t\mB}) + \mcW_2(\delta_0 e^{t\mB}, \delta_0 e^{t\mC})] \\
		&\le \omega^+_2(0,0;\mA,\mB)^{1/2} + \omega^+_2(0,0;\mB,\mC)^{1/2}.
	\end{align*}
	This completes the proof.
\end{proof}

The following result provides an alternative characterization of the Wasserstein generator metric $\mcW_\mcg$.

\begin{proposition}\label{prop:WG-equiv}
	Let $\mA,\mB \in \mcg_2^\La(\mbr^d)$, and $\mJ_*$ be an optimal L\'evy coupling generator of $\mA,\mB$. One has
	\[
	\mcW_\mcg(\mA,\mB) = e^{\mJ_*}c_2(0,0)=\inf_{\mJ \in \Ga^\La(\mA,\mB)} \, e^{\mJ} c_2(0,0).
	\]
\end{proposition}

\begin{proof}
	Let $\mJ_*\in \La(\mA,\mB)$ be a 2-optimal coupling of $\mA,\mB$. By Theorem \ref{main1}(iii)
	\begin{align*}
		(e^{\mJ_*}c_2)(0,0)= \ta_2(0,0;\mA,\mB)+ \f 12|m_\mA-m_\mB|^2 =\mcl{W}_\mcg(\mA,\mB).
	\end{align*}
	From the same theorem, $e^{t\mJ_*}c_2\le  e^{t\mJ}c_2$ for all $\mJ\in \Ga(\mA,\mB)$ and $t\ge 0$. This implies the desired equality. 
\end{proof}

Before proceeding, we state the following proposition, which will serve as a useful tool for subsequent results. Its proof follows directly from \eqref{eq:optgrowth} in Theorem~\ref{main1}.

\begin{proposition}\label{prop:Wg-bdd}
	Let $\mA,\mB\in \mcg_2^\La(\mbr^d)$. Then it holds for all $x,y\in\mbr^d$ and $t\ge 0$:
	\begin{align*}
		\mcW_2(\de_x e^{t\mA},\de_y e^{t\mB})^2\le \f 12 |x-y|^2+t[\ta(0,0;\mA,\mB)+(x-y)^\top(m_\mA-m_\mB)] +\f 12 t^2 |m_\mA-m_\mB|^2.
	\end{align*}
	Specifically, when $x=y$,
	\begin{align*}
		\mcW_2(\de_x e^{t\mA},\de_x e^{t\mB})^2 \le t\ta_2(0,0;\mA,\mB)+ \f 12 t^2 |m_\mA-m_\mB|^2\le \max\cb{t, t^2}\mcW_\mcg(\mA,\mB)^2. 
	\end{align*}
\end{proposition}

\subsection{Lower semicontinuity, weak compactness and completeness}

Since the metric $\mcW_\mcg$ is naturally induced by the Wasserstein-2 metric, it is natural to investigate compactness properties in analogy with those observed in the Wasserstein-2 setting. In particular, we aim to show that every bounded subset $G\subset (\mcg_2^\La(\mathbb{R}^d), \mcW_\mcg)$ is relatively compact with respect to a suitable weaker topology.
By boundedness, we mean there is $\mB\in \mcg_2^\La(\mbr^d)$, or equivalently $\mB=0$, and a constant $R> 0$, such that 
\begin{align*}
	\sup_{\mA\in G}\mcW_\mcg(\mA,\mB)< R. 
\end{align*}
Let us provide the following notion of convergence of probability generators.
\begin{definition}
	For $n\ge 1$, let \( \mathcal{A},\mcA_n \in \mathcal{G}(\Pi) \).  
	We say that \( \mathcal{A}_n \to \mathcal{A} \) \emph{weakly in the sense of transition kernels}, or simply \emph{weakly}, if, for all \( x \in \Pi \) and \( t \geq 0 \),
	\[
	\delta_x e^{t \mathcal{A}_n} \to \delta_x e^{t \mathcal{A}} \quad \text{weakly in } \mathcal{P}(\Pi) \text{}.
	\]
\end{definition}

\begin{remark}\label{rem:weakcon-levy}
	If $\mA_n,\mA$ are L\'evy generators, then $\mA_n \to \mA$ weakly in the sense of transition kernels if and only if 
	\(
	\delta_0 e^{\mA_n} \to \delta_0 e^{\mA},
	\)
	weakly.
	This equivalence follows from the translation invariance and infinite divisibility of L\'evy measures. 
	Indeed, weak convergence at time \(t=1\) already guarantees convergence for all \(t \ge 0\). 
	To see this, recall that the characteristic function of a L\'evy process at time \(t\) is of the form
	\(
	\exp(\,t\psi(\xi)\,),
	\)
	where \(\psi\) is the L\'evy–Khintchine exponent. 
	If \(\delta_0 e^{\mA_n} \to \delta_0 e^{\mA}\) at \(t=1\), then by L\'evy’s continuity theorem we have pointwise convergence of the characteristic functions 
	\(\exp(\psi_n(\xi)) \to \exp(\psi(\xi))\).
	This in turn implies
	\(
	\exp(t\psi_n(\xi)) \to \exp(t\psi(\xi))
	\)
	for every \(t\ge 0\), which yields weak convergence of $\delta_0 e^{t\mA_n}$ to $\delta_0 e^{t\mA}$ for all \(t\).
\end{remark}

Let us now establish the weak compactness of for every bounded sequence in $\mcg_2^\La(\mbr^d)$. 

\begin{proposition}\label{prop:weak-comp}
	Let \( \{ \mathcal{A}_n \}_n \) be a bounded sequence in \( (\mathcal{G}_2^\Lambda(\mathbb{R}^d), \mathcal{W}_\mathcal{G}) \). Then there exists a subsequence \( \{ \mathcal{A}_{n_k} \}_k \) and \( \mathcal{A} \in \mathcal{G}_2^\Lambda(\mathbb{R}^d) \) such that \( \mathcal{A}_{n_k} \to \mathcal{A} \) weakly in the sense of transition kernels.
\end{proposition}

\begin{proof}
	Since the sequence is bounded, we may assume without loss of generality that
	\[
	\{ \mathcal{A}_n \}_n \subset B_R(0) 
	:= \Big\{ \mathcal{A} \in \mathcal{G}_2^\Lambda(\mathbb{R}^d) : \mathcal{W}_\mathcal{G}(\mathcal{A},0) < R \Big\},
	\]
	where \( \mathcal{A}=0 \) denotes the trivial generator (with \( e^{t\mathcal{A}} = I \)).  
	
	Let \( \rho_t^{(n)} := \delta_0 e^{t \mathcal{A}_n} \) be the law at time \(t\) of the process generated by \( \mathcal{A}_n \). By Proposition~\ref{prop:Wg-bdd},
	\begin{align}\label{bdd:genbdd}
		\mathcal{W}_2(\rho_t^{(n)}, \delta_0)^2 
		\;\leq\; \max\{t, t^2\} \, \mathcal{W}_\mathcal{G}(\mathcal{A}_n,0)^2 
		\;\leq\; \max\{t, t^2\} R^2.	
	\end{align}
	Hence the second moments \( \int_{\mathbb{R}^d} |x|^2 \, d\rho_t^{(n)}(x) \) are uniformly bounded (see Proposition~\ref{prop-prelim}(v)).  
	Fix \(t=1\). By Prokhorov’s theorem, there exists a subsequence \( \{ \rho_1^{(n_k)} \}_k \) converging weakly to some probability measure \( \rho_1 \in \mathcal{P}_2(\mathbb{R}^d) \).  Note the weak limit $\rho_1$ has finite second moment due to Portmanteau lemma and the uniform bound of second moment of the sequence. 
	
	It is well known that the class of infinitely divisible measures is closed under weak convergence, and every such measure corresponds to the distribution of a L\'evy process. Therefore, there exists \( \mathcal{A} \in \mathcal{G}^\Lambda(\mathbb{R}^d) \) such that  
	\(
	\rho_1 = \delta_0 e^{\mathcal{A}}.
	\)
	Moreover, since \( \rho_1 \in \mathcal{P}_2(\mathbb{R}^d) \), the associated generator \( \mathcal{A} \) necessarily belongs to \( \mathcal{G}_2^\Lambda(\mathbb{R}^d) \). Thus we have $\rho_t^{(n_k)}\to \rho_t$ weakly for $t=1$. By Remark \ref{rem:weakcon-levy}, the weak convergence holds for all $t\ge 0$, that is,
	\[
	\rho_t^{(n_k)}=\delta_0 e^{t \mathcal{A}_{n_k}} \to \delta_0 e^{t \mathcal{A}}=\rho_t
	\quad \text{for all } t \ge 0.
	\]
	Finally, by translation invariance of Lévy semigroups, the same convergence holds starting from any \( x \in \mathbb{R}^d \). Thus \( \mathcal{A}_{n_k} \to \mathcal{A} \) weakly in the sense of transition kernels.
\end{proof}

The coming proposition states the lower semicontinuity of the metric w.r.t. the weak topology on generators introduced earlier.

\begin{proposition}\label{prop:wglsc}
	Let $\mA,\mB\in \mcg_2^\La(\mbr^d)$ and $\{\mA_n,\mB_n\}_n$ be bounded in $\wg$ metric. Suppose
	$\mA_n\to \mA,\mB_n\to\mB$ weakly in the sense of transition kernels. It holds
	\begin{align*}
		\mcW_\mcg(\mA,\mB)&\le \liminf_{n\to\infty}\mcW_\mcg(\mA_n,\mB_n).
	\end{align*}
\end{proposition}

\begin{proof}
	Let us denote $\kappa_n = m_{\mathcal{A}_n}$, $\zeta_n = m_{\mathcal{B}_n}$, $\kappa = m_{\mathcal{A}}$, $\zeta = m_{\mathcal{B}}$, and define $\rho_t^{(n)} := \delta_0 e^{t\mathcal{A}_n}$, $\sigma_t^{(n)} := \delta_0 e^{t\mathcal{B}_n}$, $\rho_t := \delta_0 e^{t\mathcal{A}}$, and $\sigma_t := \delta_0 e^{t\mathcal{B}}$.
	
	Since $\{\mathcal{A}_n, \mathcal{B}_n\}_n$ is bounded in $\mcg_2^\La(\mbr^d)$, the same computation from \eqref{bdd:genbdd} gives: for some $R\ge 0$,
	\[
	\frac{1}{2} \int_{\mathbb{R}^d} |x|^2 \, d\rho_t^{(n)}(x) = \mathcal{W}_2(\rho_t^{(n)}, \delta_0)^2 \leq \max\{t, t^2\} R^2.
	\]
	The same bound applies for $\sigma_t^{(n)}$ in place of $\rho_t^{(n)}$. This establishes a pointwise (in time) uniform second moment bound for the families $\{\rho_t^{(n)}\}_n$ and $\{\sigma_t^{(n)}\}_n$. By Jensen's inequality, this also implies that the sequence of first moments $\{\kappa_n, \zeta_n\}_n$ is uniformly bounded in $\mathbb{R}^d$.
	
	Now, again using Proposition~\ref{prop:Wg-bdd}, we have for all $t \geq 0$,
	\[
	\mathcal{C}_2(\rho_t^{(n)}, \sigma_t^{(n)}) \leq t \theta_2(0,0; \mathcal{A}_n, \mathcal{B}_n) + \frac{1}{2} t^2 |\kappa_n - \zeta_n|^2 \leq t \theta_2(0,0; \mathcal{A}_n, \mathcal{B}_n) + C t^2.
	\]
	Since $\mathcal{A}_n \to \mathcal{A}$ and $\mathcal{B}_n \to \mathcal{B}$ weakly in the sense of transition kernels, we have weak convergence $\rho_t^{(n)} \to \rho_t$ and $\sigma_t^{(n)} \to \sigma_t$ weakly as $n\to\infty$. By the lower semicontinuity of the quadratic transport cost $\mathcal{C}_2$, it follows that
	\[
	\mathcal{C}_2(\rho_t, \sigma_t) \leq \liminf_{n \to \infty} \mathcal{C}_2(\rho_t^{(n)}, \sigma_t^{(n)}) \leq t \liminf_{n \to \infty} \theta_2(0,0; \mathcal{A}_n, \mathcal{B}_n) + C t^2.
	\]
	Dividing both sides by $t$ and letting $t \searrow 0$, we obtain
	\[
	\omega_2^+(0,0; \mathcal{A}, \mathcal{B}) = \theta_2(0,0; \mathcal{A}, \mathcal{B}) \leq \liminf_{n \to \infty} \theta_2(0,0; \mathcal{A}_n, \mathcal{B}_n).
	\]
	
	Next, we claim that $\kappa_n \to \kappa$ and $\zeta_n \to \zeta$ in $\mathbb{R}^d$. Indeed, since $\kappa_n = \langle \rho_1^{(n)}, x \rangle$ and $\rho_1^{(n)} \to \rho_1$ weakly with uniform second moment bounds, it follows that
	\[
	\kappa_n = \int_{\mathbb{R}^d} x \, d\rho_1^{(n)}(x) \to \int_{\mathbb{R}^d} x \, d\rho_1(x) = \kappa.
	\]
	An analogous argument shows $\zeta_n \to \zeta$.
	
	Combining the convergence $(\kappa_n, \zeta_n) \to (\kappa, \zeta)$ with the lower semicontinuity result above, we conclude that
	\begin{align*}
		\wg(\mA,\mB)^2&= \f 12 |\ka-\zeta|^2 +\om_2(0,0;\mA,\mB)\\
		&\le \liminf_{n\to\infty}\sqb{\f 12 |\ka_n-\zeta_n|^2 +\ta_2(0,0;\mA_n,\mB_n)}=\liminf_{n\to\infty} \wg(\mA_n,\mB_n)^2. \qedhere
	\end{align*}
\end{proof}

Now the completeness of the Wasserstein generator metric $\wg$ follows easily from the two results above.
\begin{proof}[Proof of Theorem \ref{main3}, completeness]
	Let $\{\mA_n\}_n \subset \mcg_2^\La(\mathbb{R}^d)$ be a Cauchy sequence with respect to the metric $\wg$. Since $\{\mA_n\}_n$ is bounded, by the weak compactness established in Proposition \ref{prop:weak-comp}, along a subsequence we have $\mA_{n_k}\to \mA$ weakly in the sense of transition kernels, where $\mA\in\mcg_2^\La(\mbr^d)$. It remains to show $\mA_n\to\mA$ in the metric $\wg$. 
	
	Let us first prove $\mA_n\to \mA$ weakly in the sense of transition kernel in the full sequence. It suffices to prove: for every subsequence $\{\mA_n\}_n$ there is a further subsequence that converges to $\mA$ weakly. Indeed, for any subsequence of $\{\mA_n\}_n$ by the weak compactness in Proposition \ref{prop:weak-comp} there is a further subsequence, indexed by $\{\mA_{m_k}\}_k$ that converges weakly to some $\mA'\in\mcg_2^\La(\mbr^d)$. Since $\mA_{n_k}\to\mA$ along the subsequence $\{n_k\}$, by the lower semicontinuity of the metric (Proposition \ref{prop:wglsc}) we have
	\begin{align*}
		\wg(\mA,\mA')&\le \liminf_{n\to\infty} \wg(\mA_{m_k},\mA_{n_k})=0. 
	\end{align*}
	The limit is zero because $\{\mA_n\}_n$  is Cauchy in the metric $\wg$. This shows the weak convergence of the full sequence $\{\mA_n\}_n$ to $\mA$. 
	
	Finally we show $\wg(\mA,\mA_n)\to 0$. Since $\mA_m\to\mA$ weakly, the lower semicontinuity again implies
	\begin{align*}
		\wg(\mA,\mA_n)&\le \liminf_{m\to\infty}\wg(\mA_m,\mA_n)\le \sup_{m\ge n}\wg(\mA_m,\mA_n). 
	\end{align*}
	Since $\{\mA_n\}_n$ is Cauchy, the right hand side vanishes as $n\to\infty$. Hence it follows $\wg(\mA,\mA_n)\to 0$. The proof is complete. 
\end{proof}

\subsection{L\'evy-Wasserstein metric and its topological properties}
Let us now restrict our attention to the jump part of generators. In which case the generator metric reduces to a metric on the space $\La_2(\mbr^d)$ of L\'evy measures with finite second moment, which will be called \emph{L\'evy-Wasserstein metric}, denoted as $\mcW_\La$. The metric property of $\mcW_\La$ follows easily from Theorem \ref{main3}.

\begin{proof}[Proof of Corollary \ref{cor:main4}]
	Let $\mu,\nu \in \La_2(\mbr^d)$, and set $\mA = \Lag(0,0,\mu)$ and $\mB = \Lag(0,0,\nu)$. 
	By Proposition \ref{prop:J-opt}, we have
	\[
	\mcW_\La(\mu,\nu)^2 
	= \mcc_2^\La(\mu,\nu) 
	= \ta_2(0,0;\mA,\mB) 
	= \mcW_\mcg(\mA,\mB)^2.
	\]
	Since $\mcW_\mcg$ defines a metric, it follows that $\mcW_\La$ also defines a metric.
\end{proof}

\begin{remark}
	In the classical theory of Wasserstein metrics, the triangle inequality is usually established by constructing a coupling of three probability measures, often referred to as the \emph{gluing lemma}. 
	By contrast, in the present Lévy setting, such a gluing construction is less transparent and technically more involved.
	Instead, our approach exploits the strong duality established earlier, which allows us to derive the triangle inequality through analytic arguments on generators rather than couplings. 
\end{remark}

We next establish the following equivalent characterization of convergence in the metric \( \mcW_\Lambda \), analogous to the standard characterization of convergence in the Wasserstein-2 metric. We recall the reader the notion of $\La$-weak convergence of L\'evy measures from Section \ref{sec-La-conv}, see also Appendix \ref{appen-weaktop} for the related results.  

\begin{theorem}\label{thm:La2-conv}
	For $n\ge1$, let $\mu,\mu_n\in\La_2(\mbr^d)$. The following are equivalent: 
	
	(i) $\mcW_\La(\mu_n,\mu)\to 0$; 
	
	(ii) for all Lipschitz functions $\vphi\in C(\mbr^d)$ satisfying $|\nabla \vphi(x)|\le C|x|$ Lebesgue-a.e. for some $C\ge 0$, it holds
	\begin{align*}
		\lim_{n\to\infty}\int_{\mbr^d} \vphi d\mu_n=\int_{\mbr^d} \vphi d\mu;
	\end{align*}
	
	(iii) $\mu_n\to \mu$ $\La$-weakly, and $$\lim_{n\to\infty}\int_{\mbr^d} |x|^2 d\mu_n= \int_{\mbr^d} |x|^2 d\mu.$$
\end{theorem}

\begin{proof}
	(i) $\to$ (ii). 
	First a bounded Lipschitz functions $\vphi$ as in (ii). 
	Using the fundamental theorem of calculus for absolutely continuous functions and the bound for gradient, we have
	\begin{align*}
		|\varphi(x)-\varphi(y)|&= \abs{\int_0^1 (x-y)\cdot \nabla \varphi(sx+(1-s)y)ds}\\
		&\le C\int_0^1 |x-y|(s|x|+(1-s)|y|)ds \le C|x-y|(|x|+|y|).
	\end{align*}
	
	For $n\ge 1$, let \( \gamma_n \in \Gamma^\La(\mu_n, \mu) \) be an optimal coupling realizing \( \mcW_\La(\mu_n,\mu)^2 = \int c_2d\ga_n \).
	Using the Cauchy-Schwartz inequality, $(a+b)^2 \le 2(a^2+b^2)$, and the bound above we find
	\begin{align*}
		\abs{\int \varphi d\mu_n - \int \varphi d\mu}&= \abs{\int_{\mbr^{2d}} [\varphi(x)-\varphi(y)]d\ga_n(x,y)} \\
		&\le C\int_{\mbr^{2d}} |x-y|(|x|+|y|)d\ga_n(x,y) \\
		&\le C\sqb{\int_{\mbr^{2d}} |x-y|^2 d\ga_n(x,y)}^{1/2}\cdot\sqrt 2 \sqb{\int_{\mbr^{2d}}(|x|^2+|y|^2)d\ga_n(x,y) }^{1/2}\\
		&= 2 C\mcW_\La(\mu_n,\mu)\sqb{\int_{\mbr^d}|x|^2 d\mu_n(y)+ \int_{\mbr^d} |y|^2 d\mu(y)} ^2.
	\end{align*}
	Since $\{\mu_n\}_n$ is bounded in $\mcW_\La$, their second moment is uniformly bounded. Passing $n\to\infty$, we find $\int\varphi d\mu_n\to\int \varphi d\mu$ for all such $\vphi$.
	
	Next we show the second moment converges:  $\int_{\mbr^d} \f 12|x|^2 d\mu_n\to\int_{\mbr^d} \f12|x|^2 d\mu$. We observe that for any L\'evy measure $\nu\in \La_2(\mbr^d)$, $\int_{\mbr^d}\f 12 |x|^2 d\nu(x)= \mcW_\La(\nu,0)^2. $
	Since $\mu_n\to\mu$ in $\mcl{W}_\La$, and $\nu\mapsto \mcW_\La(\nu,0)$ is continuous w.r.t. the metric, the convergence of the second moment holds from here.
	
	(ii) $\to$ (iii). $\mu_n\xrightarrow{\La}\mu$ follows from Proposition \ref{prop:appenC4}. We note $\vphi(x)=|x|^2$ satisfies $|\nabla \vphi(x)|\le 2|x|$. Hence (ii) implies the convergence of second moment.
	
	(iii) $\to$ (i). For $n\ge 1$, let $\gamma_n$ be a 2-Levy optimal coupling of $\mu_n, \mu$.
	We first verify that the sequence $\{\gamma_n\}_n$ is $\La$-tight. Since $\mu_n \xrightarrow{\La} \mu$, the sequence $\{\mu_n\}_n\subset \La_2(\mbr^d)$ is $\La$-tight (see Lemma~\ref{prop:la-conv-tight}). By Lemma \ref{lem:couplingtight}, $\{\ga_n\}_n\subset \La_2(\mbr^{2d})$ is $\La$-tight. 
	
	By Prokhorov’s theorem (Theorem~\ref{thm:prokhorov}), there exists a subsequence \(\{\gamma_{n_k}\}_k\) and a Lévy measure \(\gamma \in \La(\mathbb{R}^{2d})\) such that \(\gamma_{n_k} \xrightarrow{\La} \gamma\).
	We now show that \(\gamma\) is a coupling of \(\mu\) with itself. Since the marginals of \(\gamma_{n_k}\) are \(\mu_{n_k}, \mu\), and \(\mu_{n_k} \xrightarrow{\La} \mu\), the marginal projections of \(\gamma_{n_k}\) converge in the \(\La\)-sense to \(\mu\). By the continuity of pushforward under weak convergence, the marginals of \(\gamma\) are both equal to \(\mu\). That is, \(\gamma \in \Gamma^\La(\mu,\mu)\).
	
	We now claim that the second moments of \(\gamma_{n_k}\) converge to that of \(\gamma\). Indeed, we have:
	\[
	\int_{\mathbb{R}^{2d}} (|x|^2 + |y|^2) \, d\gamma_n(x,y) = \int_{\mathbb{R}^d} |x|^2 \, d\mu_n(x) + \int_{\mathbb{R}^d} |y|^2 \, d\mu(y),
	\]
	and the right-hand side converges by assumption. 
	By Proposition \ref{lem:C5},  we have
	\[
	\mcW_\La(\mu_{n_k},\mu)=\int_{\mathbb{R}^{2d}} \frac{1}{2} |x - y|^2 \, d\gamma_{n_k}(x,y) \to \int_{\mathbb{R}^{2d}} \frac{1}{2} |x - y|^2 \, d\gamma(x,y)=0.
	\]
	
	Since every subsequence admits a further subsequence along which convergence to zero holds, the full sequence must also converge:
	\(
	\mcW_\La(\mu_n, \mu) \to 0. 
	\)
	\qedhere 
\end{proof}

\subsection{Separability of the Wasserstein generator metric}

We now address the separability of the metric \( \mcW_\mcg \). By the decomposition in \eqref{eq:wass-LK}, this metric is the direct sum of three parts: the Euclidean metric on \( \mathbb{R}^d \), the Bures–Wasserstein metric on \( \mathcal{S}_{\geq 0}(\mathbb{R}^d) \), and the L\'evy–Wasserstein metric on \( \Lambda_2(\mathbb{R}^d) \). Hence, to establish the separability of \( (\mcg_2^\La(\mathbb{R}^d), \mcW_\mcg^\La) \), it suffices to show the separability of each component.

The Euclidean space \( \mathbb{R}^d \) is separable, as is \( \mathcal{S}_{\geq 0}(\mathbb{R}^d) \) under the Bures–Wasserstein metric (as a closed subset of the space of symmetric matrices with respect to the Frobenius norm). The remaining part is handled by the following result:

\begin{proposition}\label{prop:WLa-sep}
	The metric space \( (\Lambda_2(\mathbb{R}^d), \mathcal{W}_\Lambda) \) is separable.
\end{proposition}

\begin{proof}
	Fix \( r > 0 \), and define \( \mcl{M}_r \subset \La_2(\mathbb{R}^d) \) to be the set of Lévy measures with total mass \( \mu(\mathbb{R}^d) = r \). We first claim that \( (\mcl{M}_r, \mcW_\La) \) is separable. 
	
	Indeed, since \( \mcl{M}_r \) consists of finite measures with fixed total mass, it is a subset of the space of probability measures (up to scaling) with finite second moment. The classical 2-Wasserstein space over \( \mathbb{R}^d \) with fixed total mass \( r \) is known to be separable. Moreover, we have $\mcW_\La\le \mcW_2$ from Proposition \ref{prop:prelim2}, that is,
 	the Lévy-Wasserstein metric is dominated by the standard Wasserstein-2 metric. Therefore, the separability of \( (\mcl{M}_r, \mcW_\La) \) follows from that of \( (\mcl{M}_r, \mcW_2) \).

	Now define
	\[
	\mcl{M} := \bigcup_{q \in \mathbb{Q}^+} \mcl{M}_q.
	\]
	This is a countable union of separable subsets, hence separable.

	We now show that \( \mcl{M} \) is dense in \( \La_2(\mathbb{R}^d) \). Let \( \mu \in \La_2(\mathbb{R}^d) \) and fix \( \varepsilon > 0 \). Choose small \( \delta, \delta' > 0 \), and define
	\[
	d\mu_{\delta,\delta'} := (1 - \delta') \chi_{B_\delta(0)^c}(x) \, d\mu(x),
	\]
	which is a finite Lévy measure. Let \( \zeta_{\delta,\delta'} := \mu - \mu_{\delta,\delta'} \), so that \( \mu = \mu_{\delta,\delta'} + \zeta_{\delta,\delta'} \). By Proposition \ref{prop-prelim}, 
	\begin{align*}
		\mcW_\La(\mu,\mu_{\de,\de'})^2&\le \mcW_\La(\mu_{\de,\de'},\mu_{\de,\de'})^2+\mcW_\La(0,\zeta_{\de,\de'})^2 \\
		&=\frac{1}{2} \left( \int_{B_\delta(0)} |x|^2 \, d\mu(x) + \delta' \int_{B_\delta(0)^c} |x|^2 \, d\mu(x) \right).
	\end{align*}
	Because \( \mu \in \La_2(\mathbb{R}^d) \), this expression can be made arbitrarily small by taking \( \delta \), \( \delta' \) sufficiently small. Thus, we may choose \( \delta, \delta' \) such that
	\[
	\mcW_\La(\mu, \mu_{\delta,\delta'}) < \ep.
	\]
	
	Note that \( \mu_{\delta,\delta'} \in \mcl{M}_r \) for \( r = (1 - \delta') \mu(B_\delta(0)^c) \). By the density of the rational numbers, we may also choose \( \delta' > 0 \) small enough so that \( r \in \mathbb{Q} \), namely \( \mu_{\delta,\delta'} \in \mcl{M} \), while ensuring the error estimate above still holds. This completes the proof.
\end{proof}

We may now complete the proof of Theorem \ref{main3} . 
\begin{proof}[Proof of Theorem \ref{main3}, separability]
	The squared generator metric \( \mcW_\mcg^2 \) is the sum of three components: the squared Euclidean distance on \( \mathbb{R}^d \), the squared Bures–Wasserstein distance \( \mcW_\mcS^2 \) on the cone of nonnegative definite matrices, and the squared Lévy–Wasserstein distance \( \mcW_\La^2 \) on \( \La_2(\mathbb{R}^d) \). Each of these metric spaces is known to be separable. Since a finite product of separable metric spaces remains separable, it follows that \( (\mcg_2^\La(\mbr^d), \mcW_\mcg) \) is separable.
\end{proof}

\subsection{Pathwise estimate for coupled Lévy processes}
We shall next establish a maximal-type inequality for the optimal coupling of two L\'evy processes in terms of the metric $\wg$. 

\begin{proof}[Proof of Theorem \ref{main5}]
	Let us first prove the case when the two processes has zero mean. In this case, the process $\{Z_t = (X_t, Y_t)\}_{t \geq 0}$ generated by a 2-optimal L\'evy coupling $\mJ_*$ of $\mA,\mB$ has zero mean. Particularly, it is a vector martingale in $\mathbb{R}^{2d}$.
	
	Define the symmetric matrix
	\[
	U = \begin{bmatrix}
		I & -I \\
		-I & I
	\end{bmatrix}.
	\]
	Since $U \geq 0$, there exists a symmetric positive semidefinite matrix $Q$ such that $Q^2 = Q^\top Q = U$. Then
	\[
	|X_t - Y_t|^2 = Z_t^\top U Z_t = |Q Z_t|^2.
	\]
	Because $\{Z_t\}$ is a vector martingale, so is $\{QZ_t\}$. By Doob’s \(L^2\)-maximal inequality,
	\begin{align*}
		\mathbb{E} \left[ \sup_{t \in [0,T]} |Q Z_t|^2 \right]
		&\leq 4 \, \mathbb{E} |Q Z_T|^2 = 2 \, \mathbb{E} |X_T - Y_T|^2 \\
		&= 4 \left(e^{T \mJ_*} c_2\right)(0,0) = 4T \, \ta(0,0; \mA, \mB) = 4T \, \mcW_\mcg(\mA, \mB)^2,
	\end{align*}
	where the final equality uses Theorem~\ref{main2}(iii). This proves the estimate in the zero-mean case.
	
	Now consider general L\'evy processes $\{X_t\}, \{Y_t\}$ with potentially nonzero means. Write
	\(
	X_t = \bar{X}_t + t m_\mA, Y_t = \bar{Y}_t + t m_\mB,
	\)
	where $\{\bar{X}_t\}, \{\bar{Y}_t\}$ are the centered L\'evy processes with generators $\bar{\mA}, \bar{\mB}$. Note that
	\[
	\mcW^\Lambda_\mcg(\mA, \mB) = \ta(0,0; \bar{\mA}, \bar{\mB}) = \ta(0,0; \mA, \mB).
	\]
	Using the triangle inequality and the elementary bound \(|a + b|^2 \leq 2(|a|^2 + |b|^2)\), we obtain:
	\begin{align*}
		\mathbb{E}^{(0,0)} \left[ \sup_{t \in [0,T]} |X_t - Y_t|^2 \right]
		&\leq 2 \, \mathbb{E}^{(0,0)} \left[ \sup_{t \in [0,T]} |\bar{X}_t - \bar{Y}_t|^2 \right] + 2 |T m_\mA - T m_\mB|^2 \\
		&\leq 8T \, \ta(0,0; \mA, \mB) + 2T^2 |m_\mA - m_\mB|^2 \\
		&\leq 8 \max\{T, T^2\} \, \mcW^\Lambda_\mcg(\mA, \mB)^2. \qedhere
	\end{align*}
\end{proof}


%% file: S6_Appen1.tex
\appendix
\renewcommand{\thelemma}{A.\arabic{lemma}}
\renewcommand{\thetheorem}{A.\arabic{theorem}}
\renewcommand{\theproposition}{A.\arabic{proposition}}
\renewcommand{\theremark}{A.\arabic{remark}}

\section{Approximation of Admissible Pairs}

In this appendix, we construct smooth approximations of admissible test function pairs, as required in the proofs of Lemmas~\ref{lem:denseL1} and~\ref{lem:approx-C3}. Throughout, we fix two L\'evy measures \(\mu, \nu \in \La_2(\mathbb{R}^d)\) with finite second moments.
By an \emph{admissible pair}, we mean any lower semicontinuous functions \((\varphi, \psi) \in L^1(\mu) \oplus L^1(\nu)\) such that
\[
\varphi(x) + \psi(y) \le c_2(x, y) := \frac{1}{2} |x - y|^2, \qquad \text{for all } x, y \in \mathbb{R}^d.
\]
Let \(\mathcal{F}=\mcl{F}(\mu,\nu)\) denote the collection of all admissible pairs. To prove Lemma~\ref{lem:denseL1}, we introduce the following subclasses of \(\mathcal{F}\):
\begin{itemize}
	\item \(\mathcal{F}_0\): pairs \((\varphi, \psi)\) vanishing in a neighborhood of the origin and satisfying the bound 
	\begin{align*}
		|\varphi(x)|,|\psi(x)|\le C\min\cb{1,|x|^2},\qquad \mbox{for some }C\ge 0.
	\end{align*}
	Particularly $\vphi,\psi$ are bounded lower semicontinuous.
	\item \(\mathcal{F}_0^1\): elements of \(\mathcal{F}_0\) with \(\varphi, \psi\) bounded  Lipschitz;
	\item \(\mathcal{F}_0^2\): elements of \(\mathcal{F}_0^1\) with \(\varphi, \psi\) in $C^2_b(\mbr^d)$.
\end{itemize}
These classes satisfy the natural inclusions:
\[
\mathcal{F}_0^2 \subset \mathcal{F}_0^1 \subset \mathcal{F}_0 \subset \mathcal{F}.
\]
Lemma \ref{lem:denseL1} follows as a consequence of the density of $\mcl{F}_0^2$ in $\mcl{F}$. 

\subsection{Density of \texorpdfstring{$\mathcal{F}_0 \subset \mathcal{F}$}{F0 in F}}

Given a pair of lower semicontinuous functions \( f, g \colon \mathbb{R}^d \to \mathbb{R} \cup \{-\infty\} \) with \( f \le g \), and a function \( \varphi \colon \mathbb{R}^d \to \mathbb{R} \), we define the \emph{cutoff} function:
\[
\hat{\varphi} = \chi(\varphi; f, g)(x) := \max\{f(x), \min\{\varphi(x), g(x)\}\} = 
\begin{cases}
	g(x) & \text{if } \varphi(x) \ge g(x), \\
	\varphi(x) & \text{if } f(x) < \varphi(x) < g(x), \\
	f(x) & \text{if } \varphi(x) \le f(x).
\end{cases}
\]
Clearly, the cutoff satisfies \( f \le \hat{\varphi} \le g \).

\begin{lemma}\label{lem:cutoff}
	(a) Let \( f \le g \) with \( (f,g) \in \mathcal{F} \). If \( (\varphi,\psi) \in \mathcal{F} \), then the pair \( (\hat{\varphi},\hat{\psi}) \) also belongs to \( \mathcal{F} \), where  
	\[
	\hat{\varphi} := \chi(\varphi; f,g),\qquad \hat{\psi} := \chi(\psi; f, g) .
	\]
	
	(b) If $(\varphi,\psi)\in\mcl{F}$ and $T:\mbr^d\to\mbr^d$ is a contraction (i.e., 1-Lipchitz). 
	Suppose $\vphi\circ T\in L^1(\mu),\psi\circ T \in L^1(\nu)$.
	Then $(\varphi\circ T,\psi\circ T)\in \mcl{F}$. 
\end{lemma}

\begin{proof}
	(a) Since lower semicontinuity is preserved under pointwise maxima and minima, and since \( f, g, \varphi, \psi \) are all lower semicontinuous, it follows that \( \hat{\varphi} \) and \( \hat{\psi} \) are lower semicontinuous as well.
	Moreover, since \( f, g \in L^1(\mu) \) and \( L^1(\nu) \), and \( \hat{\varphi}, \hat{\psi} \) are bounded by \( f \) and \( g \), we also have \( \hat{\varphi} \in L^1(\mu) \), \( \hat{\psi} \in L^1(\nu) \).
	
	It remains to verify that \( \hat{\varphi}(x) + \hat{\psi}(y) \le c_2(x, y) \) for all \( x, y \in \mathbb{R}^d \), where \( c_2(x, y) := \frac{1}{2}|x - y|^2 \). To this end, we distinguish two cases. If either \( \hat{\varphi}(x) = f(x) \) or \( \hat{\psi}(y) = f(y) \), say the former without loss of generality, then by the bounds \( \hat{\psi}(y) \le g(y) \) and  \( (f, g)\in\mcl{F} \), we obtain
	\[
	\hat{\varphi}(x) + \hat{\psi}(y) \le f(x) + g(y) \le c_2(x, y).
	\]
	In the complementary case where both \( \hat{\varphi}(x) \neq f(x) \) and \( \hat{\psi}(y) \neq f(y) \), it follows from the definition of the cutoff that \( \hat{\varphi}(x) = \min\{\varphi(x), g(x)\} \) and \( \hat{\psi}(y) = \min\{\psi(y), g(y)\} \). In particular, we then have
	\[
	\hat{\varphi}(x) + \hat{\psi}(y) \le \varphi(x) + \psi(y) \le c_2(x, y).
	\]
	In either case, the inequality \( \hat{\varphi}(x) + \hat{\psi}(y) \le c_2(x, y) \) holds, as desired.
	
	(b) Since $T$ is continuous, and $(\varphi,\psi)$ is lower semicontinuous, so is $(\varphi\circ T,\psi\circ T)$. From the assumption, the pair is in $L^1(\mu)\oplus L^1(\nu)$. 
	Moreover it holds for all $x,y\in \mbr^d$:
	\begin{align*}
		\varphi(T(x))+\psi(T(y))&\le \f 12 |T(x)-T(y)|^2 \le \f 12 |x-y|^2.
	\end{align*}
	Hence $(\varphi\circ T,\psi\circ T)\in \mcl{F}$. 
\end{proof}

To complete the proof of the density of $\mcl{F}_0 \subset \mcl{F}$, we invoke Lemma~\ref{lem:cutoff} and construct a sequence $\{(f_n, g_n)\}_{n \in \mathbb{N}}$ of bounded continuous pairs, each vanishing on a neighborhood of the origin, such that  
\(
g_n \nearrow \frac12 |x|^2, f_n \searrow \chi_\infty,
\)
where  
\[
\chi_\infty(x) =
\begin{cases}
	0, & x = 0, \\
	-\infty, & x \neq 0.
\end{cases}
\]

\begin{lemma}\label{lem:comp-seq}
	There exists a sequence $\{(f_n, g_n)\}_{n \in \mathbb{N}}$ of pairs of bounded Lipschitz functions satisfying:
	\begin{itemize}
		\item $f_n \le 0 \le g_n$ and $(f_n,g_n)\in\mcl{F}$ for all $n$;
		\item $f_n$ and $g_n$ vanish on a neighborhood of $0$;
		\item $g_n(x) \nearrow \frac12 |x|^2$ and $f_n(x) \searrow \chi_\infty(x)$ pointwise as $n \to \infty$.
	\end{itemize}
\end{lemma}

\begin{proof}
	We use the following standard construction.  
	For $\delta \in (0,1)$, set
	\[
	f_\delta(x) = \frac12 \big(1 - \delta^{-1}\big) |x|^2, 
	\qquad
	g_\delta(y) = \frac12 (1 - \delta) |y|^2.
	\]
	Then $(f_\delta, g_\delta)$ is a $c_2$-concave dual pair with $f_\delta \le0\le g_\delta$, in particular satisfying $f_\delta \oplus g_\delta \le c_2$.  
	For $\varepsilon > 0$, define $\rho_\varepsilon : [0, \infty) \to [0, \infty)$ by
	\[
	\rho_\varepsilon(r) =
	\begin{cases}
		0, & r \in [0, \varepsilon],\\
		r - \varepsilon, & r \in (\varepsilon, \tfrac{1}{\varepsilon} - \varepsilon),\\
		\tfrac{1}{\varepsilon}, & r \ge \tfrac{1}{\varepsilon} - \varepsilon.
	\end{cases}
	\]
	Let $T_\varepsilon(x) = \frac{x}{|x|} \rho_\varepsilon(|x|)$.  
	For each $\varepsilon > 0$, the map $T_\varepsilon$ is a contraction and $T_\varepsilon(x) \to x$ locally uniformly as $\varepsilon \searrow 0$.  
	
	Define
	\[
	f_{\delta,\varepsilon}(x) := f_\delta\big(T_\varepsilon(x)\big), 
	\qquad 
	g_{\delta,\varepsilon}(y) := g_\delta\big(T_\varepsilon(y)\big).
	\]
	From the construction we have $f_{\delta,\varepsilon} \le0\le g_{\delta,\varepsilon}$; each pair $(f_{\delta,\varepsilon}, g_{\delta,\varepsilon})$ is bounded, Lipschitz, and vanishes on $B_\varepsilon(0)$.  
	Also, we have $(f_{\de,\ep},g_{\de,\ep})\in \mcl{F}$, by Lemma \ref{lem:cutoff}.
	Moreover, $(\delta, \varepsilon) \mapsto f_{\delta,\varepsilon}$ decreases pointwise to $\chi_\infty$, while $(\delta, \varepsilon) \mapsto g_{\delta,\varepsilon}$ increases pointwise to $\frac12 |x|^2$ as $(\delta, \varepsilon) \searrow (0,0)$.  
	Choosing a sequence $(\delta_n, \varepsilon_n) \searrow (0,0)$ and setting 
	\(
	f_n := f_{\delta_n, \varepsilon_n}, 
	g_n := g_{\delta_n, \varepsilon_n},
	\)
	yields the desired sequence.
\end{proof}

\newcommand{\fblsc}{\mcl{F}_{\mathrm{blsc}}}

\begin{lemma}\label{lem:F1-dense}
	$\mcl{F}_0\subset \mcl{F}$ is dense w.r.t. the topology of $L^1(\mu)\oplus L^1(\nu)$.
\end{lemma}

\begin{proof}
	Choose $(\varphi,\psi)\in \mcl{F}$, and our goal is to construct a sequence $\{(\varphi_n,\psi_n)_n\}\subset \mcl{F}_0$ such that $(\varphi_n,\psi_n)\to(\varphi,\psi)$ in $L^1(\mu)\oplus L^1(\nu)$. Let $\{(f_n,g_n)\}_n$ be the sequence given in Lemma \ref{lem:comp-seq}, and define the cutoff sequences 
	\begin{align*}
		\varphi_n := \chi(\varphi;f_n,g_n),\qquad \psi_n:= \chi(\psi;f_n,g_n).
	\end{align*}
	By Lemma \ref{lem:cutoff}, $(\varphi_n,\psi_n)\in\mcl{F}$. Since $\chi_\infty(x)\le \varphi(x),\psi(x) \le \frac 12 |x|^2$ and $g_n\nearrow \frac 12 |x|^2$, $f_n\searrow \chi_{\infty}(x)$, it follows $|\varphi_n|\nearrow |\varphi|,|\psi_n|\nearrow |\psi|$ pointwise as $n\to\infty$. The dominated convergence theorem implies $\varphi_n\to \varphi$ in $L^1(\mu)$, $\psi_n\to \psi$ in $L^1(\nu)$. 
\end{proof}

\subsection{Density of smooth admissible test functions (proof of Lemma \ref{lem:denseL1})}

We now show that $\mcl{F}_0^1$, the subclass of Lipschitz pairs, is dense in $\mcl{F}$.  
This follows from the fact that any bounded lower semicontinuous function can be approximated by an increasing sequence of Lipschitz functions.  
Specifically, given a lower semicontinuous function $\varphi$ bounded below by $\varphi(x) \ge C$, define
\begin{align}\label{eq:lip-approx}
	\varphi_n(x) := \inf_{y \in \mbr^d} \big[ \varphi(y) + n|x-y| \big],
\end{align}
that is, the infimal convolution of $\varphi$ with $n|x|$ (see \eqref{def:infconvo}).  
Then $\varphi_n$ is $n$-Lipschitz and satisfies $\varphi_n \nearrow \varphi$ pointwise as $n \to \infty$.

Moreover, the following stability property holds: if $f$ is $C$-Lipschitz and $\varphi \geq f$, then $\varphi_n \geq f$ for all $n \geq C$.  In particular, if $\varphi$ is bounded and vanishes in a neighbourhood of the origin, then for all sufficiently large $n$ the same holds for $\varphi_n$.
This property will be useful in establishing a uniform bound in $n$ for the sequence $\{\varphi_n\}_{n}$.

\begin{lemma}\label{lem:F01dense}
	The set $\mcl{F}_0^1 \subset \mcl{F}$ is dense with respect to the topology of $L^1(\mu) \oplus L^1(\nu)$. 
\end{lemma}

\begin{proof}
	By Lemma \ref{lem:F1-dense}, it suffices to prove that $\mcl{F}_0^1$ is dense in $\mcl{F}_0$.  
	Fix $(\varphi,\psi) \in \mcl{F}_0$, where $\varphi$ and $\psi$ are bounded, lower semicontinuous, and vanish in a neighbourhood of the origin.  
	
	Choose a nonnegative Lipschitz function $f$ that also vanishes in a neighbourhood of the origin and satisfies
	\[
	|\varphi(x)|,\, |\psi(x)| \le f(x) \quad \text{for all $x$},
	\]
	which is clearly possible.  	
	Let $(\varphi_n,\psi_n)$ be defined as in \eqref{eq:lip-approx}.  
	We have $\varphi_n \nearrow \varphi$ and $\psi_n \nearrow \psi$, hence $\varphi_n \oplus \psi_n \le c_2$.  
	Since $f$ is Lipschitz, the stability property above ensures that $\varphi_n, \psi_n \ge -f$ for all sufficiently large $n$.  
	Thus for large $n$ we have $|\varphi_n|, |\psi_n| \le f$, so $(\varphi_n,\psi_n) \in \mcl{F}_0^1$.  
	
	Finally, since $f$ is integrable with respect to both $\mu$ and $\nu$, the dominated convergence theorem yields
	\(
	(\varphi_n,\psi_n) \to (\varphi,\psi)\) in  $L^1(\mu) \oplus L^1(\nu)$,
	completing the proof.
\end{proof}

Let us now establish the density of $\mcl{F}_0^2$ in $\mcl{F}$.  
The following lemma will be the key ingredient.

\begin{lemma}\label{lem:convo}
	Let \( q : \mathbb{R}^d \to [0, \infty) \) be given, and let \( \varphi, \psi : \mathbb{R}^d \to \mathbb{R} \) be lower semicontinuous functions satisfying
	\[
	\varphi(x) + \psi(y) \leq q(x - y), \qquad \text{for all } x, y \in \mathbb{R}^d.
	\]
	Let \( \eta \in L^1(\mathbb{R}^d) \) be a nonnegative mollifier, i.e.,
	\(
	\eta \geq 0, \int_{\mathbb{R}^d} \eta(x)\,dx = 1.
	\)
	Then the mollified pair \( (\eta * \varphi,\, \eta * \psi) \) also satisfies
	\[
	(\eta * \varphi)(x) + (\eta * \psi)(y) \leq q(x - y), \qquad \text{for all } x, y \in \mathbb{R}^d.
	\]
\end{lemma}

\begin{proof}
	For any \( x, y \in \mathbb{R}^d \),
	\[
	(\eta * \varphi)(x) + (\eta * \psi)(y)
	= \int_{\mathbb{R}^d} \big[ \varphi(x - z) + \psi(y - z) \big] \eta(z) \, dz
	\leq q(x - y),
	\]
	since the integrand is bounded above by \( q(x - y) \) pointwise.
\end{proof}

We can now deduce the desired density result.

\begin{lemma}\label{lem:F02dense}
	The set $\mcl{F}_0^2 \subset \mcl{F}$ is dense with respect to the topology of $L^1(\mu) \oplus L^1(\nu)$.
\end{lemma}

\begin{proof}
	It suffices to prove that $\mcl{F}_0^2$ is dense in $\mcl{F}_0^1$.  
	Fix a $C^\infty$ mollifier $\eta$ supported on $\overline{B_1(0)}$, and define $\eta_\varepsilon(x) = \varepsilon^{-d} \eta(\varepsilon^{-1}x)$.  
	Given $(\varphi,\psi) \in \mcl{F}_0^1$, set
	\[
	\varphi_\varepsilon := \eta_\varepsilon * \varphi, \qquad
	\psi_\varepsilon := \eta_\varepsilon * \psi.
	\]
	Since $(\varphi,\psi)$ vanishes on $B_r(0)$ for some $r > 0$, choosing $\varepsilon \in (0,r)$ ensures that $(\varphi_\varepsilon, \psi_\varepsilon)$ also vanishes on some (possibly smaller) neighborhood of the origin.
	
	By Lemma~\ref{lem:convo}, the smoothed pair satisfies $\varphi_\varepsilon \oplus \psi_\varepsilon \leq c_2$, hence $(\varphi_\varepsilon, \psi_\varepsilon) \in \mcl{F}_0^2$.  
	Moreover, since $(\varphi,\psi)$ is bounded Lipschitz, we have uniform convergence $(\varphi_\varepsilon,\psi_\varepsilon) \to (\varphi,\psi)$.  
	By the bounded convergence theorem, this implies $\varphi_\varepsilon \to \varphi$ in $L^1(\mu)$ and $\psi_\varepsilon \to \psi$ in $L^1(\nu)$.  
	Therefore $\mcl{F}_0^2$ is dense in $\mcl{F}_0^1$.
\end{proof}

\begin{proof}[Proof of Lemma \ref{lem:denseL1}]
	Denote $\mcl{F}_{(iv)}$ the family of all pairs given in Lemma \ref{lem:denseL1}(iv). We note $\mcl{F}_0^2\subset \mcl{F}_{(iv)}\subset \mcl{F}$. Since $\mcl{F}_0^2$ is dense in $\mcl{F}$ by Lemma \ref{lem:F02dense}, so is $\mcl{F}_{(iv)}$. 
\end{proof}

\subsection{Proof of Lemma \ref{lem:approx-C3}}
The proof of Lemma \ref{lem:approx-C3} relies on the approximation via the \emph{infimal convolution}.
Given two functions \( \chi, g : \mathbb{R}^d \to \mathbb{R} \cup \{\infty\} \), their \emph{infimal convolution} is defined by
\begin{align}\label{def:infconvo}
	(\chi \square g)(x) := \inf_{z \in \mathbb{R}^d} \left\{ \chi(z) + g(x - z) \right\} 
	= \inf_{\substack{z+z'=x}} \left\{ \chi(z) + g(z') \right\}.
\end{align}
The operator \( \square \) is \emph{commutative} in general, but not \emph{associative}.

In what follows, given a lower semicontinuous function \( q : \mathbb{R}^d \to [0,\infty] \), we write
\[
c_q(x,y) := q(x-y).
\]
Also, for a function \( f : \mathbb{R}^d \to \mathbb{R} \cup \{\infty\} \), we define the reflection
\[
f^\sim(x) := f(-x).
\]

\begin{lemma}\label{lem:inf-convo}
	Let $\chi, q : \mathbb{R}^d \to [0,\infty]$ and let $\varphi,\psi : \mathbb{R}^d \to \mathbb{R}$ be lower semicontinuous functions satisfying
	\(
	\varphi \oplus \psi \le c_q.
	\)
	Define
	\[
	\hat\varphi := \chi \square \varphi, 
	\quad \hat\psi := \chi \square \psi,
	\quad \hat q := q \square (\chi \square \chi^\sim),
	\quad \hat c_q(x,y) := \hat q(x-y).
	\]
	Then the pair $(\hat\varphi, \hat\psi)$ satisfies
	\(
	\hat\varphi \oplus \hat\psi \le \hat c_q.
	\)
\end{lemma}

\begin{proof}
	By the definition of the infimal convolution,
	\begin{align*}
		\hat\varphi(x) + \hat\psi(y) 
		&= \inf_{z} \left\{ \chi(z) + \varphi(x - z) \right\}
		+ \inf_{z'} \left\{ \chi(z') + \psi(y - z') \right\} \\
		&\le \inf_{z, z'} \left\{ \chi(z) + \chi(z') + \varphi(x - z) + \psi(y - z') \right\} \\
		&\le \inf_{z, z'} \left\{ \chi(z) + \chi(z') + q(x - y - (z - z')) \right\}.
	\end{align*}
	Setting \( w := x - y \) and \( u := z - z' \), we rewrite the last expression as
	\begin{align*}
		\inf_{u \in \mathbb{R}^d} \left\{ q(w - u) + \inf_{z - z' = u} \left[ \chi(z) + \chi(z') \right] \right\}&= \inf_{u \in \mathbb{R}^d} \left\{ q(w - u) + (\chi \square \chi^\sim)(u) \right\} \\
		&= q \square (\chi \square \chi^\sim)(w)= \hat q(w).
	\end{align*}
	Thus $\hat\varphi\oplus \hat\psi \le \hat c_q$, which proves the claim.
\end{proof}

In applications, the infimal convolution is often used to produce approximations from below.  
Let \( \{ \chi_\varepsilon \}_{\varepsilon > 0} \) be a family of nonnegative convex functions such that
\(
\chi_\varepsilon(0) = 0, \chi_\varepsilon \nearrow \chi_0 \text{ as } \varepsilon \searrow 0,
\)
where
\[
\chi_0(x) :=
\begin{cases}
	0 & \text{if } x = 0, \\
	\infty & \text{otherwise}.
\end{cases}
\]
Then, for any lower semicontinuous \( f \colon \mathbb{R}^d \to \mathbb{R} \), we have
\[
\chi_\varepsilon \square f \nearrow f \quad \text{pointwise as } \varepsilon \searrow 0.
\]

In the proof of Lemma~\ref{lem:approx-C3}, we take
\[
\chi_\varepsilon(x) :=
\begin{cases}
	0 & \text{if } x \in \overline{B_\varepsilon(0)}, \\
	\infty & \text{otherwise}.
\end{cases}
\]
If we define \( q(x) := \frac12 |x|^2 \), then
\[
\hat q^\varepsilon := q \square \big( \chi_{\varepsilon/2} \square \chi_{\varepsilon/2}^\sim \big) 
= q \square \chi_\varepsilon
\]
can be computed explicitly:
\[
\hat q^\varepsilon(x) = \frac12 \max(0, |x| - 2\varepsilon)^2.
\]

\begin{proof}[Proof of Lemma \ref{lem:approx-C3}]
	By Lemma \ref{lem:F01dense}, $\mathcal{F}_0^1$ is dense in $\mathcal{F}$, so it suffices to prove the result for $(\varphi,\psi) \in \mathcal{F}_0^1$. 
	Fix \( \varepsilon > 0 \), and define
	\[
	\hat\varphi_\varepsilon := \varphi \square \chi_{\varepsilon/2}, 
	\quad \hat\psi_\varepsilon := \psi \square \chi_{\varepsilon/2}.
	\]
	Since \( \varphi \) and \( \psi \) vanish in some $B_r(0)$, for any \( \varepsilon \in (0, r) \) the pair \( (\hat\varphi_\varepsilon, \hat\psi_\varepsilon) \) vanishes on $B_{r/2}(0)$. 
	Moreover, infimal convolution with $\chi_{\varepsilon/2}$ preserves Lipschitz continuity, so $(\hat\varphi_\varepsilon, \hat\psi_\varepsilon) \in \mathcal{F}_0^1$.  
	By Lemma~\ref{lem:inf-convo},
	\begin{align}\label{bdd:tempA3}
		\hat\varphi_\varepsilon \oplus \hat\psi_\varepsilon \le c_2^\varepsilon.	
	\end{align}
	
	Let $\{\eta_{\varepsilon'}\}_{\varepsilon'}$ be the smooth mollifiers from Lemma \ref{lem:F02dense}, with $\mathrm{supp}(\eta_{\varepsilon'}) \subset B_{\varepsilon'}(0)$.  
	Choosing $\varepsilon' = \varepsilon/4$ ensures that
	\(
	(\varphi_\varepsilon, \psi_\varepsilon) := \eta_{\varepsilon'} * (\hat\varphi_\varepsilon, \hat\psi_\varepsilon)
	\)
	still vanishes on $B_{r/4}(0)$. This yields a $C^2$ pair vanishing near the origin, satisfying condition (i) of Lemma \ref{lem:approx-C3}.
	
	Condition (ii) follows from \eqref{bdd:tempA3} and Lemma \ref{lem:convo}.  
	Finally, since $(\hat\varphi_\varepsilon, \hat\psi_\varepsilon) \to (\varphi,\psi)$ uniformly as $\varepsilon \searrow 0$, the mollified sequence $(\varphi_\varepsilon, \psi_\varepsilon)$ also converges uniformly.  
	By the bounded convergence theorem, $(\varphi_\varepsilon,\psi_\varepsilon) \to (\varphi,\psi)$ in $L^1(\mu) \oplus L^1(\nu)$, proving condition (iii).  
	This completes the proof.
\end{proof}

%% file: S6_Appendix.tex
\renewcommand{\thelemma}{B.\arabic{lemma}}
\renewcommand{\thetheorem}{B.\arabic{theorem}}
\renewcommand{\theproposition}{B.\arabic{proposition}}
\renewcommand{\theremark}{B.\arabic{remark}}

\section{Proofs of Propositions \ref{prop:coup-char}, \ref{om-ta-prop} and \ref{prop:c2-reg}}\label{Appen-genresult}

\subsection{Equivalent characterization of Feller couplings}\label{Appen-Coupling}

Let $(\mA,D(\mA))\in\mcg(\Pi)$, and let $\{e^{t\mA}\}_{t\ge 0}$ denote the probability semigroup generated by $\mA$. Since this semigroup admits a transition kernel $\{\ka_t(x)\}_{t\ge 0,\,x\in \Pi}$, it acts as a family of bounded operators $C_0(\Pi)\to C_0(\Pi)$ in the sense that, for all $\vphi\in C_0(\Pi)$,  
\begin{align*}
	(e^{t\mA}\vphi)(x) &= \int_{\Pi} \vphi(y)\,\ka _t(x,dy).
\end{align*}
The operator $e^{t\mA}$, $t\ge 0$, naturally extends to the space of bounded continuous functions $C_b(\Pi)$, or more generally to arbitrary continuous or measurable functions, provided the integral is well-defined. The following properties are immediate:
\begin{itemize}
	\item Order preservation: If $\vphi_1\le\vphi_2$, then $e^{t\mA}\vphi_1\le e^{t\mA}\vphi_2$.
	\item Conservation: $e^{t\mA}1=1$.
	\item Monotone convergence: If $0\le \vphi_n\nearrow \vphi$, then $e^{t\mA}\vphi_n\nearrow e^{t\mA}\vphi$. 
\end{itemize}

For $\la>0$, the \emph{$\la$-resolvent operator} of $\mA$ is the bounded operator $R_\la(\mA):C_0(\Pi)\to C_0(\Pi)$ defined by
\begin{align*}
	R_\la(\mA)\vphi(x) &= \int_0^\infty e^{-\la s}(e^{s\mA}\vphi)(x)\,ds. 
\end{align*}
Equivalently, $R_\la(\mA)$ can be expressed as an integral operator with the \emph{$\la$-potential kernel}
\begin{align*}
	R_\la(\mA)\vphi(x) &= \int_{\Pi} \vphi(y)\,U_\la(x,dy), 
	\qquad U_\la(x,dy)= \int_0^\infty e^{-\la s}\ka_s(x,dy)\,ds. 
\end{align*}
Thus, its definition extends in the same way to general measurable functions, including all $\vphi\in C_b(\Pi)$, whenever the integral is well-defined. As with the semigroup, the resolvent is order preserving and satisfies monotone convergence.  
In particular, the \emph{scaled resolvent operator}
\begin{align*}
	I_\la(\mA) := \la R_\la(\mA)
\end{align*}
is also order preserving and satisfies the monotone convergence and  conservation property $I_\la(\mA)1=1$. Moreover, it is well known that $I_\la(\mA)$ approximates the identity as $\la\to \infty$, in the sense that
\begin{align*}
	I_\la(\mA)\vphi \to \vphi, \qquad \text{for all }\vphi\in C_0(\Pi).
\end{align*}

It is well known that $R_\la(\mA)$ is the inverse of $\la I-\mA$. More precisely, for all $\vphi\in C_0(\Pi)$ and $\psi\in D(\mA)$,  
\begin{align*}
	(\la I-\mA)R_\la(\mA)\vphi &= \vphi, 
	\qquad R_\la(\mA)(\la I-\mA)\psi = \psi. 
\end{align*}
These identities extend further to $\vphi\in C_b(\Pi)$ and $\psi\in \bar D(\mA)$. In particular, the second equality---which will be useful later---follows directly from the fundamental theorem of calculus:
\begin{align*}
	R_\la(\mA)(\la I-\mA)\psi(x)
	&= \int_0^\infty e^{-\la s}\,e^{s\mA} (\la I-\mA)\psi(x)\,ds \\
	&= \int_0^\infty \frac{d}{ds}\Big[e^{-\la s}\,e^{s\mA}\psi(x)\Big]ds 
	= \big(e^{-\la s}e^{s\mA}\psi(x)\big)\Big|_{s=0}^\infty 
	= \psi(x).
\end{align*}
Here, the equality is understood pointwise in $x\in \Pi$. In particular, the fact that the map $s\mapsto e^{-\la s}e^{s\mA}\psi(x)$ is differentiable for each $x\in\Pi$ is crucial. 

In the coming lemma, we present the \emph{Kato--Trotter type} approximation formula for probability semigroups, which expresses $e^{t\mA}$ as a limit involving the scaled resolvent operator $I_\la(\mA)$. The formula in fact holds for general $C_0$-semigroups (see \cite{engel2000one}), and thus for all $\vphi\in C_0(\Pi)$. Since we require an extension to $\vphi\in C_b(\Pi)$, and we are not aware of an explicit reference in the literature, we include a proof here for completeness.

\begin{lemma}\label{lem:tech}
	Let $\mA \in \mcg(\Pi)$, and define $I_\la(\mA) = \la R_\la(\mA)$. 
	Then for every $\vphi \in C_b(\Pi)$ and $t>0$, 
	\begin{align*}
		\lim_{n\to\infty} \big(I_{n/t}(\mA)\big)^n \vphi(x) 
		= \big(e^{t\mA}\vphi\big)(x),
		\qquad \text{for all } x\in\Pi. 
	\end{align*}
\end{lemma}

\begin{proof}
	Let us write $S_n := I_{n/t}(\mA)^n$. 
	By the conservation property of both $I_\la(\mA)$ and $e^{t\mA}$, we have for any $\alpha,c\in\mbr$ and $\vphi\in C_b(\Pi)$,
	\[
	S_n(\alpha\vphi+c) = \alpha S_n(\vphi)+c, 
	\qquad 
	e^{t\mA}(\alpha\vphi+c) = \alpha e^{t\mA}\vphi + c.
	\]
	Hence, by rescaling and shifting $\vphi$, it is sufficient to prove the claim for functions $\vphi\in C_b(\Pi)$ with $0\le \vphi \le 1$. 
	
	Let $\{\chi_m\}_m\subset C_c(\Pi)$ be a sequence of compactly supported functions with $\chi_m\nearrow 1$ pointwise, and set $\chi_m' := 1-\chi_m$. We decompose
	\begin{align*}
		e^{t\mA}\vphi - S_n\vphi
		&= \big(e^{t\mA}[\vphi\chi_m] - S_n[\vphi\chi_m]\big) 
		+ \big(e^{t\mA}[\vphi\chi_m'] - S_n[\vphi\chi_m']\big).
	\end{align*}
	Since $\vphi\chi_m\in C_0(\Pi)$, and $S_n\to e^{t\mA}$ strongly on $C_0(\Pi)$ as mentioned before the proof, we have $S_n[\vphi\chi_m]\to e^{t\mA}[\vphi\chi_m]$ uniformly. 
	Passing to the limit $n\to\infty$ gives, for each $x\in\Pi$,
	\begin{align*}
		\limsup_{n\to\infty}\big|(e^{t\mA}\vphi - S_n\vphi)(x)\big|
		&\le |e^{t\mA}[\vphi\chi_m'](x)| 
		+ \limsup_{n\to\infty}|S_n[\vphi\chi_m'](x)| \\
		&\le e^{t\mA}[\chi_m'](x) 
		+ \limsup_{n\to\infty} S_n[\chi_m'](x).
	\end{align*}
	Here we used that both operators are order preserving and $0\le \vphi\chi_m'\le \chi_m'$. 
	Next, since $\chi_m\in C_0(\Pi)$, the strong convergence $S_n\to e^{t\mA}$ implies
	\begin{align*}
		\lim_{n\to\infty} S_n[\chi_m']
		&= \lim_{n\to\infty}\big(S_n[1-\chi_m]\big) 
		= 1 - \lim_{n\to\infty} S_n\chi_m 
		= 1 - e^{t\mA}\chi_m 
		= e^{t\mA}\chi_m'. 
	\end{align*}
	Hence we obtain
	\begin{align*}
		\limsup_{n\to\infty}\big|(e^{t\mA}\vphi - S_n\vphi)(x)\big|
		\le 2\big(1 - e^{t\mA}\chi_m(x)\big). 
	\end{align*}
	This estimate holds for all $m\ge 1$. Since $\chi_m\nearrow 1$, the monotone convergence property of $e^{t\mA}$ yields $e^{t\mA}\chi_m \nearrow 1$. 
	Letting $m\to\infty$ therefore gives
	\[
	\limsup_{n\to\infty}\big|(e^{t\mA}\vphi - S_n\vphi)(x)\big| = 0,
	\]
	which completes the proof.
\end{proof}

The next proposition, which contains Proposition \ref{prop:coup-char}, collects a range of equivalent characterisations for coupling generators, formulated in terms of semigroup, generator, and resolvent conditions. 
For a probability semigroup $\mA\in \mcg(\Pi)$ and $T>0$, recall the continuity domain $C(\mA)$ of $\mA$, as defined in Definition \ref{def:extdom}.

\begin{proposition}
	Let $\mA,\mB\in\mcg(\Pi)$ and $\mJ\in \mcg(\Pi^2)$. The following are equivalent.
	
	\begin{enumerate}[label=(\roman*)]
		\item For all $\vph,\psi\in C_0(\Pi)$, it holds
		\begin{align}\label{eq:semi-marg}
			e^{t\mJ}[\vph\otimes 1]=(e^{t\mA}\vph) \otimes 1,\qquad e^{t\mJ}[1\otimes\psi]= 1\otimes (e^{t\mB}\psi),\qquad t\ge 0. 
		\end{align}
		\item For all $\vph,\psi\in C_b(\Pi)$, \eqref{eq:semi-marg} holds. 
		\item For all $\vph\in C(\mA),\psi\in C(\mB)$, \eqref{eq:semi-marg} holds.
		\item For all $\vph\in \bar D(\mA),\psi\in \bar D(\mB)$, it holds
		\begin{align}\label{eq:gen-marg}
			\mJ[\vph\otimes 1]=(\mA\vph)\otimes 1,\qquad \mJ[1\otimes \psi]= 1\otimes (\mB\psi). 
		\end{align}
		\item For all $\vph\in D(\mA),\psi\in D(\mB)$, \eqref{eq:gen-marg} holds.
		\item For all $\la>0$, and $\varphi,\psi \in C_0(\Pi)$ that 
		\begin{align*}
			R_\la(\mJ)[\varphi\otimes 1]= (R_\la(\mA)\varphi)\otimes 1, \qquad R_\la(\mJ)[1\otimes \psi]= 1\otimes (R_\la(\mB)\psi).
		\end{align*}
	\end{enumerate}
\end{proposition}

\begin{proof}
	By symmetry, it suffices to prove the marginal condition involving $\mA$ for each implication.
	
	(i) $\to$ (ii).  
	It is enough to treat the case $\vphi\ge 0$ with $\vphi\in C_b(\Pi)$. Let $\{\vphi_n\}_n\subset C_0(\Pi)$ be an increasing sequence with $\vphi_n\nearrow \vphi$. Then $\vphi_n\otimes 1 \nearrow \vphi\otimes 1$. Using (i) and the monotone convergence property of $e^{t\mJ}$ and $e^{t\mA}$,
	\[
	e^{t\mJ}[\vphi\otimes 1]
	= \lim_{n\to\infty} e^{t\mJ}[\vphi_n\otimes 1]
	= \lim_{n\to\infty}(e^{t\mA}\vphi_n)\otimes 1
	= (e^{t\mA}\vphi)\otimes 1.
	\]
	
	(ii) $\to$ (iii).  
	Let $\vphi\in C(\mA)$ be nonnegative. Approximate $\vphi$ monotonically from below by a sequence in $C_b(\Pi)$, and apply (ii) together with the convergence theorem. This yields \eqref{eq:semi-marg} for $\vphi\in C(\mA)$. 
	
	(iii) $\to$ (iv).  
	Suppose $\vphi\in \bar D(\mA)$, so that $\f 1t(e^{t\mA}-1)\vphi \to\mA\vphi$ locally uniformly as $t\searrow 0$. 
	By (iii),
	\[
	\frac{e^{t\mJ}-I}{t}(\vphi\otimes 1) 
	= \Big(\frac{e^{t\mA}-I}{t}\vphi\Big)\otimes 1.
	\]
	Passing to the limit $t\searrow 0$, the right-hand side converges locally uniformly to $(\mA\vphi)\otimes 1$, and hence so does the left-hand side. This gives \eqref{eq:gen-marg}.
	
	(iv) $\to$ (v).  
	Immediate, since $D(\mA)\subset \bar D(\mA)$. 
	
	(v) $\to$ (vi).  
	Take $\vphi\in C_0(\Pi)$ and set $\vphi' := R_\la(\mA)\vphi\in D(\mA)$, so that $(\la I-\mA)\vphi'=\vphi$. Then
	\begin{align*}
		(\la I-\mJ)[\vphi'\otimes 1]
		&= \la\vphi'\otimes 1 - \mJ[\vphi'\otimes 1] \\
		&= \la\vphi'\otimes 1 - (\mA\vphi')\otimes 1 
		= (\la I-\mA)\vphi'\otimes 1 = \vphi\otimes 1.
	\end{align*}
	Applying $R_\la(\mJ)$ gives
	\[
	R_\la(\mJ)[\vphi\otimes 1] = \vphi'\otimes 1 = (R_\la(\mA)\vphi)\otimes 1.
	\]
	
	(vi) $\to$ (i).  
	For $\vphi\in C_0(\Pi)$ we have $\vphi\otimes 1\in C_b(\Pi)$. By (vi),
	\[
	I_\la(\mJ)[\vphi\otimes 1] = (I_\la(\mA)\vphi)\otimes 1.
	\]
	Iterating, the same holds for all powers $I_\la^n$. Applying Lemma~\ref{lem:tech},
	\begin{align*}
		e^{t\mJ}[\vphi\otimes 1]
		= \lim_{n\to\infty} I_{n/t}(\mJ)^n[\vphi\otimes 1] 
		= \lim_{n\to\infty} \big(I_{n/t}(\mA)^n\vphi\big)\otimes 1
		= (e^{t\mA}\vphi)\otimes 1.
	\end{align*}
	This proves (i).
\end{proof}

\subsection{Proof of Proposition \ref{om-ta-prop}}
We are to prove the inequality:
\begin{align*}
	\om_c^*(x,y;\mA,\mB)\le \om_c^-(x,y;\mA,\mB)\le \om^+(x,y;\mA,\mB)\le \ta_c(x,y;\mA,\mB). 
\end{align*}
As pointed out in Remark \ref{rem:om-bdd}, the second inequality is straightforward from the definition of $\om^\pm_c$.

Let us proceed with the last inequality. Choose a Feller coupling $\mJ\in \Ga(\mA,\mB)$. For any $(x,y)\in \Pi^2$, $\ga_t = \de_{(x,y)}e^{t\mJ}\in\mcp(\Pi^2)$ is a coupling of the probability measures $\de_{x}e^{t\mA},\de_ye^{t\mB}$. This follows
\begin{align*}
	\mcc_c(\de_x e^{t\mA},\de_y e^{t\mB})\le \int_{\Pi^2} c(x,y)d\ga_t(x,y)= e^{t\mJ}c(x,y). 
\end{align*}
Subtracting both sides with $c(x,y)$, dividing $t$, and passing $t\searrow0$ we find
\begin{align*}
	\om_c^+(x,y;\mA,\mB)&= \limsup_{t\to 0} \frac{\mcc_c(\de_x e^{t\mA},\de_y e^{t\mB})-c(x,y)}{t}\le  \liminf_{t\searrow 0} \frac{e^{t\mJ}c(x,y)-c(x,y)}{t}=\mJ^- c(x,y). 
\end{align*}
This holds for any $\mJ\in \Ga(\mA,\mB)$. Taking the infimum over $\mcl{J}\in\Ga(\mA,\mB)$ yields the second inequality. 

Finally consider the first inequality. Fix $(x_0,y_0)\in \Pi^2$ and any pair $\varphi\in \bar D(\mA),\psi\in \bar D(\mB)$ so that $\varphi\oplus\psi\le c$, and $\varphi(x_0)+\psi(y_0)=c(x_0,y_0)$. 
Fix $t\ge 0$ and denote $\mu_t = \de_{x_0}e^{t\mA}, \nu_t = \de_{y_0} e^{t\mB}$. Let $\ga_t$ be a $c$-optimal coupling of the measures $\mu_t,\nu_t$. This follows
\begin{align*}
	(e^{t\mA}\varphi)(x_0)+(e^{t\mB}\psi)(y_0)&=\int_\Pi \varphi(x) d\mu_t  + \int_\Pi \psi(x) d\nu_t\\
	&= \int_{\Pi^2}[\varphi(x)+\psi(y)]d\gamma_t(x,y) \le \int_{\Pi^2} c(x,y)d\gamma_t(x,y) \\
	&= \mcc_c(\de_{x_0} e^{t\mA},\de_{y_0} e^{t\mB}). 
\end{align*}
Subtracting both side with $\varphi(x_0)+\psi(y_0)=c_2(x_0,y_0)$ yields::
\begin{align*}
	(e^{t\mA}\varphi)(x_0)-\varphi(x_0)+(e^{t\mB}\psi)(y_0)-\psi(y_0)\le \mcc_c(\de_{x_0} e^{t\mA},\de_{y_0} e^{t\mB})-c(x_0,y_0).
\end{align*}
Dividing $t$ and taking limit inferior as $t\searrow 0$, it leads us to 
\begin{align*}
	(\mA\varphi)(x_0)+(\mB\psi)(y_0)\le \om_c^-(x_0,y_0;\mA,\mB). 
\end{align*}
We note the inequality above holds for any $\varphi\in \bar D(\mA),\psi\in \bar D(\mB)$ satisfying $\varphi\oplus \psi\le c$, $\varphi(x_0)+\psi(y_0)=c(x_0,y_0)$. Taking the supremum over all such pairs $(\varphi,\psi)$, this leads to the first inequality. 

\subsection{Proof of Proposition \ref{prop:c2-reg}}

(i) Fix $(x_0,y_0)\in\mbr^{2d}$.  
We construct two barrier functions $a \le c_2 \le b$ that touch $c_2$ at $(x_0,y_0)$.  
The lower barrier $a$ is the supporting plane of $c_2$ at $(x_0,y_0)$, while $b$ is a quadratic function touching $c_2$ from above:
\begin{align*}
	a(x,y)&= \tfrac{1}{2}(x_0-y_0)^\top(x-y)
	= \tfrac{1}{2}(x_0-y_0)^\top x - \tfrac{1}{2}(x_0-y_0)^\top y
	=: \vphi(x)+\psi(y),\\
	b(x,y)&= |x-x_0|^2+|y-y_0|^2+c_2(x_0,y_0)
	+(x-x_0)^\top(x_0-y_0)-(y-y_0)^\top(x_0-y_0) \\
	&=: \Phi(x)+\Psi(y)+c_2(x_0,y_0).
\end{align*}
Clearly $a\le c_2 \le b$, with equality at $(x_0,y_0)$.

Let $\{T_t\}_{t\ge0}$ be a Markovian coupling semigroup of $\mA,\mB$.  
By order preservation,
\begin{align}\label{eq:barrier-ineq}
	T_h a \;\le\; T_h c_2 \;\le\; T_h b, \qquad h\ge0. 
\end{align}
Since both $a$ and $b$ are of direct sum form, the marginality property yields
\begin{align*}
	(T_h a)(x_0,y_0)
	&= (e^{h\mA}\vphi)(x_0)+(e^{h\mB}\psi)(y_0) \\
	&= c_2(x_0,y_0)+\tfrac{h}{2}(m_\mA-m_\mB)^\top(x_0-y_0), \\
	(T_h b)(x_0,y_0)
	&= (e^{h\mA}\Phi)(x_0)+(e^{h\mB}\Psi)(y_0)+c_2(x_0,y_0) \\
	&= c_2(x_0,y_0)+ h^2\big(|m_\mA|^2+|m_\mB|^2\big) 
	+ h\Big[\tr(Q_\mA+Q_\mB)+(m_\mA-m_\mB)^\top(x_0-y_0)\Big].
\end{align*}
Here we used that for quadratic–linear test functions, the semigroup action can be computed explicitly from the mean vector and covariance matrix of the underlying Lévy generator:
\begin{align*}
	(e^{t\mA}\vphi)(x_0)&= \tfrac{1}{2}(x_0+tm_\mA)^\top(x_0-y_0),\\
	(e^{t\mB}\psi)(y_0)&= -\tfrac{1}{2}(y_0+tm_\mB)^\top(x_0-y_0),\\
	(e^{t\mA}\Phi)(x_0)&= t^2|m_\mA|^2 + t\tr(Q_\mA) + t\,m_\mA^\top(x_0-y_0),\\
	(e^{t\mB}\Psi)(y_0)&= t^2|m_\mB|^2 + t\tr(Q_\mB) - t\,m_\mB^\top(x_0-y_0).
\end{align*}

Recalling 
\(
U(t,x,y) = (x-y+t(m_\mA-m_\mB))^\top(m_\mA-m_\mB),
\)
the identities above for $T_ha,T_hb$, combined with \eqref{eq:barrier-ineq}, give
\[
\frac{h}{2}\,U(0,x_0,y_0)
\;\le\; (T_h c_2-c_2)(x_0,y_0)
\;\le\; h^2\big(|m_\mA|^2+|m_\mB|^2\big)
+ h\big[\tr(Q_\mA+Q_\mB)+U(0,x_0,y_0)\big].
\]
Finally, by order preservation again, applying $T_t$ to this inequality yields for all $t,h\ge0$:
\[
\frac{h}{2}\,U(t,x_0,y_0)
\;\le\; (T_{t+h}c_2-T_tc_2)(x_0,y_0)
\;\le\; h^2\big(|m_\mA|^2+|m_\mB|^2\big)
+ h\big[\tr(Q_\mA+Q_\mB)+U(t,x_0,y_0)\big],
\]
which is the desired bound. We note here the identity $T_tU(0,\cdot,\cdot)=U(t,\cdot ,\cdot)$ is used.

(ii) Let $\mJ\in\Ga(\mA,\mB)$. Since $\{e^{t\mJ}\}_{t\ge 0}$ is itself a Markovian coupling semigroup, the estimate from (a) applies to $e^{t\mJ}c_2$, which shows that $t\mapsto e^{t\mJ}c_2$ is locally uniformly continuous.  
To conclude $c_2\in C(\mJ)$, it remains to check that for each $t\ge 0$, one has $e^{t\mJ}c_2\in C(\mbr^d)$.  
Let $\{\kappa_t(x,y)\}$ denote the transition kernel of $\mJ$. By the Feller property, the map $(x,y)\mapsto \kappa_t(x,y)$ is weakly continuous. Consider $Q(x,y)=q(x)+q(y)$ with $q(z)=|z|^2$. By the marginal property of the coupling semigroup,
\[
\inn{\kappa_t(x,y),Q}= e^{t\mJ}Q(x,y)=(e^{t\mA}q)(x)+(e^{t\mB}q)(y).
\]
Since $q\in \bar D(\mA)\cap\bar D(\mB)$ (as $\mA,\mB\in \mcg_2^\La(\mbr^d)$), both terms on the right-hand side are continuous, and hence $(x,y)\mapsto \kappa_t(x,y)$ is continuous in the 2-Wasserstein topology. Consequently, for any continuous function with quadratic growth, such as $c_2$, we obtain
\[
(e^{t\mJ}c_2)(x,y)=\inn{\kappa_t(x,y),c_2},
\]
which is continuous in $(x,y)$. Thus $e^{t\mJ}c_2\in C(\mbr^d)$, and so $c_2\in C(\mJ)$.  

(iii) If $\mA,\mB\in\mcg_2^\La(\mbr^d)$, and $\mJ\in\Ga(\mA,\mB)$ is itself a L\'evy generator, then clearly  $\mJ\in\mcg_2^\La(\mbr^{2d})$, i.e., it admits finite second moments. 
This is because $e^{t\mJ}[|x|^2+|y|^2]=e^{t\mA}[|x|^2]+e^{t\mB}[|y|^2]<\infty$. 
In this case $\bar D(\mJ)$ contains all functions in $C_2^2(\mbr^{2d})$, which includes $c_2$. Hence $c_2\in\bar D(\mJ)$. 

%

%% file: S6_Appen_WeakTop.tex
\renewcommand{\thelemma}{C.\arabic{lemma}}
\renewcommand{\thetheorem}{C.\arabic{theorem}}
\renewcommand{\theproposition}{C.\arabic{proposition}}
\renewcommand{\theremark}{C.\arabic{remark}}

\section{$\La$-weak topology in the space $\La(\mbr^d)$ of L\'evy measures}\label{appen-weaktop}

This appendix reviews basic results on the weak topology for Lévy measures. While related notions appear in standard references such as Sato’s \cite{Sato1999} or Kallenberg's \cite{kallenberg1997fmp} monograph, they are not typically stated in a form tailored to our convergence analysis. The results presented here parallel the classical weak theory for probability measures—such as the Portmanteau theorem and Prokhorov’s compactness criterion—but adapted to the Lévy class $\La(\mathbb{R}^d)$, where singularity at the origin.

Let us briefly recall the notions of \emph{weak convergence}, \emph{vague convergence} and \emph{tightness} for bounded measures in the classical setting. 
Throughout the discussion, l	et $\Pi$ be a locally compact, Hausdorff, second countable space, and denote by $\mathcal{M}(\Pi)$ the space of bounded Radon (Borel) measures on $\Pi$. A sequence $\{\mu_n\}_n \subset \mathcal{M}(\Pi)$ is said to \emph{converge weakly} to $\mu \in \mathcal{M}(\Pi)$ if
\begin{align*}
	\lim_{n\to\infty}\int_{\Pi} \varphi(x)\, d\mu_n(x) = \int_{\Pi} \varphi(x)\, d\mu(x)
\end{align*}
for all bounded continuous functions $\varphi \in C_b(\Pi)$. 
The sequence \emph{converges vaguely} to $\mu$ if the same holds for all compactly supported test functions $\varphi \in C_c(\Pi)$. 
A family $\mathcal{N} \subset \mathcal{M}(\Pi)$ is said to be \emph{tight} if:
\begin{enumerate}
	\item $\sup_{\mu \in \mathcal{N}} \mu(\Pi) < \infty$;
	\item for every $\varepsilon > 0$, there exists a compact set $K \subset \Pi$ such that $\sup_{\mu \in \mathcal{N}} \mu(K^c) < \varepsilon$.
\end{enumerate}
It is well known that:
\begin{itemize}
	\item If $\mu_n \to \mu$ weakly, then $\mu_n \to \mu$ vaguely. 
	\item Conversely, if $\mu_n \to \mu$ vaguely, the family $\{\mu_n\}$ is tight, then $\mu_n \to \mu$ weakly. 
	\item In particular, for probability measures (where $\mu_n(\Pi)=1$), vague convergence together with tightness is equivalent to weak convergence. 
\end{itemize}

In our work, we utilize the following notions of convergence and tightnesss of L\'evy measures in the sense of L\'evy. 

\begin{definition}[$\La$-weak convergence and $\La$-tightness]
	(i) For $n\ge 1$, let $\mu,\mu_n\in \La(\mbr^d)$. We say that $\mu_n$ \emph{converges L\'evy weakly, or $\La$-weakly, to $\mu$}, written $\mu_n \xrightarrow{\La} \mu$, if
	\begin{align}\label{def:weak-conv}
		\lim_{n \to \infty} \int_{\mathbb{R}^d} \varphi(x) \, d\mu_n(x) = \int_{\mathbb{R}^d} \varphi(x) \, d\mu(x)	
	\end{align}
	for all bounded continuous test functions \( \varphi \in C_b(\mathbb{R}^d) \) satisfying \( |\varphi(x)| \le C\min\{1,|x|^2\} \) for some constant \( C \ge 0 \).
	
	(ii) A family \( \mathcal{M} \subset \La(\mathbb{R}^d) \) is said to be \emph{$\La$-tight} if the following conditions holds: (1) it holds
	\begin{align*}
		\sup_{\mu\in \mcl{M}} \int_{\mbr^d} \min\cb{1,|x|^2}d\mu(x)<\infty;
	\end{align*}
	(2) for every \( \varepsilon > 0 \), there exist a compact set $K\subset \mbr^d\setminus \{0\}$ such that 
	\[
	\sup_{\mu \in \mathcal{M}} \int_{K^c} \min\{1, |x|^2\} \, d\mu(x) < \varepsilon.
	\]
\end{definition}

We now relate these classical notions to their counterparts in the context of Lévy measures. 
Given a Lévy measure $\mu \in \Lambda(\mathbb{R}^d)$, define its associated re-scaled (bounded) measure 
$\tilde{\mu} \in \mathcal{M}(\mathbb{R}^d \setminus \{0\})$ by
\begin{align*}
	d\tilde{\mu}(x) := \min\{1, |x|^2\} \, d\mu(x),
\end{align*}
where $\mathbb{R}^d \setminus \{0\}$ is a Polish space.  
Likewisse, given a measurable function $\tilde \varphi:\mathbb{R}^d \setminus \{0\}\to \mathbb{R}$, define its re-scaled version by
\begin{align}\label{def:rescaled}
	\varphi(x) := \min\{1,|x|^2\}\,\tilde{\varphi}(x), 
	\qquad \varphi(0):=0.
\end{align}
Then the identity holds:
\begin{equation}\label{eq:C1}
	\int_{\mathbb{R}^d} \varphi \, d\mu 
	= \int_{\mathbb{R}^d\setminus\{0\}} \tilde{\varphi} \, d\tilde{\mu}
\end{equation}
Moreover, $\tilde\varphi$ is continuous (resp.\ lower semicontinuous, upper semicontinuous) on $\mathbb{R}^d\setminus\{0\}$ 
iff $\varphi$ is continuous (resp.\ lower semicontinuous, upper semicontinuous) on $\mathbb{R}^d$. 
Similarly, $\tilde \varphi$ is bounded above (resp.\ below) iff $\varphi$ satisfies 
$\varphi(x)\le C\min\{1,|x|^2\}$ (resp.\ $\varphi(x)\ge -C\min\{1,|x|^2\}$) for some $C\ge 0$.  

This rescaling relation allows us to transfer classical weak convergence results for $\tilde\mu$ 
to the setting of Lévy measures $\mu$, as will be seen in the following characterization.

\begin{proposition}\label{lem:char}
	\leavevmode
	\begin{enumerate}[label=(\roman*)]
		\item Let $\mu,\mu_n \in \Lambda(\mathbb{R}^d)$ for $n \geq 1$. The following are equivalent:
		\begin{itemize}
			\item $\mu_n \xrightarrow{\Lambda} \mu$ ($\Lambda$-weak convergence);
			\item $\tilde{\mu}_n \to \tilde{\mu}$ weakly on $\mathbb{R}^d \setminus \{0\}$.
		\end{itemize}
		
		\item Let $\mathcal{N} \subset \Lambda(\mathbb{R}^d)$ be a family of Lévy measures. The following are equivalent:
		\begin{itemize}
			\item $\mathcal{N}$ is $\Lambda$-tight;
			\item $\tilde{\mathcal{N}} := \{\tilde{\mu} : \mu \in \mathcal{N}\}$ is tight on $\mathbb{R}^d \setminus \{0\}$.
		\end{itemize}
	\end{enumerate}
\end{proposition}

\begin{proof}
	(i) By \eqref{def:rescaled} and \eqref{eq:C1}, testing against $\varphi$ in the definition of $\Lambda$-weak convergence is equivalent to testing against $\tilde\varphi \in C_b(\mathbb{R}^d\setminus\{0\})$. Hence $\mu_n \xrightarrow{\Lambda} \mu$ iff $\tilde\mu_n \to \tilde\mu$ weakly.  
	
	(ii) Likewise, by \eqref{eq:C1}, $\Lambda$-tightness of $\{\mu_n\}$ is equivalent to tightness of $\{\tilde\mu_n\}$ on $\mathbb{R}^d\setminus\{0\}$.  
\end{proof}

\subsection{Portmanteau Lemma for $\La$-weak convergence}
Our first result is a Portmanteau-type theorem, giving an equivalent characterization of $\La$-weak convergence in terms of lower and upper semicontinuous test functions.

\begin{lemma}[Portmanteau-type lemma for $\La$-weak convergence]\label{lem:port}
	Let \( \{\mu_n\}_n \subset \La(\mathbb{R}^d) \) be a sequence of Lévy measures, and let \( \mu \in \La(\mathbb{R}^d) \). The following are equivalent:
	
	\begin{itemize}
		\item[(i)] \( \mu_n \xrightarrow{\La} \mu \).
		
		\item[(ii)] For every lower semicontinuous function \( \varphi: \mathbb{R}^d \to \mathbb{R} \) satisfying \( \varphi(x) \ge -C \min\{1, |x|^2\} \) for some \( C \ge 0 \), it holds that
		\[
		\int_{\mathbb{R}^d} \varphi(x)\, d\mu(x) \le \liminf_{n \to \infty} \int_{\mathbb{R}^d} \varphi(x)\, d\mu_n(x).
		\]
		
		\item[(iii)] For every upper semicontinuous function \( \varphi: \mathbb{R}^d \to \mathbb{R} \) satisfying \( \varphi(x) \le C \min\{1, |x|^2\} \) for some \( C \ge 0 \), it holds that
		\[
		\int_{\mathbb{R}^d} \varphi(x)\, d\mu(x) \ge \limsup_{n \to \infty} \int_{\mathbb{R}^d} \varphi(x)\, d\mu_n(x).
		\]
	\end{itemize}
\end{lemma}

\begin{proof}
	By Lemma~\ref{lem:char}, $\mu_n \xrightarrow{\La} \mu$ iff the rescaled measures 
	$\tilde \mu_n \to \tilde \mu$ weakly on $\mathbb{R}^d \setminus \{0\}$, with the correspondence \eqref{eq:C1},
	where $\tilde \varphi$ is the rescaling of $\varphi$ defined in~\eqref{def:rescaled}.
	The conclusion then follows directly from the classical Portmanteau lemma applied to $\tilde \mu_n \to \tilde \mu$.
\end{proof}

\begin{proposition}\label{prop:appenC4}
	We have $\mu_n \xrightarrow{\La} \mu$ if and only if \eqref{def:weak-conv} holds for all bounded Lipschitz test functions $\vphi\in C_b(\mbr^d)$ satisfying 
	\[
	|\vphi(x)| \le C \min\{1,|x|^2\}, 
	\qquad 
	|\nabla \vphi(x)| \le C|x| \quad \text{(Lebesgue a.e.)}
	\]
	for some constant $C \ge 0$. 
\end{proposition}

\begin{proof}
	The implication ``$\to$'' is immediate from the definition of $\La$-weak convergence.  
	Let us prove the converse.  
	
	By assumption, \eqref{def:weak-conv} holds for every bounded Lipschitz function $\vphi$ satisfying the above growth bounds. In particular, it applies to functions $\vphi$ that vanish inside some ball $B_r(0)$, since in that case we can find constants $C,C'\ge 0$ with 
	\[
	|\nabla\vphi(x)| \le C'\chi_{B_r(0)^c}(x) \le C|x|.
	\]
	For each $r>0$, the family $\{\chi_{B_r(0)^c}\mu_n\}_n$ is tight (as the sequence of truncated measures converges weakly). By density of Lipschitz functions in $C_b(\mbr^d)$, it follows that \eqref{def:weak-conv} holds for every $\vphi \in C_b(\mbr^d)$ that vanishes on some $B_r(0)$.  
	
	Next, let us extend \eqref{def:weak-conv} to all $\vphi \in C_b(\mbr^d)$ such that $\vphi(x)\le C\min\{1,|x|^2\}$ for some $C \ge 0$.  
	Fix a cutoff $\chi \in C_b(\mbr^d)$ that is radially nondecreasing, with $\chi(x)=0$ for $|x|\le 1$ and $\chi(x)=1$ for $|x|\ge 2$.  
	For $\ep>0$, set
	\(
	\chi_\ep(x) := \chi(\tfrac{x}{\ep}), 
	\vphi_\ep := \vphi \chi_\ep.
	\)
	Then
	\[
	|\vphi_\ep(x)| \le C\min\{1,|x|^2\}, 
	\qquad 
	|\vphi(x)-\vphi_\ep(x)| 
	\le C\chi_{B_{2\ep}(0)}(x)\min\{1,|x|^2\} 
	\le C\min\{4\ep^2,|x|^2\}.
	\]
	Hence we have 
	\[
	\Bigl|\int_{\mbr^d} \vphi\, d(\mu_n-\mu)\Bigr|
	\;\le\;
	\Bigl|\int_{\mbr^d} \vphi_\ep \,d(\mu_n-\mu)\Bigr|
	+ \int_{\mbr^d} |\vphi-\vphi_\ep|\, d(\mu_n+\mu).
	\]
	
	As $n\to\infty$, the first term vanishes by the previous paragraph (since $\vphi_\ep$ vanishes on $B_\ep(0)$). Therefore
	\begin{align*}
		\limsup_{n\to\infty}\Bigl|\int_{\mbr^d} \vphi\, d(\mu_n-\mu)\Bigr|
		&\le \limsup_{n\to\infty} \int_{\mbr^d} |\vphi-\vphi_\ep|\, d(\mu_n+\mu) \\
		&\le \limsup_{n\to\infty} C \int_{\mbr^d} \min\{|x|^2, 4\ep^2\}\, d(\mu_n+\mu) \\
		&= 2C \int_{\mbr^d} \min\{|x|^2, 4\ep^2\}\, d\mu.
	\end{align*}
	Here we used that $\psi(x):= C\min\{4\ep^2,|x|^2\}$ satisfies $|\nabla \psi(x)| \le C|x|$, so \eqref{def:weak-conv} also applies to $\psi$.  
	Finally, since $\ep>0$ is arbitrary, sending $\ep \searrow 0$ yields that \eqref{def:weak-conv} holds for all admissible $\vphi$.  
	Thus $\mu_n \xrightarrow{\La} \mu$.
\end{proof}

\begin{proposition}\label{prop:la-conv-tight}
	If \( \mu_n \xrightarrow{\La} \mu \) for some \( \mu \in \La(\mathbb{R}^d) \), then the sequence \( \{ \mu_n \}_n \) is \( \La \)-tight.
\end{proposition}

\begin{proof}
	By the characterization given by Lemma \ref{lem:char}, $\mu_n\xrightarrow[]{\La}\mu$ if and only if $\tilde \mu_n\to \tilde \mu$ weakly on $\mbr^d\setminus\{0\}$. From the classical result, $\{\tilde \mu_n\}$ is tight on $\mbr^d\setminus\{0\}$ in the usual sense. Applying Lemma \ref{lem:char} again $\{\mu_n\}_n$ is $\La$-tight.
%
%
\end{proof}

Next, we characterize $\Lambda$-convergence for Lévy measures with bounded second moments. This parallels the case of probability measures with finite second moments (cf.~\cite[Definition 6.8]{MR2459454}).  

\begin{lemma}[Characterization of convergence in $\Lambda_2(\mathbb{R}^d$)] 
	\label{lem:C5}
	Let $\mu_n,\mu \in \Lambda_2(\mathbb{R}^d)$. Suppose \( \mu_n \xrightarrow{\Lambda} \mu \) and
	\begin{align}\label{conv:sec-moment}
		\lim_{n\to\infty}\int_{\mathbb{R}^d} |x|^2\, d\mu_n(x) 
		= \int_{\mathbb{R}^d} |x|^2\, d\mu(x).
	\end{align}
	Then for every \( \varepsilon > 0 \), there exists \( R \ge 1 \) such that
	\[
	\sup_n \int_{|x| \ge R} |x|^2\, d\mu_n(x) \le \varepsilon.
	\]
	Moreover, for any continuous $\varphi$ with quadratic growth, $|\varphi(x)| \le C|x|^2$ for some $C \ge 0$, one has
	\begin{align*}
		\lim_{n\to\infty} \int_{\mathbb{R}^d} \varphi(x)\, d\mu_n(x) 
		= \int_{\mathbb{R}^d} \varphi(x)\, d\mu(x).
	\end{align*}
\end{lemma}

%

\begin{proof}	
	Fix \( \varepsilon > 0 \). Since \( |x|^2 \in L^1(\mu) \), there exists \( R > 0 \) such that
	\[
	\int_{|x| \ge R} |x|^2\, d\mu(x) < \frac{\varepsilon}{2}.
	\]
	Consider the truncated function \( |x|^2 \chi_{\{|x| < R\}} \), which is lower semicontinuous and bounded below. By the Portmanteau theorem (Lemma~\ref{lem:port}),
	\[
	\liminf_{n\to\infty} \int_{|x|<R} |x|^2\, d\mu_n(x) \;\ge\; \int_{|x|<R} |x|^2\, d\mu(x).
	\]
	Using the convergence of the second moment \eqref{conv:sec-moment}, we obtain
	\[
	\limsup_{n\to\infty} \int_{|x|\ge R} |x|^2\, d\mu_n(x) 
	\;\le\; \int_{|x|\ge R} |x|^2\, d\mu(x) < \frac{\varepsilon}{2}.
	\]
	Hence, there exists \( N \) such that for all \( n \ge N \),
	\[
	\int_{|x|\ge R} |x|^2\, d\mu_n(x) < \varepsilon.
	\]
	For the finitely many indices \( n < N \), finiteness of the second moment allows enlarging \( R \) so that the same bound holds. Thus the first claim follows.
	
	For the second claim, let \( \varphi \in C(\mathbb{R}^d) \) with \( |\varphi(x)| \le C |x|^2 \). For \( R > 0 \), define the cutoff 
	\[
	\varphi_R(x) = \max\{-R, \min\{\varphi(x), R\}\} , \qquad \delta_R = \varphi - \varphi_R.
\]
Then \( \varphi_R \in C_b(\mathbb{R}^d) \) and \( |\delta_R(x)| \le C |x|^2 \chi_{\{|x|\ge R\}} \). Hence,
\[
\Big| \int \varphi \, d\mu_n - \int \varphi \, d\mu \Big|
\;\le\; \Big| \int \varphi_R \, d\mu_n - \int \varphi_R \, d\mu \Big| 
+ \int |\delta_R(x)| \, d(\mu_n+\mu).
\]
As \( n \to \infty \), the first term vanishes since \( \mu_n \xrightarrow{\La} \mu \). For the second term,
\[
\limsup_{n\to\infty} \Big| \int \varphi \, d\mu_n - \int \varphi \, d\mu \Big|
\;\le\; 2C \sup_n \int_{|x|\ge R} |x|^2\, d\mu_n(x).
\]
By the first part, the right-hand side can be made arbitrarily small by choosing \( R \) large. This proves the result.
\end{proof}

\subsection{Prokhorov-type Compactness Theorem}
We conclude with a compactness criterion for the space $\La(\mbr^d)$ of L\'evy measures, in analogy with Prokhorov's theorem.

\begin{theorem}[Prokhorov-type compactness]\label{thm:prokhorov}
	Let \( \{\mu_n\}_n \subset \La(\mathbb{R}^d) \) be a \( \La \)-tight sequence of Lévy measures. Then there exist a subsequence (still denoted \( \mu_n \)) and a Lévy measure \( \mu \in \La(\mathbb{R}^d) \) such that $\mu_n \xrightarrow{\La} \mu.$
\end{theorem}

\begin{proof}
	Let $\{\tilde \mu_n\}_n \subset \mcl{M}(\mbr^d \setminus \{0\})$ be the rescaled measures as in \eqref{def:rescaled}. 
	By Proposition~\ref{lem:char}, the $\La$-tightness of $\{\mu_n\}_n$ implies that $\{\tilde \mu_n\}_n$ is tight in $\mcl{M}(\mbr^d \setminus \{0\})$. 
	By Prokhorov’s theorem, there exists a subsequence (still denoted $\tilde \mu_n$) converging weakly to some $\tilde \mu \in \mcl{M}(\mbr^d \setminus \{0\})$.	
	Define $\mu \in \La(\mbr^d)$ by 
	\(
	d\mu(x) = (\min\{1,|x|^2\})^{-1}\, d\tilde \mu(x).
	\)
	Then Proposition~\ref{lem:char} yields $\mu_n \xrightarrow{\La} \mu$ along the subsequence. 
\end{proof}